\documentclass[10pt]{amsart}
\usepackage{mathrsfs}
\usepackage{amssymb}
\usepackage{amsmath, bm}

\usepackage{titletoc}
\usepackage{amscd}
\usepackage{amsmath, amssymb}
\usepackage{amsthm}
\usepackage{amsfonts}
\usepackage[colorlinks,linkcolor=blue,citecolor=blue, pdfstartview=FitH]{hyperref}

\usepackage[all,cmtip]{xy}

\usepackage{stmaryrd}

\usepackage{upgreek}

\usepackage{tikz-cd} 

\usepackage{extarrows} 

\numberwithin{equation}{section}

\theoremstyle{plain}
\newtheorem{thm}{Theorem}[section]

\newtheorem{lem}[thm]{Lemma}
\newtheorem{corollary}[thm]{Corollary}
\newtheorem{prop}[thm]{Proposition}

\newtheorem{thmx}{Theorem}

\theoremstyle{definition}

\newtheorem{rmk}[thm]{Remark}

\newtheorem{definition}[thm]{Definition}

\newtheorem{defn-thm}[thm]{Definition-Theorem}
\newtheorem{defn-pro}[thm]{Definition-Proposition}

\theoremstyle{remark}

\usepackage{fancyhdr}

\newcommand{\A}{{\mathsf A}}%
\newcommand{\sA}{{\mathcal A}}

\newcommand{\sB}{{\mathcal B}}
\newcommand{\rB}{{\mathrm B}}
\newcommand{\C}{{\mathbb C}}
\newcommand{\sC}{{\mathcal C}}

\newcommand{\sD}{{\mathcal D}}
\newcommand{\e}{{\mathrm e}}

\newcommand{\sE}{{\mathcal E}}
\newcommand{\F}{{\mathbb F}}

\newcommand{\sF}{{\mathcal F}}
\newcommand{\fF}{{\mathfrak{F}}}

\newcommand{\fG}{{\mathfrak{G}}}
\renewcommand{\H}{{\mathbb H}}

\newcommand{\I}{{\mathrm I}}
\newcommand{\sI}{{\mathcal I}}

\newcommand{\sJ}{{\mathcal J}}
\newcommand{\K}{{\mathbb K}}

\newcommand{\sL}{{\mathcal L}}
\newcommand{\m}{{\mathfrak m}}
\newcommand{\M}{{\mathbb M}}
\newcommand{\sM}{{\mathcal M}}
\newcommand{\fM}{{\mathfrak M}}
\newcommand{\N}{{\mathbb N}}
\newcommand{\sN}{{\mathcal N}}

\newcommand{\sO}{{\mathcal O}}
\renewcommand{\P}{{\mathbb P}}

\newcommand{\sP}{{\mathcal P}}
 
\newcommand{\rQ}{{\mathrm Q}}
\newcommand{\sQ}{{\mathcal Q}}

\newcommand{\R}{{\mathbb R}}

\newcommand{\T}{{\mathbb T}}
\newcommand{\rT}{{\mathsf T}}
\newcommand{\sT}{{\mathcal T}}
\newcommand{\U}{{\mathbf U}}
\renewcommand{\u}{{\mathbf u}}

\newcommand{\fU}{{\mathfrak U}}
\newcommand{\V}{{\mathbb V}}
\newcommand{\sV}{{\mathcal V}}

\newcommand{\bW}{{\mathrm W}}

\newcommand{\Z}{{\mathbb Z}}
\newcommand{\sZ}{{\mathcal Z}} 
\newcommand{\p}{\mathfrak{p}}

\renewcommand{\k}{\C} 

\newcommand{\id}{\mathrm{id}}

\newcommand{\pr}{\mathrm{pr}}

\newcommand{\xs}{{\ \xrightarrow{\sim} \ }} 
\newcommand{\hgt}{\mathrm{ht}}

\newcommand{\rad}{\mathrm{rad}} 
\newcommand{\Hom}{\mathrm{Hom}}
\newcommand{\End}{\mathrm{End}}

\newcommand{\Ext}{\mathrm{Ext}}
 
\newcommand{\gr}{{\mathrm{gr}}}

\newcommand{\op}{{\mathrm{op}}}

\newcommand{\Frac}{{\mathrm{Frac}}}

\newcommand{\Mod}{\text{-}\mathrm{Mod}}
\renewcommand{\mod}{\text{-}\mathrm{mod}}
\newcommand{\gmod}{\text{-}\mathrm{gmod}}

\newcommand{\Rep}{\mathrm{Rep}}
\newcommand{\rep}{\mathrm{rep}}

\newcommand{\Coh}{\mathrm{Coh}}
\newcommand{\QCoh}{\mathrm{QCoh}}

\renewcommand{\for}{\mathrm{for}}

\newcommand{\ind}{\mathrm{ind}}
\newcommand{\coind}{\mathrm{coind}}

\newcommand{\rI}{\mathrm{I}} 

\newcommand{\Spec}{\mathrm{Spec}} 
\newcommand{\Proj}{\mathrm{Proj}}

\newcommand{\Lie}{\mathrm{Lie}} 
\renewcommand{\sc}{\mathrm{sc}} 

\newcommand{\g}{{\mathfrak{g}}}
\renewcommand{\b}{{\mathfrak{b}}}
\renewcommand{\t}{{\mathfrak{t}}}
\newcommand{\n}{{\mathfrak{n}}}

\newcommand{\Fr}{\mathrm{Fr}} 
\newcommand{\HC}{\mathrm{HC}} 
 
\newcommand{\St}{\mathrm{St}} 
\newcommand{\af}{\mathrm{af}} 
 
\newcommand{\ex}{\mathrm{ex}}

\newcommand{\sEnd}{\mathcal{E}nd}
\newcommand{\sHom}{\mathcal{H}om}

\newcommand{\Fl}{\mathcal{F}{l}} 
\newcommand{\hb}{\mathrm{hb}}

\newcommand{\Db}{D^\mathrm{b}}

\newcommand{\FN}{\mathrm{FN}} 

\newcommand{\semi}{{\frac{\infty}{2}}}

\begin{document}
\title{Category $\mathcal{O}$ for hybrid quantum groups and non-commutative Springer resolutions} 

\makeatletter
\let\MakeUppercase\relax
\makeatother

\author{Quan Situ} 
\address{Yau Mathematical Sciences Center\\
Tsinghua University\\
Beijing 100084, P.~R.~China}
\email{stq19@tsinghua.org.cn, quan.situ@uca.fr}
\date{}
\begin{abstract} 
The hybrid quantum group was firstly introduced by Gaitsgory, whose category $\mathcal{O}$ can be viewed as a quantum analogue of BGG category $\mathcal{O}$. 
We give a coherent model for its principal block at roots of unity, using the non-commutative Springer resolution defined by Bezrukavnikov--Mirkovi\'{c}. 
In particular, the principal block is derived equivalent to the affine Hecke category. 
As an application, we endow the principal block with a canonical grading, and show that the graded multiplicity of simple module in Verma module is given by the generic Kazhdan--Lusztig polynomial. 
\end{abstract} 

\maketitle
\setcounter{tocdepth}{1} \tableofcontents 

\section{Introduction}\label{sect 1} 
\subsection{Hybrid quantum group and its category $\sO$} 
Let $G$ be a complex connected and simply-connected semisimple algebraic group, with a Borel subgroup $B$ and a Cartan subgroup $T\subset B$. 
Let $U_q$ be \textit{Lusztig's quantum group} \cite{Lus90} and let $\fU_q$ be \textit{De Concini--Kac's quantum group} \cite{DeCK90}, associated to $G$. 
The \textit{hybrid (or mixed) quantum group} $U^\hb_q$ is an algebra admitting a triangular decomposition
$$U^\hb_q=\fU^-_q\otimes \fU^0_q\otimes U^+_q,$$ 
which was firstly introduced by Gaitsgory \cite{Gai18} with the perspective of generalizing the Kazhdan--Lusztig equivalence. 
It also appears naturally when one considers the quantum analogue of the BGG category $\sO$ \cite{BGG76}. 
Recall that the BGG category $\sO$ (with integral weights) for the Lie algebra $\g=\Lie(G)$ can be interpreted as the category 
$$\g\mod^B$$ 
of $B$-equivariant finitely generated $\g$-modules on which the actions of $\b=\Lie(B)$ by  restriction of $\g$-action and by differential of $B$-action coincide. 
Therefore, one can define the quantum category $\sO$ to be the category 
$$\sO_q=\fU_q\mod^{U^{\geq}_q}$$ 
of finitely generated $\fU_q$-modules equipped with an equivariant and integrable $U^{\geq}_q$-module structure that is compatible through $\fU^{\geq}_q\rightarrow U^{\geq}_q$. 
It is exactly the category $\sO$ (with integral weights) associated to $U^\hb_q$ with its triangular decomposition. 

\subsection{Main result} 
Let $\sO^0_q$ be the principal block for $\sO_q$. 
From now on, we specialize $q$ to a root of unity $\zeta$ whose order satisfies some reasonable conditions specified in \textsection\ref{subsect 2.1} and \textsection\ref{subsect 3.2.2} (see also Remark \ref{rmk 5.14}). 
Our main goal is to give a coherent model for $\sO^0_\zeta$, using the non-commutative Springer (Grothendieck) resolution defined in \cite{BM13}. 

Let $\A$ be the ($G$-equivariant) $\C[\g\times_{\t/W}\t]$-algebra defined in \textit{loc. cit.} as the non-commutative counterpart of the Grothendieck resolution 
$\widetilde{\g}=G\times^B \b \rightarrow \g\times_{\t/W} \t$ (where $\t=\Lie(T)$ and $W$ is the Weyl group). 
Let $\pi_{_{\widetilde{\sN}}}:\widetilde{\sN}=G\times^B \n \rightarrow \g$ be the Springer resolution (where $\n$ is the nilpotent radical of $\b$). 
Our main result is 
\begin{thmx}[=Corollaries \ref{cor 5.14}\&\ref{cor 6.8}]\label{thm A} 
There is an equivalence of abelian categories 
\begin{equation}\label{equ 1}
\sO^0_\zeta \xs \Coh^G(\pi_{_{\widetilde{\sN}}}^*\A) 
\end{equation}
intertwining the reflection functors on both sides. 
\end{thmx} 

Note that a similar result for modular analogue of BGG category $\sO$ was obtained by Losev \cite{Los23} recently. 
The category $\Coh^G(\pi_{_{\widetilde{\sN}}}^*\A)$ firstly appeared in \cite{BLin}, which was shown to be derived equivalent to a version of the affine Hecke category and coincide with the heart of the ``new t-structure" defined by Frenkel--Gaitsgory \cite{FG09}. 
It would be interesting to explore the relationship between (\ref{equ 1}) and a conjectural equivalence of Gaitsgory \cite[Conj. 9.2.2]{Gai18} recently proved by Chen--Fu \cite{CF21}. 

After this paper is written, Ivan Losev \cite{Los23b} informed us that he has a different approach to prove the equivalence (\ref{equ 1}) independently. 

\ 

Let us explain the ingredients of the proof of Theorem \ref{thm A}. 

\subsection{Equivalence of Steinberg block and functor $\V$}\label{subsect 1.3} 
A key ingredient is an equivalence of the Steinberg block $\sO^{-\rho}_\zeta$ (the block containing the Steinberg representation) established in our previous work \cite{Situ2}. 
\begin{thm}[\cite{Situ2}]
There is an equivalence of abelian categories 
\begin{equation}\label{equ 1.3} 
\sO^{-\rho}_{\zeta} \xs \Coh^G(\widetilde{\sN}). 
\end{equation}
\end{thm} 
\noindent 
Based on the equivalence above, we establish functors $\V_r$ and $\V_c$ (here ``$r$" and ``$c$" represent ``representation" and ``coherent", respectively) on both sides of (\ref{equ 1}) that are fully-faithful on the objects admitting Verma flags (in fact our constructions are in the deformed setting, see \textsection\ref{subsect 1.5} below). 

The functors $\V$ are constructed as follows. 
Let $C=\C[0\times_{\t/W} \t]$ be the coinvariant ring. 
There is an algebra homomorphism (see \cite{BBASV}) 
$$C\rightarrow Z(\sO^0_\zeta).$$ 
We consider the category $\sO^{-\rho}_{\zeta,C}$ of $U^\hb_\zeta\otimes C$-modules that are contained in $\sO^{-\rho}_\zeta$ when viewed as $U^\hb_\zeta$-modules. 
We define a functor 
$$\V_r:\ \sO^0_\zeta\rightarrow \sO^{-\rho}_{\zeta,C}$$ 
by the translation functor from $\sO^0_\zeta$ to $\sO^{-\rho}_\zeta$ and remembering at the same time the action of $C$ on $\sO^0_\zeta$. 
If the parameter $q$ is generic, then the Steinberg block $\sO^{-\rho}_q$ is trivial, so the functor $\V_r$ is given by 
$$\V_r:\ \sO^0_q\rightarrow \sO^{-\rho}_{q,C}\simeq C\mod.$$ 
In this case, it coincides with the Soergel's functor $\V$ \cite{Soe90} for the classical BGG category $\sO$. 
When $q=\zeta$, the category $\sO^{-\rho}_\zeta$ is non-trivial. 
Hence we may view $\V_r$ as a ``relative analogue" of Soergel's functor. 

On the coherent side, we define a functor 
$$\V_c:\ \Coh^G( \pi_{_{\widetilde{\sN}}}^*\A) \rightarrow \Coh^G(\widetilde{\sN}\times_{\t/W} \t)$$
by base change along $\pi_{_{\widetilde{\sN}}}:\widetilde{\sN}\rightarrow \g$ of the non-commutative resolution 
$\A\mod\rightarrow \Coh(\g\times_{\t/W} \t)$. 
To relate the images of two functors $\V_r$ and $\V_c$, we use a $C$-linear extension of the equivalence (\ref{equ 1.3}), i.e. 
$$\sO^{-\rho}_{\zeta,C} \xs \Coh^G(\widetilde{\sN}\times_{\t/W} \t).$$ 

Our next step is to compare the images of the projective objects under the functors $\V_r$ and $\V_c$. 
To that end, we firstly establish a ``de-equvariantization" of Theorem \ref{thm A}. 

\subsection{A de-equvariantization}\label{subsect 1.4} 
More precisely, let $\fU^b_\zeta$ be the image of $\fU_\zeta\rightarrow U^\hb_\zeta$, and consider the category $\fU^b_\zeta\mod^{T}$ of $X^*(T)$-graded representations of $\fU^b_\zeta$. 
Let $\fU^b_\zeta\mod^{T,0}$ be its principal block. 
Note that the restriction of $\pi_{_{\widetilde{\sN}}}^*\A$ on the fiber $\n\hookrightarrow \widetilde{\sN}$ is $\A\otimes_{\C[\g]} \C[\n]$. 
Hence we can identify 
$$\Coh^G(\pi_{_{\widetilde{\sN}}}^*\A)=\Coh^B(\A\otimes_{\C[\g]} \C[\n])$$ 
by restriction to $\n$. 
We have the following $T$-equivariant version of Theorem \ref{thm A}. 
\begin{prop}[=Proposition \ref{prop 4.1}] \label{prop 1.2}
There is an equivalence of abelian categories 
\begin{equation}\label{equ 1.2} 
\fU^b_\zeta\mod^{T,0} \xs \Coh^T(\A\otimes_{\C[\g]} \C[\n]), 
\end{equation} 
intertwining the reflection functors on both sides. 
\end{prop} 
\noindent 
We observe that a generating family of projective objects of the categories in (\ref{equ 1}) can be obtained from the inductions of the projective objects of categories in (\ref{equ 1.2}). 
Therefore, we construct similar functors $\V$ for the categories in (\ref{equ 1.2}), and show their compatibility with the induction functors. 
So that we can identify the images of projective objects under the functors $\V_r$ and $\V_c$. 

The proof of Proposition \ref{prop 1.2} is based on results of Tanisaki \cite{Tan21}, where the localization theorem for $\fU_\zeta$ was established and it further related $\fU_\zeta$-modules to $\A$-modules by some splitting bundles for the sheaf of differential operators on the quantized flag variety. 
Note that it is not clear from their constructions whether the splitting bundles are $U_\zeta$-equivariant or not. 
So \textit{\`{a} priori}, one can not derive (\ref{equ 1}) from (\ref{equ 1.2}), by adding the $U^{\geq}_\zeta$-equivariant and $B$-equivariant conditions on both sides. 
However, after proving our main result, we can show that (\ref{equ 1.2}) does respect the $B$-actions on both categories as a consequence (see Corollary \ref{cor 6.6}). 

\ 

Finally, we have to show the fully-faithfulness for $\V_r$ and $\V_c$ on the objects admitting Verma flags, which is true only in the deformed settings (see Remark \ref{rmk 4.12}). 
Hence we need deformations of category $\sO$, which we mainly work with in this paper. 

\subsection{Deformations}\label{subsect 1.5} 
For the classical category $\sO$, the deformation technique played an important role on studying basic structures, e.g. \cite{Soe90}, \cite{Strop09} and \cite{Fie03}. 
Deformation of $\sO_\zeta$ was studied in \cite{Situ1}. 
Let $S=\C[\t]_{\hat{0}}$ be the completion of $\C[\t]$ at $0\in \t$. 
The deformation category for $\sO_\zeta$ is an $S$-linear category $\sO_{\zeta,S}$ whose specialization to the residue field is $\sO_\zeta$. 
Let $\A_{(\t/W)_{\hat{0}}}$ be the completion of $\A$ at $0\in \t/W$. 
We have a deformed version of Theorem \ref{thm A} 
\begin{thmx}[=Theorem \ref{thm 5.13}]\label{thm B} 
There is an equivalence of $S$-linear abelian categories 
\begin{equation}\label{equ 1.4} 
\sO^0_{\zeta,S}\xs \Coh^G(\pi^*\A_{(\t/W)_{\hat{0}}}),
\end{equation}
where $\sO^0_{\zeta,S}$ is the principal block of $\sO_{\zeta,S}$ and $\pi: \widetilde{\g}\rightarrow \g$ is the natural projection. 
\end{thmx} 
\noindent 
We have similar constructions as in \textsection\ref{subsect 1.3} and \textsection\ref{subsect 1.4} in the deformed setting, which together with the fact that $\V_r$ and $\V_c$ are fully-faithful on objects admitting Verma flags (see Proposition \ref{prop 5.8}) imply Theorem \ref{thm B}. 
Then the equivalence (\ref{equ 1}) follows from (\ref{equ 1.4}) by $S$-linearity. 

\subsection{Kazhdan--Lusztig polynomial} 
Consider the $\C^\times$-action on $\g$ by $t.x=t^{-2}x$, for any $t\in \C^\times$, $x\in \g$. 
The $\C[\g]$-algebra $\A$ can be equipped with a $\C^\times$-equivariant structure. 
By the equivalence (\ref{equ 1}), we can define a grading on $\sO^0_\zeta$ via 
$$\sO^{0,\gr}_\zeta:=\Coh^{G\times \C^\times}(\pi_{_{\widetilde{\sN}}}^*\A)\ \rightarrow\ \sO^{0}_\zeta.$$ 
We compute the Kazhdan--Lusztig polynomials in $\sO^{0,\gr}_\zeta$: 

\begin{thmx}[=Theorem \ref{thm 7.3}]
The Verma modules and simple modules in $\sO^{0}_\zeta$ admit liftings in $\sO^{0,\gr}_\zeta$. 
Moreover, the graded multiplicities of simple module in Verma module are given by the generic Kazhdan--Lusztig polynomials. 
\end{thmx}

\subsection{Organization of the paper} 
In Section~\ref{sect 2} we introduce notations for quantum groups and recall some basic properties of the category $\sO$ for $U^\hb_\zeta$ from \cite{Situ1}. 
In Section~\ref{sect 3} we give some reminders on the non-commutative Springer resolutions following \cite{BM13} and the localization theorem for $\fU_\zeta$ following \cite{Tan21}. 

In Section~\ref{sect 4} we prove the de-equivariantization of main result (i.e. Proposition \ref{prop 1.2}) in \textsection\ref{subsect 4.1new} and its analogue for Steinberg block in \textsection\ref{subsect 4.2}. 
In \textsection\ref{subsect 4.3new} we study the induction functors from the categories in (\ref{equ 1.2}) to the categories in (\ref{equ 1}). 
We construct the functors $\V_r^b$, $\V^b_c$ in \textsection\ref{subsect 4.2.2}, whose fully-faithfulness will be proved in \textsection\ref{subsect 4.4}. 
In \textsection\ref{subsect 4.3} we investigate the image of Verma modules under the equivalence (\ref{equ 1.2}). 

In Section~\ref{sect 5}, we define the functors $\V_r$, $ \V_c$, and show their compatibility with reflection functors in \textsection\ref{subsect 5.1}. 
The subsection \textsection\ref{subsect 5.2} is slightly technical: we consider certain truncated subcategories of $\sO^0_\zeta$ and of the coherent side, which admit enough projective objects and are compatible with functors $\V_r$, $ \V_c$. 
In \textsection\ref{subsect 5.3} we show the fully-faithfulness of $\V_r$, $ \V_c$ and then prove the main theorem. 

In Section~\ref{sect 6}, we firstly recall the formalism of a rational group action on a category in \textsection\ref{subsect 6.1}, then equip a $B$-action on LHS of (\ref{equ 1.2}) using coinduction functor in \textsection\ref{subsect 6.2}. 
In \textsection\ref{subsect 6.3}, we apply our main theorem to show that equivalence (\ref{equ 1}) is the $B$-equivariantization of (\ref{equ 1.2}), and obtain the compatibility of (\ref{equ 1}) with reflection functors as an application. 

In Section~\ref{sect 7}, we firstly recall the periodic and generic polynomials in \textsection\ref{subsect 7.1}. 
Then we discuss the liftings of the simple, Verma and projective objects in the graded principal block and compute the related graded multiplicities in \textsection\ref{subsect 7.2}. 

\subsection{Conventions} 
Let $A$ be an algebra over a commutative Noetherian ring $R$. 
For any closed point $z\in \Spec R$, we denote by $A_{\hat{z}}$ the completion of $A$ at $z$. 
Suppose $f: X\rightarrow \Spec R$ is a morphism of schemes. 
By \textit{formal neighborhood} of $Y=f^{-1}(z)$ in $X$, we mean a \textit{scheme} 
$$\FN_{X}(Y):=X_{\hat{z}}:= X\times_{\Spec R} \Spec(R_{\hat{z}}).$$ 

Let $\sA$ be a quasi-coherent sheaf of algebras on a scheme $X$. 
We will denote by $\sA\mod$ (resp. $\sA\Mod$) the category of coherent (resp. quasi-coherent) $\sA$-modules. 
For a closed subscheme $Y$ in $X$, we denote by $\sA\mod_Y$ (resp. $\sA\Mod_Y$) the category of sheaves in $\sA\mod$ (resp. $\sA\Mod$) that are \textit{set-theoretically} supported on $Y$. 
Abbreviate $\Coh_Y(X)=\sA\mod_Y$ if $\sA=\sO_X$. 
We will view any algebra $A$ as a sheaf on $\Spec Z(A)$. 

For an algebraic group $K$, we will denote by $\Rep(K)$ the category of rational representations of $K$, and by $\rep(K)$ the full subcategory of finite dimensional representations in $\Rep(K)$. 

For a Hopf algebra $H$, we usually denote by $\Delta$ its comultiplication and by $\mathrm{S}$ its antipode. 
We will use the Sweedler's notation $\Delta(u)=u^{(1)}\otimes u^{(2)}$, for any $u\in H$. 

For a category $\sC$, we denote by $Z(\sC)=\End(1_\sC)$ its center. 
We will write $\Hom=\Hom_\sC$ if without confusions. 

\subsection{Acknowledgements} 
The author sincerely thanks his supervisor Peng Shan for helpful discussions. 
The author is deeply grateful to Simon Riche for his invitation of visiting Universit\'{e} Clermont Auvergne and for the enlightening discussions. 
The author is deeply grateful to Eric Vasserot for his invitation of visiting IMJ-PRG, Universit\'{e} Paris Cit\'{e} (where the major part of this work was done), and for the constant discussions and enlightening suggestions. 
These visits are supported by Tsinghua Scholarship for Overseas Graduate Studies. 
The author also thanks Lin Chen, Gurbir Dhillon, Yuchen Fu, Ivan Losev, Cris Negron and Ruotao Yang for stimulating discussions. 
The author is supported by NSFC Grant No. 12225108. 

\section{Quantum groups and their category $\sO$}\label{sect 2} 
\subsection{Root data}\label{subsect 2.1} 
Let $G$ be a connected and simply-connected semisimple algebraic group over $\k$, with a Borel subgroup $B$ and a maximal torus $T$ contained in $B$. 
Denote their Lie algebras by 
$$\g=\Lie(G),\quad \b=\Lie(B), \quad \n=\Lie(N),\quad \t=\Lie(T).$$ 
We identify $T_{\hat{1}}=\t_{\hat{0}}$ via the exponential map $\t\rightarrow T$. 
Denote by $\sB=G/B$ the flag variety. 
Let $W=N_G(T)/T$ be the Weyl group for $G$. 
Let $\I\subset W$ be the subset of simple reflections. 
Let $w_0\in W$ be the longest element, and set $B^-=w_0Bw^{-1}_0$, $N^-=w_0Nw^{-1}_0$.  

Let $(X^*(T),X_*(T), \Phi, \check{\Phi})$ be the root datum associated with $G$. 
Let $\langle -,-\rangle$ be the natural pairing of $\Lambda:=X^*(T)$ and $X_*(T)$. 
Let $\Phi^+\subset \Phi$ be the subsets of positive roots. 
Denote by $\alpha_s\in \Phi^+$ the simple root associated to $s\in \I$. 
Let $\rQ\subset \Lambda$ be the root lattices. 
There is a partial order $\leq$ on $\Lambda$ given by $\mu\leq \lambda$ if $\lambda-\mu\in \sum_{s\in \I} \N\alpha_s$. 
Let $(d_s)_{s\in \I}\in \N^\I$ be the unique tuple of relatively prime positive integers such that the matrix $\big(d_s \langle \check{\alpha}_{s}, \alpha_t\rangle \big)_{s,t\in \I}$ is symmetry and positive definite. 
It defines a pairing $(-,-):\rQ \times \rQ \rightarrow \Z$ by $(\alpha_s,\alpha_t):= d_s \langle \check{\alpha}_{s}, \alpha_t\rangle$ and extents to 
$$(-,-):\Lambda \times \Lambda \rightarrow \frac{1}{\e} \Z, \quad \e:=|\Lambda/\rQ|.$$ 

Let $l$ be an odd positive integer larger than the Coxeter number of $G$ which is prime to $\e$, and to $3$ if $G$ contains a component of type $G_2$. 
Let $\zeta_\e\in \k$ be a primitive $l$-th root of unity, and let $\zeta=(\zeta_\e)^\e$. 

\subsubsection{Affine Weyl group} 
Let $W_\af=W\ltimes \rQ$ be the affine Weyl group. 
Let $\I_{\af}$ be the subset of simple reflections in $W_{\af}$. 
Let $W_\ex=W\ltimes \Lambda$ be the extended affine Weyl group. 
We have $W_\ex\simeq W_\af\rtimes \Lambda/\rQ$, where $\Lambda/\rQ$ acts by automorphisms of the affine root system. 
Denote by $t(\lambda)$ the translation by $\lambda\in \Lambda$ in $W_\ex$. 
For any $n\in \Z$, there is an $n$-dilated $(-\rho)$-shifted action of $W_\ex$ on $\Lambda$ given by 
$$wt(\nu)\bullet_n \lambda=w(\lambda+n\nu+\rho)-\rho, \quad \forall w\in W,\ \forall \nu, \lambda\in \Lambda,$$ 
where $\rho=\frac{1}{2}\sum_{\alpha\in \Phi^+} \alpha$. 
We abbreviate $\bullet=\bullet_n$ if $n=1$. 

We set $\Xi=\Lambda/(W_{\ex},\bullet_l)$. 
For any $\omega\in \Xi$, we will fix a representative (still denoted by $\omega$) in the fundamental alcove $\Xi_\sc$ of $W_{\af}$-action on $\Lambda$, where 
$$\Xi_\sc:= \{\lambda \in \Lambda\ |\ 0\leq\langle \lambda+\rho, \check{\alpha} \rangle\leq l,\ \forall \alpha\in \Phi^+ \} \leftrightarrow \Lambda/(W_{\af},\bullet_l). $$ 

Let $\mathfrak{B}_\ex$ be the extended affine braid group associated with $W_\ex$, which admits standard generators $\{\tilde{s}\}_{s\in \I}$ and $\{\theta_\lambda\}_{\lambda\in \Lambda}$ (see their relations in e.g. \cite[\textsection 1.1]{BR12}). 
There is a canonical lift $W_\af\rightarrow \mathfrak{B}_\ex$ of the natural projection, such that $s\mapsto \tilde{s}$, $s\in \I$. 
We denote by $\tilde{s}$ the canonical lift of any $s\in \I_\af\backslash \I$ in $\mathfrak{B}_\ex$. 

\subsection{Quantum groups} 
Let $q$ be a formal variable. 
The quantum group $\mathscr{U}_q$ associated with $G$ is the $\k(q^{\frac{1}{\e}})$-algebra generated by the standard generators $E_i, F_i, K_\lambda \ (i\in \I, \lambda\in \Lambda)$ modulo the quantum Chevalley--Serre relations. 
It is a Hopf algebra with comultiplication, antipode and counit given by 
$$\Delta(E_i)=E_i\otimes K_{\alpha_i}+ 1\otimes E_i,\quad \Delta(F_i)=F_i\otimes 1+ K_{\alpha_i}^{-1}\otimes F_i,\quad \Delta(K_\lambda)=K_\lambda\otimes K_\lambda, $$ 
$$\mathrm{S}(E_i)=-E_iK_{\alpha_i}^{-1},\quad \mathrm{S}(F_i)=-K_{\alpha_i}F_i,\quad \mathrm{S}(K_\lambda)=K_\lambda^{-1},$$ 
$$\varepsilon(E_i)=\varepsilon(F_i)=0, \quad \varepsilon(K_\lambda)=1.$$ 
The \textit{Lusztig's quantum group} $U_q$ is a $\k[q^{\pm\frac{1}{\e}}]$-subalgebra of $\mathscr{U}_q$ generated by $E_i^{(n)}, F_i^{(n)}, K_\lambda$; the \textit{De Concini--Kac's quantum group} $\fU_q$ is generated by $E_i, F_i, K_\lambda$. 
The \textit{hybrid quantum group} $U^\hb_q$ is the $\k[q^{\pm\frac{1}{\e}}]$-subalgebra generated by $E_i^{(n)}, F_i, K_\lambda$. 
There are all Hopf subalgebras with inclusions 
$$\fU_q \subset U_q^{\hb} \subset U_q.$$ 
We denote the $\k(q^{\frac{1}{\e}})$-subalgebras $\mathscr{U}^+_q:=\langle E_i\rangle_{i\in \I}$, $\mathscr{U}^-_q:=\langle F_i\rangle_{i\in \I}$ and $\mathscr{U}^0_q:=\langle K_\lambda\rangle_{\lambda\in \Lambda}$, and denote by $\fU_q^+, \fU_q^-, \fU^0_q$ and $U_q^+, U_q^-, U^0_q$ their intersections with $\fU_q$ and $U_q$. 
There is a triangular decomposition 
$$U_q^{\hb}=\fU_q^-\otimes \fU_q^0 \otimes U_q^+.$$ 
We will identify $\mathscr{U}^0_q=\k(q^{\frac{1}{\e}})[\Lambda]=\k(q^{\frac{1}{\e}})[T]$. 

For any integral form $A_q$ above, we let the $\k$-algebra $A_\zeta:=A_q\otimes_{\k[q^{\pm\frac{1}{\e}}]} \k$ be the specialization at $q^{\frac{1}{\e}}= \zeta_\e$. 
The specialization yields a chain of maps 
$$\fU_\zeta \rightarrow U_\zeta^{\hb} \rightarrow U_\zeta .$$
Let $u_\zeta$ be the \textit{small quantum group} in $U_\zeta$, which coincides with the image of $\fU_\zeta\rightarrow U_\zeta$. 
Denote by $\fU^b_\zeta$ the image of $\fU_\zeta \rightarrow U_\zeta^{\hb}$. 
Then there are triangular decompositions 
$$u_\zeta=u_\zeta^-\otimes u_\zeta^0 \otimes u_\zeta^+, \quad \fU^b_\zeta= \fU^-_\zeta\otimes \fU^0_\zeta \otimes u_\zeta^+.$$ 
We set $\fU^t_\zeta$ be the algebra with triangular decomposition $\fU^t_\zeta=u_\zeta^-\otimes \fU_\zeta^0 \otimes u_\zeta^+$. 
For $A=A_q$ or $A_\zeta$ above, we abbreviate the subalgebras $A^\leq:=A^-A^0$ and $A^\geq:=A^+A^0$. 

\subsubsection{Harish-Chandra center} 
We set $\mathscr{U}_q^{0,ev}:=\k(q^{\frac{1}{\e}}) \langle K_{2\lambda}\rangle_{\lambda\in \Lambda}$. 
There is an algebra isomorphism 
$$Z(\mathscr{U}_q) \xs (\mathscr{U}_q^{0,ev})^{(W,\bullet)}$$ 
given by projecting $Z(\mathscr{U}_q)$ to $\mathscr{U}_q^0$ under the triangular decomposition, where $(W,\bullet)$ is the shifted action by $w\bullet K_\lambda =q^{(w\lambda -\lambda, \rho)}K_{w\lambda}$, for any $w\in W$, $\lambda\in\Lambda$. 
The natural identification 
$$(\mathscr{U}_q^{0,ev})^{(W,\bullet)}= \k(q^{\frac{1}{\e}})[T]^{(W,\bullet)},\quad f(K_{2\lambda})\mapsto f(K_\lambda),$$ 
gives an isomorphism 
$$Z(\mathscr{U}_q) \xs \k(q^{\frac{1}{\e}})[T/(W,\bullet)].$$ 
The center of $\fU_q$ is by $Z(\fU_q)= Z(\mathscr{U}_q) \cap \fU_q$. 
We obtain an isomorphism 
$$Z_\HC:= Z(\fU_q)/(q^{\frac{1}{\e}}-\zeta_\e)Z(\fU_q) \xs \k[T/(W,\bullet)].$$ 
Here $Z_\HC$ is called the \textit{Harish-Chandra center} of $\fU_\zeta$. 
For any $\lambda\in \Lambda$, we denote by $\chi_\lambda$ the character of $Z_\HC$ corresponding to the image of $\zeta^\lambda$ under the map $T\rightarrow T/(W,\bullet)$, by $t\mapsto W\bullet t^2$ for any $t\in T$. 

\subsubsection{Frobenius center} 
The \textit{Frobenius center} of $\fU_\zeta$ is the $\k$-subalgebra 
$$Z_\Fr:=\langle K^l_{\lambda}, F^l_\beta , E^l_\beta \rangle_{\lambda\in \Lambda, \beta\in \Phi^+}, $$ 
where $E_\beta\in \fU^+_\zeta$ and $F_\beta\in \fU^-_\zeta$ are the root vectors. 
We abbreviate $Z_\Fr^{\flat}:= Z_\Fr \cap \fU_\zeta^{\flat}$ for $\flat=-,+,0,\leq$ and $\geq$. 
By \cite{DeCKP92} (see also \cite{Tan12}) there are isomorphisms of $\k$-algebras 
\begin{equation}\label{equ 2.0} 
Z_\Fr^{-}\xs \k[N],\quad Z_\Fr^{+}\xs \k[N^-], \quad \text{and}\quad Z_\Fr^{0}\xs \k[T],
\end{equation}
which give an isomorphism 
$$Z_\Fr \xs \k[G^*], $$ 
where $G^*=N^-\times T \times N$ is the Poisson dual group of $G$. 
There is a unramified covering to the big open cell $\kappa: G^*\rightarrow G$, sending $(n_1,t,n_2)\in N^-\times T \times N$ to $n_1t^2n_2^{-1}$. 
We have the following isomorphisms by \cite{DeCP92}, 
\begin{equation}\label{equ 2.-1} 
Z(\fU_\zeta)=Z_\Fr\otimes_{Z_\Fr\cap Z_\HC}Z_\HC=\k[G^*\times_{T/W}T/(W,\bullet)], 
\end{equation} 
where $G^*\rightarrow T/W$ is the composition of $\kappa$ with $G\rightarrow G/\!\!/G=T/W$, and $T/(W,\bullet) \rightarrow T/W$ is induced by taking $l$-th power. 

\subsection{Module categories} 
There is a $\Lambda$-action on $\mathscr{U}_q^0$ such that any $\mu\in \Lambda$ corresponds to the $\k(q^{\frac{1}{\e}})$-algebra automorphism  
$$\tau_\mu: K_\lambda \mapsto q^{(\mu,\lambda)}K_\lambda, \quad \forall\lambda\in\Lambda.$$ 
The $\Lambda$-action on $\mathscr{U}_q^0$ preserves the integral forms $\fU^0_q$ and $U^0_q$, and specializes to an action on $\fU^0_\zeta$ and $U^0_\zeta$. 
Let $U^0$ be a Noetherian sub-quotient algebra of $\fU_\zeta^0$ or $U_\zeta^0$ such that the $\Lambda$-action induces an action on $U^0$. 
Let $U=\bigoplus_{\lambda\in \rQ} U_\lambda$ be a Noetherian $\rQ$-graded $\k$-algebra with triangular decomposition $U=U^-\otimes U^0\otimes U^+$, such that 
$$fm=m\tau_\lambda(f),\quad \forall f\in U^0,\ \forall m\in U_\lambda,$$ 
and satisfies further conditions as in \cite[\textsection 2.3]{Situ1}. 
We abbreviate $U^\leq=U^-U^0$ and $U^\geq=U^+U^0$. 
A \textit{deformation ring} $R$ for $U$ is a commutative Noetherian $U^0$-algebra. 
Let $\iota: U^0\rightarrow R$ be the structure map. 
If $U^0$ is an augmented algebra, then we will view $\C$ as a deformation ring via the counit map $U^0\rightarrow \C$. 

We define $U\Mod_R^{\Lambda}$ to be the category consisting of $U\otimes R$-modules $M$ endowed with a decomposition $M=\bigoplus\limits_{\mu\in \Lambda}M_\mu$ of $R$-modules (called the \textit{weight spaces}), such that $M_\mu$ is killed by the elements in $U\otimes R$ of the form 
$$f\otimes 1-1\otimes \iota({\tau_\mu}(f)), \quad f\in U^0.$$ 
Let $U\mod_R^{\Lambda}$ be the full subcategory of $U\Mod_R^{\Lambda}$ consisting of finitely generated $U\otimes R$-modules whose weight spaces are finitely generated $R$-modules. 
Define the \textit{category $\sO$ for $U$} to be the full subcategory $\sO^U_R$ of $U\mod_R^{\Lambda}$ of modules that are locally unipotent for the action of $U^+$. 
It is an abelian subcategory of $U\Mod_R^{\Lambda}$. 
If $U^+$ is finite dimensional, then we have $\sO^U_R=U\mod^\Lambda_R$. 

Define the \textit{Verma module} 
$$M^U(\lambda)_R:= U \otimes_{U^{\geq}} R_\lambda \ \in \sO^U_R$$ 
where $R_\lambda$ is a $U^{\geq}$-module via $U^{\geq}\twoheadrightarrow U^0\xrightarrow{\iota\circ\tau_{\lambda}}R_\lambda$. 
We have an isomorphism $M^U(\lambda)_R=U^-\otimes R_\lambda$ as $U^\leq$-modules. 
A \textit{Verma flag} of an object in $U\Mod_R^{\Lambda}$ is a filtration \textit{of finite length} with composition factors by Verma modules. 
The composition factors of such an object are called \textit{Verma factors}, which (including multiplicities) are independent on the choice of Verma flags. 

If $R=\F$ is a field, then $M^U(\lambda)_\F$ admits a unique simple quotient $ L^U(\lambda)_\F$, and $\{L^U(\lambda)_\F\}_{\lambda\in \Lambda}$ provides a complete family of pairwise non-isomorphic simple objects in $\sO^U_R$. 

The following is standard, see e.g. \cite[Prop 3.1]{Hum08}. 
\begin{lem}\label{lem 2.1} 
We have 
\begin{equation}\label{equ new2.3}
\Hom(M^U(\lambda)_{R}, M^U(\mu)_{R})\neq 0 \quad \text{only if}\ \lambda\leq \mu,
\end{equation}
\begin{equation}\label{equ new2.4}
\Ext^1(M^U(\lambda)_{R}, M^U(\mu)_{R})\neq 0 \quad \text{only if}\ \lambda< \mu.
\end{equation}
\end{lem}

\subsubsection{Truncated categories}\label{subsect 2.3.1} 
For any subset $\Omega \subset \Lambda$, 
there is a \textit{truncated subcategory} (by $\Omega$) 
$$U\Mod_R^{\Omega}$$ 
consisting of the module $M$ in $U\Mod_R^{\Lambda}$ such that $M_\mu=0$ unless $\mu \in \Omega$. 
It is a Serre subcategory of $U\Mod_R^{\Lambda}$. 
The \textit{truncation functor} 
$$\tau^{\Omega}:\ U\Mod_R^{\Lambda} \rightarrow U\Mod_R^{\Omega} ,\quad 
M \mapsto M  \big/ U. \big( \bigoplus_{\mu\notin \Omega} M_\mu \big)$$ 
by taking the maximal quotient in $U\Mod_R^{\Omega}$ is left adjoint to the natural inclusion functor $U\Mod_R^{\Omega} \hookrightarrow U\Mod_R^{\Lambda}$. 
We abbreviate $U\mod_R^{\Omega}:=U\mod_R^{\Lambda}\cap U\Mod_R^{\Omega}$. 

Let now $\Omega\subset \Lambda$ be a bounded above poset ideal. 

Then any finitely generated $U\otimes R$-module in $U\Mod_R^{\Omega}$ is contained in $\sO^{U}_R$. 
We set 
$$\sO^{U,\in \Omega}_R:=\sO^{U}_R\cap U\Mod_R^{\Omega}.$$ 
The category $\sO^{U}_R$ is the union of its full subcategories $\sO^{U,\in \Omega}_R$, running over all bounded above poset ideals $\Omega\subset \Lambda$. 
The category $\sO^{U,\in \Omega}_R$ always admits enough projective objects, in contrast to $\sO^{U}_R$. 
Moreover, there is a generating family of projective objects in $\sO^{U,\in \Omega}_R$ admitting Verma flags, with factors of the form $M^U(\lambda)_R$ with $\lambda\in \Omega$, see e.g. \cite[Lem 2.3]{Fie03}. 
If $R$ is local, then any projective module in $\sO^{U,\in \Omega}_R$ admits Verma flags, see e.g. \cite[Lem 2.5]{Fie03}. 
If $R$ is a complete local domain with residue field $\F$, we write $Q^U(\lambda)_R^{\Omega}$ as the projective cover of $L^U(\lambda)_\F$ in $\sO^{U,\in \Omega}_R$, for any $\lambda\in \Omega$. 

Let $R'$ be a commutative Noetherian $R$-algebra. 
We have a natural functor 
$$-\otimes_R R':\ U\Mod^{\Lambda}_R\rightarrow U\Mod^{\Lambda}_{R'}.$$
By \cite[Prop 2.4]{Fie03}, it yields an equivalence 
\begin{equation}\label{equ basechange}
\Proj(\sO^{U,\in \Omega}_R)\otimes_R R' \xs \Proj(\sO^{U,\in \Omega}_{R'}),
\end{equation}
for any bounded above poset ideal $\Omega$ in $\Lambda$.

\begin{lem}\label{lem 2.1} 
Let $\Omega$ be a poset ideal in $\Lambda$. 
For any object in $U\Mod_R^{\Lambda}$ admitting a Verma flag, its truncation in $U\Mod_R^{\Omega}$ is the quotient by the submodule composed by its Verma factors in $\{M^U(\lambda)_{R}\}_{\lambda\notin \Omega}$. 
\end{lem} 
\begin{proof}
Let $P$ be an object in $U\Mod_R^{\Lambda}$ admitting a Verma flag. 
By (\ref{equ new2.4}), there is a short exact sequence $0\rightarrow P_1\rightarrow P\rightarrow P_2\rightarrow 0$, where $P_1$ is the submodule composed by the Verma factors in $\{M^U(\lambda)\}_{\lambda\notin \Omega}$ and $P_2$ is the quotient module composed by the Verma factors in $\{M^U(\lambda)\}_{\lambda\in \Omega}$. 
It is clear that $P_2\in U\Mod_R^{\Omega}$. 
Note that we have 
$$\tau^{\Omega}M^U(\lambda)_R =\begin{cases}
M^U(\lambda)_R, & \lambda\in \Omega \\ 0, & \lambda\notin \Omega
\end{cases}.$$ 
Using the fact that $\tau^\Omega$ is right exact, we deduce that $\tau^\Omega(P_1)=0$ and that $\tau^\Omega(P)=\tau^\Omega(P_2)=P_2$. 
\end{proof}

For $\Omega=\{\lambda\in \Lambda|\lambda\leq \nu\}$, we abbreviate ``$\Omega$" and ``$\in \Omega$" by ``$\leq \nu$" in all the superscripts above. 
We will drop the superscripts $U$ from the notations above when $U=U^\hb_\zeta$. 
For example, if $\Omega=\{\lambda\in \Lambda|\lambda\leq \nu\}$ and $U=U^\hb_\zeta$, we abbreviate $\sO^{U,\in \Omega}_R$, $Q^{U}(\lambda)_R^{\Omega}$ by $\sO^{\leq \nu}_R$, $Q(\lambda)_R^{\leq \nu}$. 

\subsubsection{Projectives and simples in $\sO_R$} 
We set 
$$S=\k[T]_{\hat{1}}=\k[\t]_{\hat{0}}.$$ 
Then $S$ is naturally an algebra of $\fU^0_\zeta=\k\langle K_\lambda \rangle_{\lambda\in \Lambda}=\k[T]$. 
We thus view $S$ as a deformation ring for both $\fU^b_\zeta$ and $U^\hb_\zeta$. 

\begin{lem}[{\cite[Lem 3.1]{Situ1}}]\label{lem 2.1*} 
The $U_\zeta^\hb$-module $L(\lambda)_\C$ remains simple as a $\fU^b_\zeta$-module, and the $\fU^b_\zeta$-action on $L(\lambda)_\C$ factors through the quotient $\fU^b_\zeta\twoheadrightarrow u_\zeta$. 
Hence $L(\lambda)_\C$ is a simple $u_\zeta$-module. 
\end{lem} 

Let $R$ be a complete local domain over $\fU^0_\zeta$. 
Let $P^b(\lambda)_R$ be the projective cover of the simple module of highest weight $\lambda$ in $\fU^b_\zeta\mod^\Lambda_R$. 
Consider the induced module 
$$Q(\lambda)_R=U^\hb_\zeta\otimes_{\fU^b_\zeta} P^b(\lambda)_R$$ 
in $U^\hb_\zeta\Mod^\Lambda_R$. 

\begin{lem}[{\cite[Lem 3.7]{BBASV} and \cite[Lem 3.2]{Situ1}}]\label{lem 3.2} \label{lem 2.2}
\ 
\begin{enumerate} 
\item The functor $\Hom_{U^\hb_{\zeta}\Mod^{\Lambda}_R}(Q(\lambda)_R,-)$ is exact on $\sO_{R}$; 
\item Let $\Omega$ be a poset ideal in $\Lambda$, and let $\lambda\in \Omega$. 
The projective cover for $L(\lambda)_\C$ in $\sO^{\in \Omega}_{S}$ (resp. $\sO^{\in \Omega}_{\C}$) is $Q(\lambda)_S^{\Omega}=\tau^{\Omega}Q(\lambda)_S$ (resp. $Q(\lambda)_\C^{\Omega}=\tau^{\Omega}Q(\lambda)_\C$). 
\end{enumerate}
\end{lem}

\subsubsection{Block decomposition}\label{subsect 2.2.1} 
Till the end of this section, we let $U$ be either $\fU^t_\zeta$, $\fU^t_\zeta U^{+}_\zeta$, $\fU^b_\zeta$, $\fU_\zeta$ or $U^\hb_\zeta$. 
Then there are algebra homomorphisms $Z_\HC \hookrightarrow \fU_\zeta \rightarrow U$. 
Let $R$ be a commutative Noetherian $S$-algebra. 
We have algebra homomorphisms 
$$S\otimes Z_\HC \rightarrow R\otimes Z_\HC\rightarrow Z(U\Mod^\Lambda_{R}).$$ 
By the Harish-Chandra isomorphism, it factors through a homomorphism 
\begin{equation}\label{equ 2.2} 
\k[T_{\hat{1}}\times_{T/W} T/(W,\bullet)] \rightarrow Z(U\Mod^\Lambda_{S}).
\end{equation}
There are algebra isomorphisms (see a specialized version in \cite[Thm 4.5]{BG01}) 
\begin{equation}\label{equ 2.1} 
\begin{aligned}
\k[T_{\hat{1}}\times_{T/W} T/(W,\bullet)] &= 
\bigoplus_{\omega\in \Xi} \k[T]_{\hat{1}} \otimes_{\k[T/W]_{\hat{1}}} \k[T/(W,\bullet)]_{\widehat{\chi_\omega}} \\ 
&\simeq \bigoplus_{\omega\in \Xi} \k[\t]_{\hat{0}} \otimes_{\k[\t/W]_{\hat{0}}} \k[\t/W_{\omega}]_{\hat{0}} \\ 
&= \bigoplus_{\omega\in \Xi} \k[\t\times_{\t/W} \t/W_{\omega}]_{\hat{0}}, 
\end{aligned}
\end{equation} 
where $W_{\omega}$ is the stabilizer of the $(W,\bullet)$-action on $\zeta^\omega\in T$. 
The idempotents on the RHS of (\ref{equ 2.1}) yield a block decomposition 
$$U\Mod^\Lambda_{R}= \bigoplus_{\omega\in \Xi} U\Mod^{\Lambda,\omega}_{R},$$ 
such that the Verma module $M^U(\lambda)_R$ is contained in $U\Mod^{\Lambda,\omega}_{R}$ if and only if $\omega=W_{\ex}\bullet_l \lambda$. 
We have algebra homomorphisms 
\begin{equation}\label{equ new2.7}
\k[\t\times_{\t/W} \t/W_{\omega}]_{\hat{0}}\rightarrow Z(U\Mod^{\Lambda,\omega}_{S}), \quad \forall \omega\in \Xi.
\end{equation}
The block decomposition induces a decomposition 
$$\sO^{U}_R= \bigoplus_{\omega\in \Xi} \sO^{U,\omega}_R.$$ 
Since $u_\zeta$, $u_\zeta U^+_\zeta$ are quotients of $\fU_\zeta^t$, $\fU^t_\zeta U^+_\zeta$, we obtain block decompositions 
$$u_\zeta\mod^{\Lambda}_\C=\bigoplus_{\omega\in \Xi} u_\zeta\mod^{\Lambda,\omega}_\C \quad \text{and}\quad 
\sO^{u_\zeta U^+_\zeta}_\C=\bigoplus_{\omega\in \Xi} \sO^{u_\zeta U^+_\zeta,\omega}_\C. $$ 

\begin{rmk}
Since $U$ is $\rQ$-graded and $U\Mod^{\Lambda}_{S}$ consists of $\Lambda$-graded modules, the blocks above labelled by $\omega\in \Xi$ further decompose into direct sums of $\e$ equivalent sub-blocks. 
\end{rmk}

For a subset $\Omega\subset \Lambda$, we abbreviate $U\Mod^{\Omega ,\omega}_{S}:=U\Mod^{\Lambda ,\omega}_{S}\cap U\Mod^{\Omega}_{S}$ and $\sO^{U,\omega,\in \Omega}_R:=\sO^{U,\omega}_R\cap U\Mod^{\Omega}_{S}$. 
Again, for $\Omega=\{\lambda\in \Lambda|\lambda\leq \nu\}$, we abbreviate ``$\Omega$" and ``$\in \Omega$" by ``$\leq \nu$" in the superscripts. 

We will also drop the superscripts $U$ from the notations above when $U=U^\hb_\zeta$. 

\subsubsection{Translation functors} 
Let $R$ be a commutative Noetherian $S$-algebra. 
For any $\omega_1,\omega_2 \in \Xi$, there is a unique dominant element $\nu$ in the $W$-orbit of $\omega_2-\omega_1$. 
Denote by $V(\nu)$ the Weyl module for $U_\zeta$ of highest weight $\nu$, see e.g. \cite{APW91}. 
We view $V(\nu)$ as a $U$-module via the natural map $U\rightarrow U_\zeta$. 
We define the \textit{translation functors} by 
$$\rT_{\omega_1}^{\omega_2} : U\Mod^{\Lambda,\omega_1}_{R} 
\rightarrow U\Mod^{\Lambda,\omega_2}_{R}, \quad M \mapsto \pr_{\omega_2} (M \otimes V(\nu)),$$ 
$$\rT^{\omega_1}_{\omega_2} : U\Mod^{\Lambda,\omega_2}_{R} \rightarrow 
U\Mod^{\Lambda,\omega_1}_{R}, \quad M \mapsto \pr_{\omega_1} (M \otimes V(\nu)^*),$$ 
where $\pr_{\omega_i}$ is the natural projection to the block $U\Mod^{\Lambda,\omega_i}_{R}$. 
Note that $\rT_{\omega_1}^{\omega_2}$ and $\rT^{\omega_1}_{\omega_2}$ are exact and biadjoint to each other. 
The translation functors preserve the subcategory $\sO^{U}_{R}$ and induce translation functors between blocks in $\sO^{U}_R$. 

For any simple reflection $s\in \I_{\af}$, we fix an element $\omega_s\in \Xi_\sc$ whose stabilizer under the $W_{\af}$-action is $\{1,s\}$. 
We abbreviate $\rT^0_s=\rT^0_{\omega_s}$ and $\rT_0^s=\rT_0^{\omega_s}$. 
Define the \textit{reflection functor} on $U\Mod^{\Lambda,0}_R$ by 
$$\Theta^r_s :=\rT^0_{s}\circ \rT_0^s.$$ 
We also consider the functor $\T^r_s :=\ker(\Theta^r_s\xrightarrow{\text{counit}} 1)$. 

We also obtain translations and reflections functors for the blocks of $u_\zeta\mod^{\Lambda}_\C$ and $\sO^{u_\zeta U^+_\zeta}_\C$. 

We can similarly define translation functors and reflection functors on the categories $\fU_\zeta\mod_{\chi_\omega}$ with $\omega\in \Xi$. 

\subsubsection{$\Lambda$-translations} 
Let $U'$ be $U$ as in \textsection\ref{subsect 2.2.1}, or be $u_\zeta$, $u_\zeta U^+_\zeta$. 
Let $R$ be a deformation ring for $U'$. 
For any $\nu\in \Lambda$, there is a trivial $U'$-module $\k_{l\nu}$ supported on the weight $l\nu$. 
It gives auto-equivalences $-\otimes \k_{l\nu}$, $\nu\in \Lambda$ on $U'\Mod^{\Lambda}_R$ preserving each blocks, called the \textit{$\Lambda$-translations}.

\section{Non-commutative Springer resolutions and localization theorem}\label{sect 3} 
In this section, we give reminders on the non-commutative Springer resolutions developed in \cite{BM13} and on the localization theorem for $\fU_\zeta$ following \cite{Tan21}. 

\subsection{Non-commutative Springer resolutions} 
For any standard parabolic subgroup $P$ in $G$, we set $\widetilde{\g}_P=G\times^P \Lie(P)$ and abbreviate $\widetilde{\g}=\widetilde{\g}_B$. 
There is a natural map $\pi_P: \widetilde{\g}\rightarrow \widetilde{\g}_P$. 
The \textit{Grothendieck resolution} is $\pi:\widetilde{\g}\rightarrow \g$. 
If $P=P_s$ is the minimal parabolic subgroup associated to $s\in \I$, we will abbreviate $\widetilde{\g}_s=\widetilde{\g}_{P_s}$ and $\pi_s=\pi_{P_s}$. 
Denote by $\sO_\sB(\lambda)$ the line bundle $G\times^B \k_\lambda \rightarrow \sB$, and denote by $\sO_{Y}(\lambda)$ its pullback along any $Y\rightarrow \sB$. 
Consider the $\C^\times$-action on $\g$ via $t.x=t^{-2}x$, for any $t\in \C^\times$, $x\in \g$. 
It induces $\C^\times$-actions on $\widetilde{\g}_P$ for any $P$. 

\subsubsection{Affine braid group actions} 
Let $X$ be a $\g$-scheme satisfying 
$${\sT or}_{i}^{\sO_\g}(\sO_{\widetilde{\g}}, \sO_X)=0, \quad \forall i>0.$$ 
We set $\widetilde{\g}_{P,X}:=\widetilde{\g}_P\times_\g X$ 
and write $\pi_{P}=\pi_P\times_\g X$ by abuse of notations. 
There is a functor $\Theta_{P}^c:=\pi_P^* \pi_{P*}$ on $\Db\Coh(\widetilde{\g}_X)$. 
We abbreviate $\Theta^c_s=\Theta^c_{P_s}$ for any $s\in \I$, which are called the \textit{reflection functors}. 

\begin{thm}[{\cite{BR12}}]\label{thm 3.1} 
For any $s\in \I$, there is a self-equivalence $\T^c_s$ on $\Db\Coh(\widetilde{\g}_X)$ with natural transformations 
$$\T^c_s \rightarrow \Theta^c_s\xrightarrow{\text{counit}} 1 \xrightarrow{+1},$$ 
which is a distinguished triangle when applying on any object in $\Db\Coh(\widetilde{\g}_X)$. 
The assignments 
$$\tilde{s} \mapsto (\T^c_s)^{-1} \quad \text{and}\quad 
\theta_\lambda\mapsto -\otimes^L_{\sO_{\widetilde{\g}_X}} \sO_{\widetilde{\g}_X}(\lambda), \qquad \forall s\in \I,\ \forall \lambda\in \Lambda,$$ 
define a (weak, right) $\mathfrak{B}_\ex$-action on $\Db\Coh(\widetilde{\g}_X)$. 
Moreover, if $X\rightarrow \g$ is equivariant for a closed subgroup $K\subset G\times \C^\times$, the constructions above are naturally $K$-equivariant. 
\end{thm} 

For any affine reflection $s\in \I_\af\backslash \I$, there is an element $b\in \mathfrak{B}_\ex$ and a finite simple reflection $s'\in \I$ satisfying $\tilde{s}=b\tilde{s'}b^{-1}$, see \cite[Lem 6.1.2]{R10}. 
We define the reflection functor associated to $s$ as $\Theta^c_{s}=b^{-1}\circ \Theta^c_{s'}\circ b$, where $b$ is viewed as an auto-functor on $\Db\Coh(\widetilde{\g}_X)$. 

\subsubsection{Non-commutative Springer resolutions}\label{subsect 3.1.2} 
Let $\mathfrak{B}^+_\ex$ be the semi-group of $\mathfrak{B}_\ex$ generated by the canonical lifts $\tilde{s}$ with $s\in \I_\af$. 

\begin{thm}[{\cite{BM13}}]\label{thm BM} 
\begin{enumerate}
\item There is a unique t-structure (called the \textit{exotic t-structure}) on $\Db\Coh(\widetilde{\g}_X)$ defined by 
$$\Db\Coh(\widetilde{\g}_X)^{\leq}=\{\sF|\ \pi_*(b\sF)\in D^{\leq}\Coh X,\ \forall b\in \mathfrak{B}^+_\ex\};$$ 
$$\Db\Coh(\widetilde{\g}_X)^{\geq}=\{\sF|\ \pi_*(b^{-1}\sF)\in D^{\geq}\Coh X,\ \forall b\in \mathfrak{B}^+_\ex\}.$$ 

\item There exists a $G\times \C^\times$-equivariant tilting bundle $\sE$ on $\widetilde{\g}$ such that 
$$R\sHom_{\widetilde{\g}_X}(\sE_X,-):\ \Db\Coh(\widetilde{\g}_X) \rightarrow \Db(\A_X\mod)$$ 
is an equivalence, where $\sE_X:=\sE \times_\g X$, $\A_X=\A\otimes_{\sO_\g} \sO_X$ and $\A=\sEnd_{\widetilde{\g}}^\op(\sE)$, and such that this equivalence is t-exact with respect to the exotic t-structure on the LHS and the trivial t-structure on the RHS. 
\item Moreover, if $X\rightarrow \g$ is equivariant for a closed subgroup $K\subset G\times \C^\times$, we have similar statements for $\Db\Coh^K(\widetilde{\g}_X)$. 
\end{enumerate} 
\end{thm} 

By definition, the functor $\pi_*$ is t-exact for the exotic t-structure on $\Db\Coh(\widetilde{\g}_X)$. 
It yields a functor $\pi_*: \A_X\mod \rightarrow \Coh(X)$. 
By \cite{BM13}, the functors $\Theta^c_P$, $\Theta^c_s$ ($s\in \I_\af$) are also exact for the exotic t-structure. 
We thus view them as functors on $\A_X\mod$. 

Consider the completion $X=\g_{\hat{0}}$, then there is an equivalence 
$$\Db(\A_{\g_{\hat{0}}}\mod)\simeq \Db\Coh(\widetilde{\g}_{\hat{0}}),$$ 
whose restriction on the full subcategories of objects set-theoretically supported at $0\in \g$ induces an equivalence 
\begin{equation}\label{equ 3.5} 
\Db(\A\mod_0)\simeq \Db\Coh_{\sB}(\widetilde{\g}). 
\end{equation} 

\subsection{Localization theorem for $\fU_\zeta$}\label{subsect 3.2} 
The localization theorem for $\fU_\zeta$ was established by Backelin--Kremnizer \cite{BK08} and Tanisaki \cite{Tan12, Tan14, Tan21}. 
In this subsection, we give a reminder on the localization theorem mainly following \cite{Tan21}. 
\subsubsection{Quantum differential operators} 
Lunts---Rosenberg \cite{LR99} introduced the quantized flag variety $\sB_\zeta$ as a non-commutative scheme over $\sB$. 
Let $\sO_{\sB_\zeta}$ be the pushforward of the structure sheave of $\sB_\zeta$ on $\sB$, which is a coherent sheaf of (in general non-commutative) algebras on $\sB$. 
In \cite{Tan05}, Tanisaki defined the ring of enhanced differential operators on 
$\sB_\zeta$ (note that in \textit{loc. cit.} the convention for flag variety is $B^-\backslash G$, which is different from us). 
Let $\widetilde{\sD}_\zeta$ be its localization on $\sB$. 
There is an algebra homomorphism 
\begin{equation}\label{equ 3.99}
\fU_\zeta \otimes \k[T] \rightarrow \Gamma(\widetilde{\sD}_\zeta).
\end{equation} 

Consider the scheme 
$$\widetilde{G^*}=\{(gB,\chi)\in \sB \times G^*\ |\ g^{-1}\kappa(\chi)g \in B\},$$ 
and set 
$$\sV=\widetilde{G^*}\times_T T 
=\{(gB,\chi,t)\in \sB \times G^*\times T \ |\ g^{-1}\kappa(\chi)g \in t^lN\},$$ 
where $T\rightarrow T$ is by taking $l$-th power, and $\widetilde{G^*}\rightarrow T$ is induced by $B=N\times T\rightarrow T$, $nt\mapsto t^2$ for any $n\in N$ and $t\in T$. 
Note that $\sV$ is affine over $\sB$. 
By \cite[Thm 5.2]{Tan12}, there is an isomorphism from the central subsheaf of $\widetilde{\sD}_\zeta$ to the structure sheaf $\sO_{\sV}$ of $\sV$. 
We thus also view $\widetilde{\sD}_\zeta$ as a (coherent) sheaf of algebras on $\sV$. 
The homomorphism $\fU_\zeta\rightarrow \widetilde{\sD}_\zeta$ restricting on their centers is compatible with the natural map 
$$\varpi:\ \sV \rightarrow {G^*}\times_{T/W} T \rightarrow {G^*}\times_{T/W} T/(W,\bullet).$$ 

\subsubsection{Localization theorem}\label{subsect 3.2.2} 
From now on, we assume moreover that $l$ is a prime power, following \cite{Tan21}. 

The homomorphism $\fU_\zeta\rightarrow \Gamma(\widetilde{\sD}_\zeta)$ yields a functor 
$$R\Gamma: \Db(\widetilde{\sD}_\zeta\mod)\rightarrow \Db(\fU_\zeta \mod). $$ 

\begin{thm}[{\cite{Tan21}}]\label{thm 3.3} 
For any $(\chi,\zeta^{2\lambda})\in {G^*}\times_{T/W} T$ (where $\lambda\in \Lambda$), if $\lambda$ is regular (i.e. the stabilizer of $(W_\ex,\bullet_l)$-action on $\lambda$ is trivial), then there is an equivalence 
$$R\Gamma: \Db(\widetilde{\sD}_\zeta\mod_{(\chi,\zeta^{2\lambda})})\xs \Db(\fU_\zeta \mod_{(\chi,\chi_\lambda)}).$$ 
\end{thm} 

\subsubsection{Splitting Azumaya algebras}\label{subsect 3.2.3} 
For any $(\chi,t)\in {G^*}\times_{T/W} T$, we denote by ${\sV}_{(\chi,t)}$ the fiber of $(\chi,t)$ in $\sV$. 
Although the following result holds for any ${\sV}_{(\chi,t)}$ with more assumptions on $l$, we only consider the case when $t=\zeta^{\lambda}$ for some $\lambda\in \Lambda$. 

\begin{thm}[{\cite{Tan21}}]\label{thm 3.4}
The sheaf $\widetilde{\sD}_\zeta$ is an Azumaya algebra on $\sV$, whose restriction on the formal neighborhood of ${\sV}_{(\chi,\zeta^{\lambda})}$ splits, for any $(\chi,\zeta^{\lambda}) \in {G^*}\times_{T/W} T$. 
Therefore, for any splitting bundle $\fM^{\chi}_{\zeta^\lambda}$ of $\widetilde{\sD}_\zeta$ on $\FN(\sV_{(\chi,\zeta^{\lambda})})$, there is an equivalence 
\begin{equation}\label{equ 3.98} 
\fM^{\chi}_{\zeta^\lambda} \otimes_{\sO_{\FN({\sV}_{(\chi,\zeta^\lambda)})}}-: 
\Coh_{{\sV}_{(\chi,\zeta^\lambda)}}(\sV)\xs \widetilde{\sD}_\zeta\mod_{(\chi,\zeta^\lambda)}.
\end{equation} 
\end{thm} 

Let $(\chi,t)=(1,\zeta^{\lambda})\in {G^*}\times_{T/W} T$. 
Then we have 
$$\sV_{(1,\zeta^{\lambda})}=\sB\times \{(1,\zeta^{\lambda})\}\simeq \sB \quad \text{in}\quad \sV.$$ 
We abbreviate $\fM^{}_{\lambda}=\fM^{1}_{\zeta^{2\lambda}}$. 

The equivalence (\ref{equ 3.98}) above depends on the choice of splitting bundles. 
We will fix and recall the construction of splitting bundles $\fM_\lambda$ given in \cite[\textsection 5.4]{Tan21} as follows. 
Firstly, we consider the completion of $\fU_\zeta$ at $(1,\chi_{-\rho})\in G^*\times_{T/W} T/(W,\bullet)$. 
By \cite[Thms 2.5 and 4.9]{BG01}, the algebra $(\fU_\zeta)_{\widehat{(1,\chi_{-\rho})}}$ is a split Azumaya algebra over its center 
$$\k[G^*\times_{T/W} T/(W,\bullet)]_{\widehat{(1,\chi_{-\rho})}}=\k[G^*]_{\hat{1}}.$$ 
Fix a splitting module $M_{\chi_{-\rho}}$ of $(\fU_\zeta)_{\widehat{(1,\chi_{-\rho})}}$, then we have an equivalence 
\begin{equation}\label{equ 3.2}
M_{\chi_{-\rho}}\otimes_{\k[\g]_{\hat{0}}}-:\ 
\Coh_1(G^*) \xs \fU_\zeta\mod_{(1,\chi_{-\rho})}. 
\end{equation} 
On the other hand, by \cite{Tan12} the natural map $\varpi^*\fU_\zeta \rightarrow \widetilde{\sD}_\zeta$ yields an isomorphism 
\begin{equation}\label{equ 3.0} 
\varpi^*\fU_\zeta|_{\FN(1,\chi_{-\rho})} \xs \widetilde{\sD}_\zeta|_{\FN({\sV}_{(1,\zeta^{-2\rho})})}. 
\end{equation} 
Hence $\fM^{}_{{-\rho}}=\varpi^*M_{\chi_{-\rho}}$ is a splitting bundle for $\widetilde{\sD}_\zeta$ on $\FN({\sV}_{(1,\zeta^{-2\rho})})$. 
In general, in \cite{Tan05} the author constructed a locally free $\sO_{\sB_\zeta}$-module $\sO^{\lambda}_{\sB_\zeta}$ of rank 1 for any $\lambda\in \Lambda$. 
By \cite[Thm 5.3]{Tan21}, one can set $\fM^{ }_{\lambda}=\sO_{\sB_\zeta}^{\lambda+\rho}\otimes_{\sO_{\sB_\zeta}} \fM_{{-\rho}}$. 

\subsubsection{Formal neighborhood of $\sB$} 
We set $\widetilde{G}=G\times^B B$ where $B$ acts on itself by conjugation. 
There is a natural map $\widetilde{G}\rightarrow G$. 
Note that $\widetilde{G^*}$ fits into the Cartesian diagram 
$$\begin{tikzcd}
\widetilde{G^*} \arrow[r,"\tilde{\kappa}"]\arrow[d] 
& \widetilde{G} \arrow[d]\\ 
G^* \arrow[r,"\kappa"] 
& G. 
\end{tikzcd}$$ 
Since $\kappa$ is \'{e}tale, $\tilde{\kappa}$ is \'{e}tale. 
It induces isomorphisms of formal neighborhoods 
$$\FN_{\sV}(\sB\times \{(1,\zeta^\lambda)\})
\simeq \FN_{\widetilde{G^*}}(\sB)
\simeq \FN_{\widetilde{G}}(\sB), \quad \forall \lambda \in \Lambda.$$ 
The exponential map $\exp:\g \rightarrow G$ induces isomorphisms of schemes 
$$\g_{\hat{0}}\xs G_{\hat{1}} \quad \text{and} \quad \widetilde{\g}_{\hat{0}} \xs \widetilde{G}_{\hat{1}}.$$ 
Hence we can identify $\FN_{\widetilde{G}}(\sB)=\FN_{\widetilde{\g}}(\sB)$, and it follows that 
$$\Coh_{\sB\times \{(1,\zeta^{2\lambda})\}}(\sV)
\simeq 
\Coh_{\sB}(\widetilde{\g}), \quad \forall \lambda \in \Lambda.$$  
Therefore, by Theorems \ref{thm 3.3} and \ref{thm 3.4}, we have an equivalence 
\begin{equation}\label{equ 3.3}
\Db\Coh_{\sB}(\widetilde{\g})\xs \Db(\fU_\zeta \mod_{(1,\chi_0)}). 
\end{equation}

We set $\A_0$ be the specialization of $\A$ at $0\in \g$. 

\begin{prop}[{\cite{Tan21}}]\label{prop 3.5} 
The equivalence (\ref{equ 3.3}) is exact for the exotic t-structure on the LHS and the trivial t-structure on the RHS. 
Therefore, there is an equivalence 
\begin{equation}\label{equ 3.1} 
\A\mod_0\xs \fU_\zeta \mod_{(1,\chi_0)} .
\end{equation} 
Moreover, this equivalence intertwines the reflection functors $\Theta^c_s$ and $\Theta^r_s$ ($s\in \I_\af$), and there is a commutative diagram 
$$\begin{tikzcd} 
\A\mod_0\arrow[r,"\simeq"]\arrow[d,"\pi_*"] 
& \fU_\zeta\mod_{(1,\chi_0)} \arrow[d,"\rT^{-\rho}_0"]\\ 
\Coh_0(\g) \arrow[r,"\simeq"] 
& \fU_\zeta\mod_{(1,\chi_{-\rho})}. 
\end{tikzcd}$$ 
\end{prop} 
\noindent 
By construction, the equivalence (\ref{equ 3.1}) is linear with respect to 
\begin{equation}\label{equ 3.9}
\k[\g_{\hat{0}}\times_{(\t/W)_{\hat{0}}} \t_{\hat{0}}]\xs 
\k[G^*_{\hat{1}}\times_{(T/W)_{\hat{1}}} \big(T/(W,\bullet)\big)_{\widehat{\chi_0}}].
\end{equation} 
Therefore, (\ref{equ 3.1}) can be specialized to an equivalence 
$$\A_0\mod \xs u_\zeta\mod_{\chi_0}.$$

\subsubsection{Graded version} 
The constructions above can be made $T$-equivariantly as follows. 
Define $\fU_\zeta \mod^{\Lambda}_{(1,\chi_\lambda)}$ to be the category of $\Lambda$-graded $\fU_\zeta$-modules $M$ such that $M\in \fU_\zeta \mod_{(1,\chi_\lambda)}$, and that for any $K_\mu$ and $m_\nu\in M_\nu$ there exists $n>0$ satisfying 
$$(K_\mu -\zeta^{(\mu,\nu)})^n. m_\nu =0.$$ 
By \cite[\textsection 7.4]{Tan21} there exists a $\Lambda$-grading (in completion sense) on $M_{\chi_{-\rho}}$. 
It yields an equivalence 
$$M_{\chi_{-\rho}}\otimes_{\k[\g]_{\hat{0}}}-:\ 
\Coh_0^T(\g) \xs \fU_\zeta \mod^{\Lambda}_{(1,\chi_{-\rho})}.$$ 
Moreover, the splitting bundles $\fM_{\lambda}$, $\lambda\in \Lambda$ are thus endowed with $\Lambda$-grading (in completion sense) by their constructions in \textsection\ref{subsect 3.2.3}. 
We have the following graded version of Proposition \ref{prop 3.5}. 

\begin{prop}[{\cite{Tan21}}] \label{prop 3.7} 
The $\Lambda$-grading structure on $\fM_{{0}}$ induces an equivalence 
\begin{equation}\label{equ 3.4}
\A\mod^T_0\xs \fU_\zeta \mod^{\Lambda}_{(1,\chi_0)} , 
\end{equation}
intertwining the reflection functors $\Theta^c_s$ and $\Theta^r_s$ ($s\in \I_\af$). 
There is a commutative diagram 
$$\begin{tikzcd} 
\A\mod_0^T\arrow[r,"\simeq"]\arrow[d,"\pi_*"] 
& \fU_\zeta\mod^{\Lambda}_{(1,\chi_0)} \arrow[d,"\rT^{-\rho}_0"]\\ 
\Coh^T_0(\g) \arrow[r,"\simeq"] 
& \fU_\zeta\mod^{\Lambda}_{(1,\chi_{-\rho})}. 
\end{tikzcd}$$ 
\end{prop}
\noindent 
We can identify $u_\zeta\mod^{\Lambda,0}_{\C}$ with the subcategory of modules in $\fU_\zeta \mod^{\Lambda}_{(1,\chi_\lambda)}$ on which $Z_\Fr$ acts trivially. 
Hence (\ref{equ 3.4}) can be specialized to an equivalence 
$$\A_0\mod^T \xs u_\zeta\mod^{\Lambda,0}_{\C}.$$

\section{Equivalence $\fF^b$}\label{sect 4} 
We define the following scheme and algebras  
$$\b_S=\b\times_\t \t_{\hat{0}},\quad \A^b=\A\otimes_{\k[\g]}\k[\b], \quad \text{and} \quad 
\A^b_S=\A^b\otimes_{\k[\b]} \k[\b_S].$$ 
The exponential map induces isomorphisms of schemes $\n \xs N$ and $\b_S\xs B\times_T T_{\hat{1}}$. 
Since the equivalence (\ref{equ 3.4}) is $\k[\g]_{\hat{0}}$-linear, it specializes to an equivalence 
\begin{equation}\label{equ 4.3} 
\fF':\ \A^b\mod^T_0\xs \fU^b_\zeta \mod^{\Lambda}_{(1,\chi_0)}. 
\end{equation} 
In this section, we firstly extend (\ref{equ 4.3}) to an equivalence 
$$\fF^b:\ \A^b_S\mod^T\xs \fU^b_\zeta\mod^{\Lambda,0}_S$$ 
and then investigate some properties of this equivalence. 

\subsection{Equivalence $\fF^b$}\label{subsect 4.1new} 
Note that $\A^b\mod^T_0$ is the full subcategory of $\A^b_S\mod^T$ of the modules set-theoretically supported at $0\in \b_S$. 
We also view $\fU^b_\zeta \mod^{\Lambda}_{(1,\chi_0)}$ as a full subcategory of $\fU^b_\zeta\mod^{\Lambda,0}_S$ as follows. 
Let $M$ be in $\fU^b_\zeta \mod^{\Lambda}_{(1,\chi_0)}$, then by definition the operator $\zeta^{-(\mu,\lambda)}K_\mu$ on $M_\lambda$ is unipotent for any $\mu,\lambda \in \Lambda$. 
Hence the assignment 
$$K_\mu\mapsto \bigoplus\limits_{\lambda} \zeta^{-(\mu,\lambda)}K_\mu, \quad \forall \mu\in \Lambda $$ 
induces an $S$-module structure on $M=\bigoplus_\lambda M_\lambda$, making $M$ as a module in $\fU^b_\zeta\mod^{\Lambda,0}_S$. 

Since the reflection functors $\Theta^c_s$ ($s\in \I_\af$) on $\A\mod$ are $G$-equivariant and $\k[\g]$-linear, they induce reflection functors (still denoted by $\Theta^c_s$) on $\A^b_S\mod^{K}$ by base change along $\b_S\rightarrow \g$, for any subgroup scheme $K$ in $B$. 

\begin{prop}\label{prop 4.1} 
The equivalence $\fF'$ in (\ref{equ 4.3}) extends to an $S$-linear equivalence 
\begin{equation}\label{equ 4.2} 
\fF^b:\ \A^b_S\mod^T\xs \fU^b_\zeta\mod^{\Lambda,0}_S, 
\end{equation}
compatible with the reflection functors $\Theta^c_s$, $\Theta^r_s$ ($s\in \I_\af$) and $\Lambda$-translations on both sides. 
\end{prop}
\begin{proof} 
Let $\m_0$ (resp. $\m_1$) be the maximal ideal of $\k[\b]$ (resp. of $Z^{\leq}_\Fr= \k[B]$) corresponding to $0\in \b$ (resp. $1\in B$). 
Consider the specializations 
$$S_n=S/(\m_1\cap \k[T])^{n+1}S, \qquad \A^b_n=\A^b_S/\m_0^{n+1} \A^b_S \quad \text{and} \quad \fU^b_{\zeta,n}=\fU^b_{\zeta}/\m_1^{n+1} \fU^b_{\zeta}.$$ 
Since $\fF'$ is $\k[\b]_{\hat{0}}$-linear, it yields equivalences
$$\fF':\ \A^b_n\mod^T\xs \fU^b_{\zeta,n}\mod^{\Lambda,0}_{S_n}, 
\quad \forall n\geq 0.$$ 
There are natural functors by specializations 
$$\tau_n: \A^b_S\mod^T \rightarrow \A^b_n\mod^T, 
\quad \tau_n:\fU^b_\zeta\mod^{\Lambda,0}_S \rightarrow \fU^b_{\zeta,n}\mod^{\Lambda,0}_{S_n},$$ 
and $\fF'$ is clearly compatible with $\tau_n$. 
Note that $\fF'(\A^b_n)$ is projective in $\fU^b_{\zeta,n}\mod^{\Lambda,0}_{S_n}$, the family $\{\fF'(\A^b_n)\}_n$ admits a lifting $P^b$ in $\Proj(\fU^b_\zeta\mod^{\Lambda,0}_S)$. 
We choose a family of isomorphisms $\fF'(\A^b_n)\cong \tau_n(P^b)$, $n\geq 0$ that are compatible with specializations. 
To proceed, we need the following lemmas. 
\begin{lem}\label{lem 4.2.1} 
Let $R=\bigoplus_{i\geq 0} R_i$ be a commutative graded ring that is finitely generated over a Noetherian complete local ring $R_0$. 
Let $I$ be a graded ideal of $R$ such that $I_0$ is the maximal ideal of $R_0$. 
Then for any finitely generated graded $R$-module $M$ that is complete over $R_0$, the natural maps 
$$M_i \rightarrow \varprojlim_n \big(M/I^nM\big)_i, \quad i\in \Z,$$ 
are isomorphisms. 
\end{lem}
\begin{proof}[Proof of Lemma \ref{lem 4.2.1}] 
Let $n_0$ be the minimal integer such that $M_{n_0}\neq 0$. 
By degree consideration, we have $(I^nM)_i\subseteq I_0^{n+n_0-i}M_i$ for any $n$ large enough. 
Hence 
$$\varprojlim_{n} \big(M/I^nM\big)_i=\varprojlim_{n} M_i/I_0^nM_i =M_i.\qedhere $$ 
\end{proof}

\begin{lem}\label{lem 4.2}
\begin{enumerate}
\item The natural maps 
$$\A^b_{S,\lambda} \rightarrow \varprojlim_{n} \A^b_{n,\lambda}, \quad \lambda\in \Lambda,$$ 
are isomorphisms. 
\item Let $P, M\in \fU^b_\zeta\mod^\Lambda_S$. 
Suppose that $P$ is projective and $M$ admits Verma flags, then the natural map 
\begin{equation}\label{equ 4.00}
\Hom(P,M)\rightarrow \varprojlim_{n} \Hom(\tau_n(P),\tau_n(M)) 
\end{equation}
is an isomorphism. 
\end{enumerate}
\end{lem}
\begin{proof}[Proof of Lemma \ref{lem 4.2}] 
(1) Since $\widetilde{\g}\rightarrow \g$ is a projective morphism to an affine scheme, $\A=\End^\op_{\widetilde{\g}}(\sE)$ is a finitely generated $\k[\g]$-module by Serre's Theorem. 
Hence $\A^b_S$ is a finitely generated $\k[\b_S]$-module. 
Note that $\A^b_S$ is moreover $\Lambda$-graded and complete over $S$. 
We view $\k[\b_S]$ as a $\Z_{\geq 0}$-graded ring via $\k[\b_S]_i=\bigoplus_{\mathrm{ht}(\eta)=i} \k[\b_S]_{-\eta}$, and consider the graded ideal $\m_0\k[\b_S]$, then our assertion follows from Lemma \ref{lem 4.2.1}. 

(2) Consider the module $P^\lambda= \fU^b_\zeta\otimes_{\fU^0_\zeta \otimes u^+_\zeta} (S_\lambda \otimes u^+_\zeta)$ in $\fU^b_\zeta\mod^\Lambda_S$, which represents the functor $M\mapsto M_\lambda$. 
Hence $P^\lambda$, $\lambda\in \Lambda$ form a generating family of projective modules in $\fU^b_\zeta\mod^\Lambda_S$. 
With loss of generality, we may assume $P=P^\lambda$. 
Then the RHS of (\ref{equ 4.00}) is $M_\lambda$, and the LHS of (\ref{equ 4.00}) is $\varprojlim_{n} \tau_n(M)_{\lambda}$. 
Since $M$ admits Verma flags, it is free of finite rank over $Z_\Fr^-\otimes S$. 
We view $Z_\Fr^-\otimes S$ as a $\Z_{\geq 0}$-graded ring via $(Z_\Fr^-\otimes S)_{i}=\bigoplus_{\mathrm{ht}(\eta)=i} (Z_{\Fr}^-)_{-l\eta}\otimes S$, and consider the graded ideal $\m_1(Z_\Fr^-\otimes S)$, then our assertion follows from Lemma \ref{lem 4.2.1}. 
\end{proof}

There are isomorphisms of $\Lambda$-graded algebras 
\begin{equation}\label{equ 4.0} 
\begin{aligned} 
\bigoplus_\nu \Hom\big(P^b,P^b\otimes \k_{l\nu})&=
\bigoplus_\nu \varprojlim_{n} \Hom\big(\tau_n(P^b),\tau_n(P^b\otimes \k_{l\nu})\big) \\ 
&\cong \bigoplus_\nu \varprojlim_{n} \Hom(\A^b_n,\A^b_n\otimes \k_\nu)\\ 
&=\bigoplus_\nu \varprojlim_{n} \A^b_{n,-\nu} =\bigoplus_\nu \A^b_{S,-\nu}=\A^b_S, 
\end{aligned} 
\end{equation} 
where the first and the second last equalities are by Lemma \ref{lem 4.2}. 
Note that the simple modules of $\fU^b_\zeta\mod^{\Lambda,0}_S$ are contained in $u_\zeta\mod^{\Lambda,0}_\k= \fF'(\A^b_0\mod^T)$, and therefore they admit a surjection from $\bigoplus_\nu \fF'(\A^b_1\otimes \k_\nu)=\bigoplus_\nu \tau_1(P^b\otimes \k_{l\nu})$. 
It implies that $\{P^b\otimes \k_{l\nu}\}_\nu$ forms a generating family of projective modules in $\fU^b_\zeta\mod^{\Lambda,0}_S$. 
Hence there is a Morita equivalence from $\fU^b_\zeta\mod^{\Lambda,0}_S$ to the category of $\Lambda$-graded modules of the algebra $\bigoplus_\nu \Hom\big(P^b,P^b\otimes \k_{l\nu})$. 
Now our desired equivalence (\ref{equ 4.2}) follows from the isomorphism (\ref{equ 4.0}). 

Finally, by Proposition \ref{prop 3.7}, the equivalence (\ref{equ 4.3}) is compatible with reflection functors and $\Lambda$-translations. 
The compatibility for (\ref{equ 4.2}) follows. 
\end{proof}

We set $\A_\n:=\A\otimes_{\k[\g]} \k[\n]$. 
\begin{corollary}
There is an equivalence of abelian categories 
\begin{equation}\label{equ 4}
\A_\n\mod^T\xs \fU^b_\zeta\mod^{\Lambda,0}_\k. 
\end{equation}
\end{corollary}

\subsection{Equivalence $\fG^b$ for Steinberg block}\label{subsect 4.2} 
Recall the equivalence (\ref{equ 3.2}). 
By the $\k[\g]_{\hat{0}}$-linearity it specializes to an equivalence 
$$M^{b}_{\chi_{-\rho}}\otimes_{\k[\b]_{\hat{0}}}-:\ 
\Coh_0^T(\b) \xs \fU^b_\zeta \mod^{\Lambda}_{(1,\chi_{-\rho})},$$ 
where $M^{b}_{\chi_{-\rho}}=M_{\chi_{-\rho}}\otimes_{\k[\g]_{\hat{0}}} \k[\b]_{\hat{0}}$. 
As in Proposition \ref{prop 4.1}, we can extend it to an equivalence 
\begin{equation}\label{equ 4.1} 
\fG^b:\ \Coh^T(\b_S)\xs \fU^b_\zeta\mod^{\Lambda,-\rho}_S. 
\end{equation} 
Let us describe $\fG^b$ more precisely. 
Note that $M^{b}_{\chi_{-\rho}}$ and $M(-\rho)_S\otimes_{\k[\b_S]} \k[\b]_{\hat{0}}$ are both (the unique) indecomposable projective modules of $(\fU^b_\zeta)_{\widehat{(1,\chi_{-\rho})}}$. 
Hence there is an isomorphism of $(\fU^b_\zeta)_{\widehat{(1,\chi_{-\rho})}}$-modules 
\begin{equation}\label{equ 4.4}
M^{b}_{\chi_{-\rho}}\cong M(-\rho)_S\otimes_{\k[\b_S]} \k[\b]_{\hat{0}}. 
\end{equation} 
We can adjust the $\Lambda$-grading on $M_{\chi_{-\rho}}$ such that the isomorphism (\ref{equ 4.4}) respects the $\Lambda$-gradings (in completion sense) on both sides. 
Therefore, the equivalence (\ref{equ 4.1}) is given by 
$$\fG^b=M(-\rho)_S\otimes_{\k[\b_S]}-.$$ 
In particular, we have an isomorphism $\fG^b(\k[\b_S]\otimes \C_\lambda)=M(-\rho+l\lambda)_S$ for any $\lambda\in \Lambda$. 

By (\ref{equ basechange}) the equivalence $\fG^b$ induces an equivalence 
$$\fG^b:\ \Coh^T(\b_R)\xs \fU^b_\zeta\mod^{\Lambda,-\rho}_R$$ 
for any commutative Noetherian $S$-algebra $R$, where $\b_R:=\b_S\times_{\Spec S} \Spec R$. 

The functor $\pi_*: \A\mod \rightarrow \Coh(\g)$ is represented by the $(\k[\g],\A)$-bimodule $\Gamma(\widetilde{\g},\sE)$. 
By $\k[\g]$-linearity and $G$-equivariance, it induces a functor for any subgroup scheme $K$ in $B$, 
$$\pi_*: \A_S^b\mod^K \rightarrow \Coh^K(\b_S)$$ 
which is represented by the $(\k[\b_S],\A_S^b)$-bimodule $\Gamma(\widetilde{\g},\sE)\otimes_{\C[\g]} \C[\b_S]$. 
The functor $\pi_*$ admits a right adjunction 
$$\pi^!:=\Hom_{\k[\b_S]}(\Gamma(\widetilde{\g},\sE)\otimes_{\C[\g]} \C[\b_S],-):\ \Coh^K(\b_S)\rightarrow \A_S^b\mod^K.$$
From Proposition \ref{prop 3.7} we deduce the following 

\begin{prop}\label{prop 4.3}
There are commutative diagrams 
\begin{equation}\label{equ new4.8}
\begin{tikzcd}
\A_S^b\mod^T\arrow[r,"\fF^b","\simeq"']\arrow[d,"\pi_*"] 
& \fU^b_\zeta\mod^{\Lambda,0}_S \arrow[d,"\rT^{-\rho}_0"]\\ 
\Coh^T(\b_S) \arrow[r,"\fG^b","\simeq"'] 
& \fU^b_\zeta\mod^{\Lambda,-\rho}_S, 
\end{tikzcd}
\end{equation}
and 
\begin{equation}\label{equ new4.9}
\begin{tikzcd}
\A_S^b\mod^T\arrow[r,"\fF^b","\simeq"'] 
& \fU^b_\zeta\mod^{\Lambda,0}_S \\ 
\Coh^T(\b_S) \arrow[r,"\fG^b","\simeq"'] \arrow[u,"\pi^!"'] 
& \fU^b_\zeta\mod^{\Lambda,-\rho}_S \arrow[u,"\rT_{-\rho}^0"']. 
\end{tikzcd}
\end{equation}
\end{prop} 

\subsection{Induction functors}\label{subsect 4.3new} 

Recall that by \cite[\textsection 3.1.2]{Situ2} there is a subalgebra $\fU^\hb_\zeta=Z_\Fr^{-}\otimes \fU^0_\zeta\otimes U^+_\zeta$ of $U^\hb_\zeta$, equipped with a Frobenius map 
$$\Fr:\ \fU^\hb_\zeta \rightarrow \k[B\times_T T]\rtimes U\n$$ 
which is compatible with the quantum Frobenius map $U^+_\zeta \rightarrow U\n$, and with the isomorphism $Z^{-}_\Fr\otimes \fU^0_\zeta \xs \k[B\times_T T]$ induced by (\ref{equ 2.0}), where $T\rightarrow T$ is by the $l$-th power. 
The map $\Fr$ is a homomorphism of $\Lambda$-graded algebras, where we view $\k[B\times_T T]\rtimes U\n$ as an $l\Lambda$-graded ring. 
In \textit{loc. cit.} we defined a functor for any deformation ring $R$ of $U^\hb_\zeta$, 
$$\I^\hb_1:\ (\k[B\times_T T]\rtimes U\n)\Mod^{\Lambda}_R \rightarrow U^\hb_\zeta\Mod^{\Lambda}_R,$$ 
$$\text{by}\qquad V \mapsto U^\hb_\zeta \otimes_{\fU^\hb_\zeta} \Fr^*(V\otimes \k_{-\rho}),$$ 
where $\k_{-\rho}$ is a one dimensional representation of $\k[B\times_T T]\rtimes U\n$ whose restriction on $\k[T]=\k\langle K_\lambda\rangle_{\lambda}$ is by the character $\zeta^{-\rho}\in T$. 

When $R$ is a commutative Noetherian $S$-algebra, we will identify 
$$(\k[B\times_T T]\rtimes U\n)\Mod^{l\Lambda}_R=(\k[\b_R]\rtimes U\n)\Mod^{T}.$$ 
Define an induction functor 
$$\ind_c:= (\k[\b_R]\rtimes U\n)\otimes_{\k[\b_R]} -:\ 
\QCoh^T(\b_R)\rightarrow (\k[\b_R]\rtimes U\n)\Mod^{T}.$$ 
Note that $\ind_c M=U\n \otimes M$ as vector spaces for any $\k[\b_R]$-module $M$. 
We also consider the induction 
$$\ind_{r}:=U^\hb_\zeta\otimes_{\fU^b_\zeta}-:\ \fU^b_\zeta\Mod^{\Lambda}_R \rightarrow U^\hb_\zeta\Mod^{\Lambda}_R.$$ 
Note that $\ind_{r}$ preserves the block decomposition in \textsection\ref{subsect 2.2.1}. 

\begin{prop}\label{prop 4.4}
Let $R$ be a commutative Noetherian $S$-algebra. 
There is a commutative diagram 
\begin{equation}\label{equ new4.10}
\begin{tikzcd}
\Coh^T(\b_R) \arrow[r,"\fG^b","\simeq"']\arrow[d,"\ind_c"] 
& \fU^b_\zeta\mod^{\Lambda,-\rho}_R \arrow[d,"\ind_r"]\\ 
(\k[\b_R]\rtimes U\n)\Mod^{T} \arrow[r,"\I^\hb_1"] 
& U^\hb_\zeta\Mod^{\Lambda}_R. 
\end{tikzcd}
\end{equation}
\end{prop}
\begin{proof}
On one hand, since $U\n=U^+_\zeta\otimes_{u^+_\zeta} \k$ as $U^+_\zeta$-modules, we have natural isomorphisms 
\begin{align*}
\I^\hb_1\circ \ind_c &= U^\hb_\zeta\otimes_{\fU^\hb_\zeta} \Fr^*\big(U\n\otimes - \otimes \k_{-\rho}\big) \\ 
&= U^\hb_\zeta\otimes_{Z^-_\Fr\otimes \fU^0_\zeta\otimes u^+_\zeta} (\k_{-\rho}\otimes -).  
\end{align*}
On the other hand, since $M(-\rho)_R=\fU^b_\zeta\otimes_{\fU^0_\zeta\otimes u^+_\zeta}R_{-\rho}$ as $\fU^b_\zeta$-modules, we have natural isomorphisms 
\begin{align*}
\ind_r \circ \fG^b &=
U^\hb_\zeta\otimes_{\fU^b_\zeta}(M(-\rho)_R\otimes_{\k[\b_R]}-)\\ 
&= U^\hb_\zeta\otimes_{Z_\Fr^-\otimes \fU^0_\zeta\otimes u_\zeta^+}(R_{-\rho}\otimes_R -).
\end{align*} 
We thus obtain a natural isomorphism $\I^\hb_1\circ \ind_c=\ind_r \circ \fG^b$. 
\end{proof}

Let $R$ be a commutative Noetherian $S$-algebra. 
Let $\nu\in \Lambda$. 
We denote the truncation functor on $U^\hb_\zeta\mod^{\Lambda}_R$ by the poset ideal $\{\lambda\in \Lambda|\lambda\leq \nu\}$ (see \textsection\ref{subsect 2.3.1}) as 
$$\tau_r^{\leq \nu}:\ U^\hb_\zeta\Mod^{\Lambda}_R \rightarrow U^\hb_\zeta\Mod^{\leq \nu}_R.$$ 

We denote by 
$$\C[\b_R]\Mod^{B,\leq \nu}$$ 
the category of modules in $(\C[\b_R]\rtimes U\n)\Mod^{T}$ whose $T$-weights are $\leq \nu$, which is a subcategory in $\C[\b_R]\Mod^{B}$ (since the $U\n$-action on a module with $T$-weights $\leq \nu$ is locally unipotent). 
By similar argument as in \textsection\ref{subsect 2.3.1}, there is a truncation functor 
$$\tau^{\leq \nu}_c:\ (\C[\b_R]\rtimes U\n)\Mod^T\rightarrow \C[\b_R]\Mod^{B,\leq \nu}$$ 
by taking the maximal quotient whose $T$-weights are $\leq \nu$, which is left adjoint to the natural inclusion. 
We abbreviate $\C[\b_R]\mod^{B,\leq \nu}:=\C[\b_R]\Mod^{B,\leq \nu}\cap \C[\b_R]\mod^{B}$. 

\begin{lem}
There is a natural isomorphism of functors 
\begin{equation}\label{equ new4.11.0}
\I^\hb_1 \circ \tau^{\leq \nu}_c = \tau^{\leq -\rho+l\nu}_r \circ \I^\hb_1. 
\end{equation}
from $(\C[\b_R]\rtimes U\n)\Mod^T$ to $U^\hb_\zeta\Mod^{\leq -\rho+l\nu}_R$. 
\end{lem}
\begin{proof}
By the right exactness of the functors, it is enough to construct natural isomorphism on the projective generators $\C[\b_R]\rtimes U\n\otimes \C_\lambda$, with $\lambda\in \Lambda$. 
We have an isomorphism of $\Lambda$-graded $U^\hb_\zeta\otimes R$-modules 
$$\I^\hb_1(\C[\b_R]\rtimes U\n\otimes \C_\lambda)=U^\hb_\zeta\otimes_{U^{\hb,\geq}_\zeta}(U\n \otimes R_{-\rho+l\lambda}),$$
where $U\n \otimes R_\mu$ is viewed as a $U^{\hb,\geq}_\zeta$-module via the map $U^{+}_\zeta\otimes \fU^0_\zeta \xrightarrow{\Fr\otimes \tau_\mu} U\n\otimes \fU^0_\zeta\rightarrow U\n \otimes R$, for any $\mu\in \Lambda$. 
Hence we have isomorphisms 
\begin{align*}
\I^\hb_1\big(\tau^{\leq \nu}_c(\C[\b_R]\rtimes U\n\otimes \C_\lambda)\big)
&= \I^\hb_1\big(\C[\b_R]\otimes \big(U\n/\bigoplus_{\eta\nleq \nu-\lambda}(U\n)_\eta\big) \otimes \C_\lambda\big)\\ 
&= U^\hb_\zeta\otimes_{U^{\hb,\geq}_\zeta}(U\n/\bigoplus_{\eta\nleq \nu-\lambda}(U\n)_\eta\big) \otimes R_{-\rho+l\lambda}) \\ 
&= \tau^{\leq -\rho+l\nu}_r \big(\I^\hb_1(\C[\b_R]\rtimes U\n\otimes \C_\lambda)\big). \qedhere 
\end{align*}
\end{proof}

\subsection{Functors $\V^b$}\label{subsect 4.2.2} 
We identify the rings 
$$\k[\t\times_{\t/W} \t]_{\hat{0}}=S\otimes_{\k[\t/W]}\k[\t]=S\otimes_{S^W} S.$$ 

Recall that $\A=\End^\op_{\widetilde{\g}}(\sE)$ for some tilting bundle $\sE$ on $\widetilde{\g}$. 
By \cite[Lem 2.5.3]{BM13}, $\sE$ has $\sO_{\widetilde{\g}}$ as a direct summand. 
It follows that $\A$ is an algebra over $\Gamma(\sO_{\widetilde{\g}})=\k[\g\times_{\t/W} \t]$, and that $\A$ has $\k[\g\times_{\t/W} \t]$ as a direct summand as a $\k[\g\times_{\t/W} \t]$-module. 
Moreover, the subalgebra $\k[\g\times_{\t/W} \t]$ is central in $\A$. 
Hence $\A_S^b$ contains a central subalgebra $\k[\b_S \times_{\t/W} \t]$. 
Note that $\k[\b_S \times_{\t/W} \t]^T=\k[\t\times_{\t/W} \t]_{\hat{0}}$. 
Therefore, we have an algebra homomorphism 
\begin{equation}\label{equ 4.6}
\k[\t\times_{\t/W} \t]_{\hat{0}} \rightarrow Z(\A_S^b\Mod^T).
\end{equation}

The functor $\pi_*: \A\mod \rightarrow \Coh(\g)$ factors through a functor (still denoted by $\pi_*$) $\A\mod\rightarrow \Coh(\g\times_{\t/W} \t)$. 
By base change along $\b_S\rightarrow \g$, we obtain a functor 
\begin{equation}\label{equ new4.11}
\V^b_c:\ \A_S^b\Mod^T \rightarrow \QCoh^T(\b_S\times_{\t/W} \t)
\end{equation}
that is represented by the $(\k[\b_S\times_{\t/W} \t],\A^b_S)$-bimodule $\Gamma(\widetilde{\g},\sE)\otimes_{\C[\g]} \C[\b_S]$. 

On the other hand, by (\ref{equ new2.7}) there is an algebra homomorphism 
\begin{equation}\label{equ 4.5} 
\k[\t\times_{\t/W} \t]_{\hat{0}} \rightarrow Z(\fU^b_\zeta\Mod^{\Lambda,0}_S).
\end{equation} 
Hence any module in $\fU^b_\zeta\Mod^{\Lambda,0}_S$ is naturally a module in $\fU^b_\zeta\Mod^{\Lambda,0}_{S\otimes_{\k[\t/W]} \k[\t]}$. 
Consider the following composition 
$$\V^b_r:\ \fU^b_\zeta\Mod^{\Lambda,0}_S \rightarrow 
\fU^b_\zeta\Mod^{\Lambda,0}_{S\otimes_{\k[\t/W]} \k[\t]} \xrightarrow{\rT^{-\rho}_0}
\fU^b_\zeta\Mod^{\Lambda,-\rho}_{S\otimes_{\k[\t/W]} \k[\t]}.$$ 
By the identification (\ref{equ 3.9}) and the compatibility of $\fF'$ and $\fF^b$, the two maps (\ref{equ 4.6}) and (\ref{equ 4.5}) coincide under the equivalence $\fF^b$. 
Therefore, we obtain the following enhancement of (\ref{equ new4.8})  
\begin{equation}\label{equ new4.12}
\begin{tikzcd}
\A_S^b\mod^T\arrow[r,"\fF^b","\simeq"']\arrow[d,"\V^b_c"] 
& \fU^b_\zeta\mod^{\Lambda,0}_S \arrow[d,"\V^b_r"]\\ 
\Coh^T(\b_S\times_{\t/W} \t) \arrow[r,"\fG^b","\simeq"'] 
& \fU^b_\zeta\mod^{\Lambda,-\rho}_{S\otimes_{\k[\t/W]} \k[\t]}. 
\end{tikzcd}
\end{equation}

\subsection{Verma modules}\label{subsect 4.3} 
In this section, we investigate the objects in $\A_S^b\mod^T$ corresponding to the Verma modules under the equivalence $\fF^b$, and then show that they admit natural $B$-equivariant structures. 

We define an algebra structure on $(U^{\leq}_\zeta)^*$ by the comultiplication of $U^{\leq}_\zeta$, and define an action of $U^{\leq}_\zeta$ on $(U^{\leq}_\zeta)^*$ by the right multiplication of $U^{\leq}_\zeta$ on itself. 
Consider the \textit{quantum coordinate ring} of $N^-$, 
$$\k_\zeta[N^-]:=\bigoplus_{\eta\leq 0} (U^-_{\zeta,\eta})^*.$$ 
We embed $\k_\zeta[N^-]$ into $(U^{\leq}_\zeta)^*$ by $\varphi(k u)=\varepsilon(k) \varphi(u)$, for any $u\in U^-_\zeta$ and $k\in U^0_\zeta$ (here $\varepsilon$ is the counit of $U_q$). 
Then $\k_\zeta[N^-]$ is an algebra of $U^{\leq}_\zeta$-modules. 

Set $e=B/B\in \sB$ and let $U_e=N^-B/B$ be the big open cell. 
By \cite[Proof of Prop 3.12]{Tan12}, there is an algebra isomorphism 
$$\Gamma(U_e, \sO_{\sB_\zeta})\simeq \k_\zeta[N^-],$$ 
which is compatible via the quantum Frobenius map with the identification 
$$\Gamma(U_e,\sO_\sB)=\k[N^-]=\bigoplus_{\eta\leq 0} \big((U\n^-)_\eta\big)^*.$$ 
The counit $\epsilon: \k_\zeta[N^-]\rightarrow \k$ by evaluating at $1\in U^{\leq}_\zeta$ defines an $\sO_{\sB_\zeta}$-module $\sO_\epsilon$ supported on $e$. 
As $\sO_{\sB}$-modules, $\sO_\epsilon=\sO_e$. 

Let $(\fU_\zeta^0)_{-2\rho}$ be the rank one $\fU_\zeta^0$-module given by the automorphism $\tau_{-2\rho}: \fU_\zeta^0\rightarrow \fU_\zeta^0$, and view it as a $\fU^{\geq}_\zeta$-module via the projection $\fU^{\geq}_\zeta\rightarrow \fU_\zeta^0$. 
Define the \textit{universal Verma module} $\M=\fU_\zeta\otimes_{\fU^{\geq}_\zeta} (\fU_\zeta^0)_{-2\rho}$.


\begin{lem}\label{lem 4.6} 
$R\Gamma(\widetilde{\sD}_\zeta\otimes_{\sO_{\sB_\zeta}} \sO_{\epsilon})=\M$ as $\fU_\zeta$-modules. 
\end{lem}
\begin{proof} 
Note that $\widetilde{\sD}_\zeta\otimes_{\sO_{\sB_\zeta}} \sO_{\epsilon}$ is supported on $e\in \sB$. 
Since $U_e$ is affine, we have 
$$R\Gamma(\sB, \widetilde{\sD}_\zeta\otimes_{\sO_{\sB_\zeta}} \sO_{\epsilon})=\Gamma(U_e, \widetilde{\sD}_\zeta\otimes_{\sO_{\sB_\zeta}} \sO_{\epsilon})= \Gamma(U_e,\widetilde{\sD}_\zeta)\otimes_{\k_\zeta[N^-]} \k.$$ 
We define an algebra structure on $\k_\zeta[N^-]\otimes \fU^{\leq}_\zeta$ by 
$$u \otimes \varphi = u^{(1)}.\varphi \otimes u^{(2)}, \quad \forall u\in \fU^{\leq}_\zeta,\ \forall \varphi\in \k_\zeta[N^-].$$ 
There is an algebra homomorphism 
\begin{equation}\label{equ 4.16}
\big( \k_\zeta[N^-]\otimes \fU^{\leq}_\zeta \big)\otimes_{\k[N^-\times B]} \Gamma(U_e,\sO_\sV)\rightarrow \Gamma(U_e, \widetilde{\sD}_\zeta).
\end{equation} 

Under the projection $\sV\rightarrow G^*\times_T T$, the fiber of $e$ in $\sV$ is isomorphic to $B\times_T T$, where $B\rightarrow T$ is given by $nt\mapsto t^2$, $\forall n\in N$, $\forall t\in T$, and $T\rightarrow T$ is by taking $l$-th power. 
We write 
$$\k[B\times_T T]=Z^{\leq}_\Fr\otimes_{\k[T]} \k[T],$$ 
where $\k[T]\rightarrow Z^{\leq}_\Fr$ is given by $k_\mu\mapsto K^{2l}_\mu$, and $\k[T]\rightarrow \k[T]$ is given by $k_\mu\mapsto k^{l}_\mu$ (here we use $\{k_\mu\}_{\mu\in \Lambda}$ as the canonical basis of $\k[T]$). 
By \cite[Lem 6.3]{Tan12}, the map (\ref{equ 4.16}) induces an isomorphism of $\fU^-_\zeta\otimes Z^0_\Fr$-modules 
\begin{equation}\label{equ 4.17}
\fU^-_\zeta\otimes Z^0_\Fr\otimes_{\k[T]} \k[T] =\fU^-_\zeta\otimes_{\k[N]} \k[B\times_T T]\xs \Gamma(U_e, \widetilde{\sD}_\zeta)\otimes_{\k_\zeta[N^-]} \k.
\end{equation} 

We study the action of $\fU^0_\zeta$ on the LHS of (\ref{equ 4.17}). 
Fix basis $\{x_r\}_r$ in $U^+_\zeta$ and $\{y_r\}_r$ in $\fU^-_\zeta$ which are dual to each other under the canonical pairing between $U^+_\zeta$ and $\fU^-_\zeta$. 
By \cite[Prop 4.13]{Tan12}
\footnote{Because of our different conventions on the flag variety and the comultiplications of the quantum group, the formulas in \textit{loc. cit.} are adapted after exchanging $E_i\leftrightarrow F_i$ and $K_\lambda\leftrightarrow K^{-1}_\lambda$.}
, on the LHS of (\ref{equ 4.16}), there is an equality for any dominant $\lambda\in \Lambda$, 
$$K^{-2}_{\lambda}=\sum_r \varphi_r \otimes y_r \otimes k_{\lambda}, \quad 
\varphi_r:\ u\in U^-_\zeta \mapsto \langle1^*_\lambda, \mathrm{S}(ux_r).1_\lambda\rangle,$$ 
where $1_\lambda$ is a highest vector of $\bW(\lambda)$ (the coWeyl module of $U_\zeta$ of highest weight $\lambda$), and $1^*_\lambda\in \bW(\lambda)^*$ is given by $\langle 1^*_\lambda, 1_\lambda\rangle=1$ and $\langle 1^*_\lambda, \bW(\lambda)_{\mu}\rangle=0$ if $\mu\neq \lambda$. 
Modulo the augmentation ideal of $\k_\zeta[N^-]$ from the right, we have 
\begin{equation}\label{equ 4.50} 
\begin{aligned} 
K^{-2}_{\lambda}&=\sum_r y_r^{(2)} \cdot \mathrm{S}^{-1}(y_r^{(1)}).\varphi_r \otimes k_{\lambda}\\ 
&\equiv \sum_r \langle 1^*_\lambda, \mathrm{S}(x_r) y_r^{(1)}. 1_\lambda\rangle \cdot y_r^{(2)}\cdot k_{\lambda}\\ 
&= \sum_r \langle 1^*_\lambda, \mathrm{S}(x_r) y_r. 1_\lambda\rangle \cdot k_{\lambda}\\ 
&= \zeta^{4(\rho,\lambda)} k_{\lambda} \ , 
\end{aligned} 
\end{equation}
where the third equality is because $\Delta(y_r)=y_r\otimes 1+ \sum_{\text{deg}(y_r^{(2)})<0} y_r^{(1)}\otimes y_r^{(2)}$, and the last equality is because the \textit{quantum Casimir element} $\sum_r \mathrm{S}(x_r) y_r$ acts on $1_\lambda\in \bW(\lambda)$ by scalar $\zeta^{4(\rho,\lambda)}$. 
Finally, we identify 
$$Z^0_\Fr\otimes_{\k[T]} \k[T]=\fU^0_\zeta, \quad \text{by}\quad K^{l}_\mu\otimes k_\lambda \mapsto K_{l\mu+2\lambda}, \quad \forall \mu,\lambda\in \Lambda.$$ 
Then we have $\M= \Gamma(U_e, \widetilde{\sD}_\zeta)\otimes_{\k_\zeta[N^-]} \k$ as $\fU^\leq_\zeta$-modules. 
It is in fact an isomorphism of modules of $\fU_\zeta$, since the two modules are $\Lambda$-graded. 
\end{proof}

Recall the tilting bundle $\sE$ in $\widetilde{\g}$. 
Let $\sE^{\vee,b}$ be the restriction of the dual bundle $\sE^{\vee}$ on $\{e\}\times \b$ in $\widetilde{\g}$, and set $\sE^{\vee,b}_S=\sE^{\vee,b}\otimes_{\k[\t]} S$, which is an $\A^b_S$-module. 

\begin{prop}\label{prop 4.7}
We have $\fF^b(\sE^{\vee,b}_S)=M(-2\rho)_S$. 
\end{prop}
\begin{proof} 
As a sheave on $\sV$, $\widetilde{\sD}_\zeta\otimes_{\sO_{\sB_\zeta}} \sO_{\epsilon}$ is supported on the fiber of $e\in \sB$ in $\sV$, which is identified with $B\times_T T$ in the proof of the lemma above. 
We denote by $e_n$ the $n$-th infinitesimal neighborhood of $0$ in $\b$. 
We denote by $e^{\lambda}_n$ the $n$-th infinitesimal neighborhood of $(1,\zeta^{2\lambda})$ in $B\times_T T$, and view it as a closed subscheme in $\sV$. 
Consider the subsheaf of algebras $\sO_{\sB_\zeta}\otimes_{\sO_\sB} \sO_\sV$ in $\widetilde{\sD}_\zeta$. 
The $\sO_{\sB_\zeta}$-module $\sO_\epsilon$ and the $\sO_\sV$-module $\sO_{e^{\lambda}_n}$ define a sheaf of module $\sO_{(\epsilon,e^{\lambda}_n)}$ of $\sO_{\sB_\zeta}\otimes_{\sO_\sB} \sO_\sV$. 
We abbreviate $\widetilde{\sD}_\zeta|_{(\epsilon,e^{\lambda}_n)}
=\widetilde{\sD}_\zeta \otimes_{(\sO_{\sB_\zeta}\otimes_{\sO_\sB} \sO_\sV)} 
\sO_{(\epsilon,e^{\lambda}_n)}$. 
By Lemma \ref{lem 4.6} and formula (\ref{equ 4.50}), we compute that 
\begin{equation}\label{equ 4.10} 
R\Gamma(\widetilde{\sD}_\zeta|_{(\epsilon,e^{\lambda}_n)})=\M/\sI_{e^{\lambda}_n} \M= M(-\lambda-2\rho)_S/\m_1^n M(-\lambda-2\rho)_S, 
\end{equation}
as $\fU_\zeta$-modules, where $\sI_{e^{\lambda}_n}$ is the ideal in $\k[B\times_T T]$ associated to $e^{\lambda}_n$, and $\m_1$ is the maximal ideal of $1\in B$ in $\k[B]$. 

In particular we have $R\Gamma(\widetilde{\sD}_\zeta|_{(\epsilon,e^{0}_n)})=M(-2\rho)_S/\m_1^n M(-2\rho)_S$ as $\fU_\zeta$-modules. 
Consider the equivalence 
$$R\sHom_{\widetilde{\sD}_\zeta|_{\FN(\sB)}}(\fM_{0},-): \Db(\widetilde{\sD}_\zeta\mod_{(1,1)})\xs \Db\Coh_{\sB}(\widetilde{\g}).$$ 
We claim that it sends $\widetilde{\sD}_\zeta|_{(\epsilon,e^0_n)}$ to $\sO_{e_n}$. 
Indeed, we have 
\begin{align*}
R\sHom_{\widetilde{\sD}_\zeta|_{\FN(\sB)}}
(\fM_{0},\widetilde{\sD}_\zeta|_{(\epsilon,e^0_n)})
&=\fM_{0}^{\vee}\otimes_{\widetilde{\sD}_\zeta|_{\FN(\sB)}}
\widetilde{\sD}_\zeta|_{(\epsilon,e^0_n)} \\ 
&=\fM_{0}^{\vee}\otimes_{(\sO_{\sB_\zeta}\otimes_{\sO_\sB} \sO_\sV)} 
\sO_{(\epsilon,e^0_n)}. 
\end{align*} 
By construction that $\fM_{0}=\sO^{\rho}_{\sB_\zeta} \otimes_{\sO_{\sB_\zeta}} \fM_{{-\rho}}$ (see \textsection \ref{subsect 3.2.3}), we obtain that 
$$\fM_{0}^{\vee}\otimes_{(\sO_{\sB_\zeta}\otimes_{\sO_\sB} \sO_\sV)} 
\sO_{(\epsilon,e^0_n)} 
=\fM_{{-\rho}}^{\vee}\otimes_{(\sO_{\sB_\zeta}\otimes_{\sO_\sB} \sO_\sV)} \sO_{(\epsilon,e^{-\rho}_n)}$$ 
as sheaves on $\sB$. 
Similarly we have 
\begin{align*}
\fM_{{-\rho}}^{\vee}\otimes_{(\sO_{\sB_\zeta}\otimes_{\sO_\sB} \sO_\sV)} \sO_{(\epsilon,e^{-\rho}_n)}
&= R\sHom_{\widetilde{\sD}_\zeta|_{\FN(\sB)}}
(\fM_{{-\rho}},\widetilde{\sD}_\zeta|_{(\epsilon,e^{-\rho}_n)}) \\ 
&= \sHom_{\widetilde{\sD}_\zeta|_{\FN(\sB)}}(\fM_{{-\rho}}|_{e^{-\rho}_n} ,\widetilde{\sD}_\zeta|_{(\epsilon,e^{-\rho}_n)}). 
\end{align*} 
By (\ref{equ 4.10}), we have $R\Gamma(\widetilde{\sD}_\zeta|_{(\epsilon,e^{-\rho}_n)})=M(-\rho)_S/\m_1^n M(-\rho)_S$. 
Using the isomorphisms (\ref{equ 3.0}) and (\ref{equ 4.4}), we have 
\begin{equation}\label{equ 4.11} 
\sHom_{\widetilde{\sD}_\zeta|_{\FN(\sB)}}(\fM_{{-\rho}}|_{e^{-\rho}_n} ,\widetilde{\sD}_\zeta|_{(\epsilon,e^{-\rho}_n)}) 
=\End_{(\fU_\zeta)_{\widehat{1,\chi_{-\rho}}}} \big(M(-\rho)_S/\m_1^n M(-\rho)_S\big).
\end{equation}
By the equivalence (\ref{equ 4.1}), the RHS of (\ref{equ 4.11}) is equal to $\sO_{e_n}$ (note that we identify the formal neighborhood of $0$ in $\b_S$ with the formal neighborhood of $(1,1)$ in $B\times_T T$). 

Finally, under the equivalence (\ref{equ 3.5}), the sheaf $\sO_{e_n}$ corresponds to the $\A$-module $\sE^\vee\otimes_{\sO_{\widetilde{\g}}} \sO_{e_n}=\sE^{\vee,b}_S/\m^n_{0} \sE^{\vee,b}_S$, where $\m_0$ is the maximal ideal of $0\in \b$ in $\k[\b]$. 
Hence we have 
$$\fF^b(\sE^{\vee,b}_S/\m_0^n \sE^{\vee,b}_S)=M(-2\rho)_S/\m_1^n M(-2\rho)_S, \quad \forall n\geq 0.$$ 
Taking the limits, we obtain that $\fF^b(\sE^{\vee,b}_S)=M(-2\rho)_S$. 
\end{proof}

\begin{definition} 
We define the \textit{Verma objects} in $\A_S^b\mod^T$ to be 
$$\sM_x=(\fF^{b})^{-1}(M(x\bullet_l 0)_S), \quad x\in W_{\ex}.$$ 
An object in $\A_S^b\Mod^T$ admits \textit{Verma flags} if it admits a filtration \textit{of finite length} with composition factors given by Verma objects. 
\end{definition}

For any $w\in W$, we choose a reduced expression $w=s_{1}\cdots s_{n}$ and set $\T^c_w=\T^c_{s_{1}}\cdots \T^c_{s_{n}}$. 

\begin{corollary}\label{cor 4.9} 
For any $w\in W$ and $\nu\in \Lambda$, we have 
\begin{equation}\label{equ 4.13} 
\sM_{t(\nu)w_0w^{-1}}=\T^c_w(\sE^{\vee,b}_S)\otimes \k_\nu, \quad \forall w\in W,\ \forall \nu\in \Lambda. 
\end{equation}
As a consequence, any Verma object in $\A_S^b\mod^T$ admits a natural $B$-equivariant structure. 
\end{corollary} 
\begin{proof} 
For any $w\in W$ and $s\in \I$ such that $ws<w$ in the Bruhat order, there is a short exact sequence (see e.g. \cite[Thm 7.14(a)]{Hum08}) 
$$0\rightarrow M(ws\bullet 0)_S \rightarrow \Theta^r_s(M(w\bullet 0))_S \xrightarrow{\text{counit}} M(w\bullet 0)_S \rightarrow 0,$$ 
which shows that $M(ws\bullet 0)_S=\T^r_s(M(w\bullet 0)_S)$. 
By Proposition \ref{prop 4.1} that the equivalence $\fF^b$ intertwines the reflection functors on both sides, we have 
\begin{equation}\label{equ 4.131} 
\sM_{ws}=\T^c_s(\sM_w), \quad \text{if}\quad ws<w . 
\end{equation} 
Applying (\ref{equ 4.131}) several times on $\sM_{w_0}=\sE^{\vee,b}_S$ (Proposition \ref{prop 4.7}) and using the $\Lambda$-translations, we obtain (\ref{equ 4.13}). 
Finally, since $\sE^{\vee,b}_S$ is naturally endowed with a $B$-equivariant structure and $\T^c_w$ is $B$-equivariant, any Verma object is $B$-equivariant. 
\end{proof}

\subsection{Fully-faithfulness for $\V^b$}\label{subsect 4.4} 
For any Noetherian commutative $S$-algebra $R$, one can define the functor 
$$\V^b_{r,R}:\ \fU^b_\zeta\Mod^{\Lambda,0}_R \rightarrow \fU^b_\zeta\Mod^{\Lambda,-\rho}_{R\otimes_{\k[\t/W]} \k[\t]}$$ 
as in \textsection \ref{subsect 4.2.2} (note that for $S=R$ we abbreviate $\V^b_r=\V^b_{r,S}$). 
Let $\K=\Frac(S)$ be the fraction field of $S$. 

\begin{lem}\label{lem 4.9} 
\begin{enumerate} 
\item The functor $\V^b_{r,R}$ is faithful on the objects admitting Verma flags. 
\item The functor $\V^b_{r,\K}$ is an equivalence of categories. 
\end{enumerate}
\end{lem}
\begin{proof}
(1) By definition of $\V^b_{r,R}$ we have an inclusion 
$$\Hom(\V^b_{r,R} M_1,\V^b_{r,R} M_2)\hookrightarrow \Hom(\rT^{-\rho}_0 M_1,\rT^{-\rho}_0 M_2), \quad \forall M_i.$$ 
Hence the faithfulness of $\V^b_{r,R}$ (on objects admitting Verma flags) follows from the faithfulness of $\rT^{-\rho}_0$, see \cite[Lem 4.9(2)]{Situ2}. 

(2) There is an isomorphism 
\begin{equation}\label{equ 4.14}
\K\otimes_{\k[\t/W]} \k[\t]\simeq \prod_{w\in W} \K. 
\end{equation} 
It induces a decomposition 
\begin{equation}\label{equ 4.15} 
\fU^b_\zeta\Mod^{\Lambda,-\rho}_{\K \otimes_{\k[\t/W]} \k[\t]}= \bigoplus_{w\in W} \fU^b_\zeta\Mod^{\Lambda,-\rho}_\K.
\end{equation}
Moreover, the isomorphism (\ref{equ 4.14}) together with the natural map induced from (\ref{equ 4.5}) 
$$\K\otimes_{\k[\t/W]} \k[\t]\rightarrow Z(\fU^b_\zeta\Mod^{\Lambda,0}_\K)$$
yields a block decomposition 
\begin{equation}\label{equ 4.7}
\fU^b_\zeta\Mod^{\Lambda,0}_\K= \bigoplus_{w\in W} \fU^b_\zeta\Mod^{\Lambda,0,w}_\K,
\end{equation}
such that a Verma module $M(\lambda)_\K$ is contained in $\fU^b_\zeta\Mod^{\Lambda,0,w}_\K$ if and only if $\lambda-w\bullet 0 \in l\Lambda$. 
Under the decompositions (\ref{equ 4.15}) and (\ref{equ 4.7}), we can write 
$$\V^b_{r,\K}= \bigoplus_{w\in W} \rT^{-\rho}_0 \circ \iota_w,$$ 
where $\iota_w$ is the inclusion from the block labelled by $w$ in (\ref{equ 4.7}). 
The functor $\rT^{-\rho}_0 \circ \iota_w$ is an equivalence with quasi-inverse given by $\pr_w\circ \rT^0_{-\rho}$, where $\pr_w$ is the corresponding block projection. 
Hence $\V^b_{r,\K}$ is an equivalence of categories. 
\end{proof}

\begin{prop}\label{prop 4.11} 
The functors $\V^b_r$ and $\V^b_c$ are fully-faithful on the objects admitting Verma flags. 
\end{prop}
\begin{proof} 
By (\ref{equ new4.12}) we only have to prove the assertion for $\V^b_r$. 
It follows standardly from Lemma \ref{lem 4.9} by the general theory of highest weight categories, see e.g. \cite[Prop 2.18]{RSVV}. 
For readers' convenience, we present the proof as follows. 

We abbreviate $\V=\V^b_{r,R}$ and $M_R=M\otimes_S R$ for any $S$-module $M$ and any $S$-algebra $R$. 
Let $M_1,M_2$ be objects in $\fU^b_\zeta\mod^{\Lambda,0}_S$ admitting Verma flags. 
In particular they are free as $S$-modules. 
Hence the base change yields an embedding $\Hom(M_1,M_2)\subseteq \Hom(M_{1,\K},M_{2,\K})$. 
By \cite[Thm 38]{Mat1970}, we have $M_i=\bigcap\limits_{\hgt(\p)=1} M_{i,S_\p}$ as subspaces in $M_{i,\K}$ (the intersection is taken for height 1 prime ideals $\p$ of $S$). 
Hence we can identify 
\begin{equation}\label{equ 4.8}
\Hom(M_1,M_2)=\bigcap_{\hgt(\p)=1}\Hom(M_{1,S_\p},M_{2,S_\p}) 
\end{equation}
as subspaces in $\Hom(M_{1,\K},M_{2,\K})$. 
Similarly, since $\V M_i=T^{-\rho}_0 M_i$ are free $S$-modules, we have 
\begin{equation}\label{equ 4.9} 
\Hom(\V M_1,\V M_2)=\bigcap_{\hgt(\p)=1}\Hom(\V M_{1,S_\p},\V M_{2,S_\p})
\end{equation}
as subspaces in $\Hom(\V M_{1,\K},\V M_{2,\K})$. 

By (\ref{equ 4.8}) and (\ref{equ 4.9}), it suffices to prove the isomorphism 
$$\varphi: \Hom(M_{1,S_\p},M_{2,S_\p})\xs \Hom(\V M_{1,S_\p},\V M_{2,S_\p}),$$ 
for any height $1$ prime ideal $\p$. 
Note that $S_\p$ is a D.V.R., since it is a regular local ring of dimension 1. 
Lemma \ref{lem 4.9}(1) shows that $\varphi$ is an injection. 
For the surjectivity, by Nakayama's lemma, we need to show the surjectivity for the specialization $\varphi\otimes_{S_\p} \Bbbk(\p)$ at $\Bbbk(\p)=S_\p/\p S_\p$. 
Consider the following commutative diagram 
$$\begin{tikzcd} 
\Hom(M_{1,S_\p},M_{2,S_\p})\otimes_{S_\p} \Bbbk(\p) \arrow[r,"\varphi\otimes_{S_\p} \Bbbk(\p)"] \arrow[d] 
& {\Hom(\V M_{1,S_\p},\V M_{2,S_\p})}\otimes_{S_\p} \Bbbk(\p) \arrow[d]\\ 
\Hom(M_{1,\Bbbk(\p)}, M_{2,\Bbbk(\p)}) \arrow[r,hook] 
& \Hom(\V M_{1,\Bbbk(\p)}, \V M_{2,\Bbbk(\p)}), 
\end{tikzcd}$$ 
where the lower inclusion is by Lemma \ref{lem 4.9}(1). 
We claim that the vertical maps are injective. 
Indeed, for any $f\in \Hom(M_{1,S_\p},M_{2,S_\p})$ vanishing at $\Bbbk(\p)$, we have $f(M_{1,S_\p})\subset u M_{2,S_\p}$ (where $u$ is a uniformizer of $S_\p$). 
Since $M_{2,S_\p}$ is free over $S_\p$, there is $f'\in \Hom(M_{1,S_\p},M_{2,S_\p})$ such that $f=uf'$. 
It shows the injectivity for the left vertical map, and the same argument works for the right one. 
It follows that $\varphi\otimes_{S_\p} \Bbbk(\p)$ is an injection. 
To proceed, we need the following lemma. 

\begin{lem}\label{lem 4.13} 
Let $R\rightarrow R'$ be a homomorphism of Noetherian commutative $S$-algebras. 
Suppose that $R'$ is flat over $R$. 
For any $M,M'\in \fU^b_\zeta\mod^{\Lambda}_R$, there is a natural isomorphism 
$$\Hom(M,M')\otimes_R R'\simeq \Hom(M\otimes_R R',M'\otimes_R R').$$ 
\end{lem}
\begin{proof}[Proof of Lemma \ref{lem 4.13}] 
Choose a projective resolution $P_1\rightarrow P_2\rightarrow M\rightarrow 0$, we have a commutative diagram with exact rows 
$$\begin{tikzcd}
0\arrow[r] & \Hom(M,M')_{R'} \arrow[d]\arrow[r] 
& \Hom(P_2,M')_{R'} \arrow[d] \arrow[r] 
& \Hom(P_1,M')_{R'} \arrow[d]\\ 
0\arrow[r] & \Hom(M_{R'},M'_{R'}) \arrow[r] 
& \Hom(P_{2,R'},M'_{R'}) \arrow[r] 
& \Hom(P_{1,R'},M'_{R'}), 
\end{tikzcd}$$ 
where the subscript $R'$ means $-\otimes_R R'$. 
By \cite[Prop 2.4]{Fie03}, the right two vertical maps are isomorphisms, so is the left one. 
\end{proof}
\noindent 
Consider the commutative diagram 
$$\begin{tikzcd}
\Hom(M_{1,S_\p},M_{2,S_\p})\otimes_{S_\p} \K \arrow[r,"\varphi\otimes_{S_\p} \K"] \arrow[d,"\simeq"] 
& {\Hom(\V M_{1,S_\p},\V M_{2,S_\p})}\otimes_{S_\p} \K \arrow[d,"\simeq"]\\ 
\Hom(M_{1,\K}, M_{2,\K}) \arrow[r,"\simeq"] 
& \Hom(\V M_{1,\K}, \V M_{2,\K}). 
\end{tikzcd}$$ 
where the vertical isomorphisms are by Lemma \ref{lem 4.13} applied on the tuple $(M,M',R,R')$ $=(M_{1,S_\p},M_{2,S_\p},S_\p,\K)$ or $(\V M_{1,S_\p}$, $\V M_{2,S_\p}$, $S_\p\otimes_{\k[\t/W]} \k[\t]$, $\K\otimes_{\k[\t/W]} \k[\t])$, and the lower horizontal isomorphism is by Lemma \ref{lem 4.9}(2). 
Hence $\varphi\otimes_{S_\p} \K$ is an isomorphism. 
We conclude that $\varphi$ is an isomorphism. 
\end{proof}

\begin{rmk}\label{rmk 4.12} 
For the non-deformed category $\fU^b_\zeta\Mod^{\Lambda,0}_\k$, the functor $\V^b_{r,\k}$ fails to be fully-faithful on objects admitting Verma flags. 
\end{rmk}

\section{Main result: equivalence $\fF$}\label{sect 5} 

\subsection{Functors $\V$}\label{subsect 5.1} 
We construct the functors $\V_r$ and $\V_c$, and discuss their compatibility with reflection functors associated to $\rI$. 

\subsubsection{Construction of $\V$} 
Recall that in (\ref{equ new2.7}) we construct an algebra homomorphism 
$$\k[\t\times_{\t/W}\t]_{\hat{0}}\rightarrow Z(U^\hb_\zeta\Mod^{\Lambda,0}_S).$$ 
By similar discussions as in \textsection\ref{subsect 4.2.2}, the translation functor $\rT^{-\rho}_0: U^\hb_\zeta\Mod^{\Lambda,0}_S \rightarrow U^\hb_\zeta\Mod^{\Lambda,-\rho}_S$ can be enhanced to a functor 
$$\V_r:\ U^\hb_\zeta\Mod^{\Lambda,0}_S \rightarrow U^\hb_\zeta\Mod^{\Lambda,-\rho}_{S\otimes_{\k[\t/W]}\k[\t]}.$$ 
It restricts to a functor 
$$\V_r: \sO_S^{0} \rightarrow \sO^{-\rho}_{S\otimes_{\k[\t/W]}\k[\t]}.$$ 

\begin{rmk}
In fact, $\V_r$ is equivalent to the quotient functor $\bigoplus\limits_{\lambda\in \Lambda}\Hom(Q(w_0\bullet 0+l\lambda)_S,-)$. 
We will not use this fact below. 
\end{rmk}

Recall the functor $\V^b_c$ defined in (\ref{equ new4.11}). 
By $U\n$-equivariance it upgrades to a functor 
$$\V_c:\ (\A_S^b\rtimes U\n)\Mod^T \rightarrow (\k[\b_S\times_{\t/W} \t]\rtimes U\n)\Mod^T, $$ 
which restricts to a functor 
$$\V_c:\ \A_S^b\Mod^B \rightarrow \QCoh^B(\b_S\times_{\t/W} \t).$$ 
It is automatic that $\V_r$ and $\V_c$ are compatible with $\Lambda$-translations. 

\subsubsection{Compatibility with reflection functors} 
We show that $\V_r$ and $\V_c$ are compatible with reflection functors associated with $s\in \I$. 
The results of this subsection will not be used in the sequel. 

Let $\omega \in \Xi$. 
Consider the composition of functors 
$$\sO^\omega_S \xrightarrow{\rT_{\omega}^0} \sO^0_S \rightarrow \sO^0_{S\otimes_{\k[\t/W]} \k[\t]} \xrightarrow{\rT^{\omega}_0} \sO^\omega_{S\otimes_{\k[\t/W]} \k[\t]},$$ 
where the middle one is induced by the homomorphism $S\otimes_{\k[\t/W]} \k[\t]\rightarrow Z(\sO^0_S)$. 

\begin{lem}[{\cite[Prop B.9]{Situ1}}] 
There is a natural isomorphism 
$$\rT_{0}^\omega\circ \rT_{\omega}^0\simeq -\otimes_{S^{W_\omega}} S$$ 
of functors from $\sO^\omega_S$ to $\sO^\omega_{S\otimes_{S^W} S}$. 
\end{lem} 

For any $s\in \I$, one can define the \textit{Soergel bimodule} 
$$B_s=\big(\k[\t]\otimes_{\k[\t]^s} \k[\t]\big)_{\hat{0}}$$ 
in $\Coh\big((\t\times_{\t/W} \t)_{\hat{0}}\big)$, where $\hat{0}$ represents the completion at $0\in \t/W$. 
Tensoring with $B_s$ defines a functor on $\Coh^B(\b_S\times_{\t/W} \t)$ and on $\sO^{-\rho}_{S\otimes_{\k[\t/W]}\k[\t]}$. 

\begin{prop} 
For any $s\in \I$, there is a natural isomorphism 
$$\V_r\circ \Theta^r_s= (-\otimes_S B_s) \circ \V_r$$ 
of functors from $\sO^0_S$ to $\sO^{-\rho}_{S\otimes_{\k[\t/W]}\k[\t]}$. 
\end{prop}
\begin{proof}
From the lemma above we deduce that 
\begin{align*}
\rT^{-\rho}_0\circ \Theta^r_s &=\rT^{-\rho}_s\circ \rT^s_0\circ \rT^0_s\circ \rT_0^s\\
&=\rT^{-\rho}_s \big(\rT_0^s (-) \otimes_{S^{s}} S \big)\\ 
&=\Theta^r_s(-) \otimes_{S^{s}} S, 
\end{align*}
where for the first and the last equality we use a natural isomorphism $\rT_{0}^{-\rho}=\rT_s^{-\rho}\circ \rT_0^s$. 
\end{proof} 

\begin{prop}\label{prop 5.3} 
For any $s\in \I$, there is a natural isomorphism 
$$\V_c\circ \Theta^c_s = (-\otimes_S B_s) \circ \V_c,$$ 
such that the counit $\Theta^c_s\rightarrow 1$ corresponds to the bimodule homomorphism $B_s\rightarrow S$ by multiplication. 
\end{prop}
\begin{proof} 
Let $P$ be any (standard) parabolic subgroup of $G$, and let $W_P$ be the corresponding parabolic subgroup in $W$. 
Denote by $\pi^P: \widetilde{\g}_P \rightarrow \g$ the natural projection. 
There is a ($G$-equivariant) natural isomorphism 
$$\pi_{P,*}\pi_P^*\simeq -\otimes_{\k[\t/W_P]}\k[\t]  \qquad \text{on}\quad \Db\Coh(\widetilde{\g}_P).$$ 
Hence we have the following ($G$-equivariant) natural isomorphisms 
\begin{align*} 
\pi_*\circ \Theta^c_P &=\pi^P_*(\pi_{P,*}\pi_P^*) \pi_{P,*}\\ 
&\simeq \pi^P_* \circ (-\otimes_{\k[\t/W_P]}\k[\t])\circ \pi_{P,*}\\ 
&=\pi_{*}(-) \otimes_{\k[\t/W_P]}\k[\t]  
\end{align*} 
of functors from $\Db\Coh(\widetilde{\g})$ to $\Db\Coh(\g\times_{\t/W}\t)$. 
Consider the special case when $P=P_s$, and then it implies our desired isomorphism by base change along $\b_S\rightarrow \g$. 
The second assertion is easy to verify. 
\end{proof}

Denote by ${}^wS=S$ the $S\otimes_{\k[\t/W]}\k[\t]$-module given by 
$$S\otimes_{\k[\t/W]}\k[\t]\rightarrow S,\quad f\otimes g\mapsto fw(g), \quad \forall f\in S,\ \forall g\in \k[\t].$$ 
Recall that the Soergel bimodule $B_s$ admits a short exact sequence 
\begin{equation}\label{equ 5.6}
0\rightarrow {}^sS \rightarrow B_s \rightarrow S \rightarrow 0,
\end{equation} 
where the first map is induced by $1\mapsto \alpha_s\otimes 1-1\otimes \alpha_s$ and the second map is by multiplication. 

\begin{lem}\label{lem 5.4} 
For $x=t(\lambda)w_0w \in W_\ex$ with $w\in W$ and $\lambda\in \Lambda$, we have an isomorphism 
$$\V_c(\sM_{x})=\sO_{\b_S}\otimes_S {}^wS \otimes \k_{\lambda}.$$ 
\end{lem} 
\begin{proof} 
Firstly we have 
$$\V_c(\sM_{w_0})=\V_c(\sE^{\vee,b}_S)=\pi_*\sE^{\vee,b}_S=\sO_{\b_S},$$ 
on which the $S\otimes_{\k[\t/W]}\k[\t]$-action factors through the multiplication $S\otimes_{\k[\t/W]}\k[\t] \rightarrow S$. 
By the short exact sequence (\ref{equ 5.6}) and Proposition \ref{prop 5.3}, we have a natural isomorphism $\V_c\circ \T^c_s=\V_c(-)\otimes_S {}^sS$ for any $s\in \I$. 
It follows that 
$$\V_c\circ \T^c_w =\V_c(-)\otimes_S {}^{w^{-1}}S, \quad \forall w\in W.$$ 
Now we use (\ref{equ 4.13}) to obtain our desired formula.  
\end{proof}

\subsection{Truncated categories}\label{subsect 5.2} 
In this subsection, we review the truncated subcategories of $\sO^{0}_S$ studied in \cite[\textsection4.2]{Situ2}, which is a refinement of the truncations discussed in \textsection\ref{subsect 2.3.1}. 
Next, we define the corresponding truncated subcategories in $\A^b_S\mod^B$. 
The advantage of the new truncations is their compatibility with the functors $\V$, see Lemmas~\ref{lem new5.18} and \ref{lem 5.19new}. 

\subsubsection{Truncation in $\sO^{0}_S$}

\begin{definition}\label{defn 5.6}
The \textit{semi-infinite order} $\leq^{\frac{\infty}{2}}$ on $W_\ex$ is the partial order generated by 
$$x\leq^{\frac{\infty}{2}} tx \quad \text{if}\quad x\bullet_l 0\leq tx\bullet_l 0,$$ 
for $x\in W_\ex$ and reflection $t\in \bigcup_{x\in W_\af} x\I_\af x^{-1}$ in $W_\af$. 
\end{definition}

\begin{rmk}
\begin{enumerate}
\item The order $\leq^{\frac{\infty}{2}}$ is invariant under $\Lambda$-translation from the left. 
Under the decomposition $W_\ex=W_\af \rtimes \Lambda/\rQ$, two elements $(x,\gamma)$ and $(x',\gamma')$ (with $x,x'\in W_\af$, $\gamma,\gamma'\in \Lambda/\rQ$) are comparable only if $\gamma=\gamma'$. 
And we have $(x,\gamma)\leq^\semi (x',\gamma)$ if and only if $x\leq^\semi x'$. 

\item The order $\leq^{\frac{\infty}{2}}$ coincides with the order ``$\prec$" defined in \cite{Soe97}. 
More precisely, let $\mathrm{Alc}$ be the set of alcoves associated to the $W_\af$-action on $\R\otimes_\Z \Lambda$, equipped with an order $\prec$ (see \cite[\textsection 4]{Soe97} for details). 
There is a fundamental alcove $A_{\mathrm{fun}}\in \mathrm{Alc}$ given by 
$$A_{\mathrm{fun}}=\{\lambda \in \R\otimes_\Z \Lambda\ |\ 0< \langle \lambda, \check{\alpha} \rangle <1,\ \forall \alpha\in \Phi^+ \}.$$ 
We have a bijection $W_\af\xs \mathrm{Alc}$, $x\mapsto xA_{\mathrm{fun}}$. 
Then the order $\leq^{\frac{\infty}{2}}$ on $W_\af$ corresponds to $\prec$ on $\mathrm{Alc}$. 
In particular, the order $\leq^{\frac{\infty}{2}}$ is independent of $l$. 
\end{enumerate}
\end{rmk}

Under the identification $W_{\ex}\simeq W_\ex\bullet_l 0$, the order $\leq^\semi$ on $W_\ex$ coincides with the order $\uparrow$ introduced in \cite[\textsection 2.4.2]{Situ2} restricted on $W_\ex\bullet_l 0$. 

\begin{lem}[{\cite[Lem 4.6(1)]{Situ2}}]
\label{lem new5.10}
Let $\nu,\mu \in \Lambda$ and $w\in W$. 
We have $t(\mu)w\leq^\semi t(\nu)$ if and only if $\mu\leq \nu$. 
Namely, we have $\{x\in W_\ex|x\leq^\semi t(\nu)\}= \{t(\lambda)w\}_{\lambda\leq \nu, w\in W}$. 
\end{lem}

Let $\nu\in \Lambda$. 
In \cite[\textsection4.2]{Situ2} we introduce a truncated subcategory of $\sO^0_S$ (denoted by $\sO_S^{\uparrow l\nu}$ in the \textit{loc. cit.}) 
$$\sO^{0,\leq^\semi t(\nu)}_S,$$ 
which is the the full subcategory of modules $M$ in $\sO^{0}_S$ that admit a surjection $Q\twoheadrightarrow M$ from a module $Q$ admitting a Verma flag with factors $M(x \bullet_l 0)_S$ with $x\leq^\semi t(\nu)$. 
Since $W_\ex$ is covered by the poset ideals of the form $\{x\in W_\ex|x\leq^\semi t(\nu)\}$, any module in $\sO^{0}_S$ is a finite direct sum of modules in $\sO^{0,\leq^\semi t(\nu)}_S$ for some $\nu$. 

Recall the induction functor 
$$\ind_{r}=U^\hb_\zeta\otimes_{\fU^b_\zeta}-:\ 
\fU^b_\zeta\Mod^{\Lambda,0}_S \rightarrow U^\hb_\zeta\Mod^{\Lambda,0}_S.$$ 

\begin{prop}[{\cite{Situ2}}]
\label{lem new5.17}
\begin{enumerate}
\item There is a truncation functor 
$$\tau_r^{\leq^\semi t(\nu)}:\ U^\hb_\zeta\mod^{\Lambda}_S \rightarrow \sO^{0,\leq^\semi t(\nu)}_S$$ 
by taking the maximal quotient in $\sO^{0,\leq^\semi t(\nu)}_S$, which is left adjoint to the natural inclusion. 
\item The family $\{\tau_r^{\leq^\semi t(\nu)}(\ind_r(P))\}$ with $P$ running over a generating family of projective objects in $\fU^b_\zeta\mod^{\Lambda,0}_S$ forms a generating family of projective objects in $\sO^{0,\leq^\semi t(\nu)}_S$. 
Moreover, any projective object in $\sO^{0,\leq^\semi t(\nu)}_S$ admits a Verma flag, with factors of the form $M(x\bullet_l 0)_S$ with ${x\leq^\semi t(\nu)}$. 
\item The category $\sO^{0,\leq^\semi t(\nu)}_S$ is a Serre subcategory in $\sO_S$. 
\end{enumerate}
\end{prop}
\begin{proof}
Part (1) is \cite[Lem 4.4]{Situ2}, where the truncation functor is denoted by $\tau^{\uparrow l\nu}$. 
Part (2) and (3) are by \cite[Lem 4.5(2)\&(3)]{Situ2}. 
\end{proof}

\begin{lem}\label{lem new5.18} 
The functor $\V_r$ restricts to 
$$\V_r:\ \sO^{0,\leq^\semi t(\nu)}_S\rightarrow \sO^{-\rho,\leq -\rho+l\nu}_{S\otimes_{\k[\t/W]}\k[\t]}.$$ 
There is a natural isomorphism of functors from $U^\hb_\zeta\mod^{\Lambda,0}_S$ to $\sO^{-\rho}_{S\otimes_{\k[\t/W]}\k[\t]}$, 
\begin{equation}\label{equ new5.8}
\V_r\circ \tau_r^{\leq^\semi t(\nu)}=\tau_r^{\leq -\rho+l\nu}\circ \V_r
\end{equation}
\end{lem}
\begin{proof} 
It is an immediate enhancement of \cite[Lem 4.7]{Situ2}. 
\end{proof} 

\subsubsection{Truncation by a subset}\label{subsect 5.2.2} 
We now consider truncations for the category $\A^b_S\mod^B$. 
We start by studying truncations by a subset $\Omega\subset \Lambda$, as \textsection\ref{subsect 2.3.1} for the representation side. 

We extends the equivalence $\fF^b$ to the equivalence of ind-completions of both sides 
$$\fF^b:\ \A^b_S\Mod^T\ \simeq\ \fU^b_\zeta\Mod^{\Lambda,0}_S.$$ 
Define the truncation of $\A^b_S\Mod^T$ by a subset $\Omega\subset \Lambda$ as 
$$\A^b_S\Mod^{T,\Omega} :=(\fF^{b})^{-1}(\fU^b_\zeta\Mod^{\Omega,0}_S),$$ 
which is a Serre subcategory of $\A^b_S\Mod^T$. 
We denote the truncation functor by $\Omega$ on $\A^b_S\Mod^T$ corresponding to the one on $\fU^b_\zeta\Mod^{\Lambda,0}_S$ (introduced in \textsection\ref{subsect 2.3.1}) by 
$$\tau^{T,\Omega}:\ \A^b_S\Mod^T\rightarrow \A^b_S\Mod^{T,\Omega}.$$
We abbreviate $\A^b_S\mod^{T,\Omega}:=\A^b_S\Mod^{T,\Omega}\cap \A^b_S\mod^{T}$. 
The category $\A^b_S\mod^T$ is the union of $\A^b_S\mod^{T,\Omega}$, with $\Omega$ running over all bounded above poset ideals in $\Lambda$. 
We define 
$$(\A^b_S\rtimes U\n) \Mod^{T,\Omega}$$ 
to be the full subcategory of modules in $(\A^b_S\rtimes U\n) \Mod^{T}$ which are contained in $\A^b_S\Mod^{T,\Omega}$. 
It is a Serre subcategory in $(\A^b_S\rtimes U\n) \Mod^{T}$. 
We abbreviate 
$$\A^b_S\Mod^{B,\Omega}:=\A^b_S\Mod^{B}\cap (\A^b_S\rtimes U\n) \Mod^{T,\Omega}.$$ 
We define $(\A^b_S\rtimes U\n)\mod^{T,\Omega}$ and $\A^b_S\mod^{B,\Omega}$ similarly. 
The category $\A^b_S\mod^{B}$ is the union of $\A^b_S\mod^{B,\Omega}$, with $\Omega$ running over all bounded above poset ideals in $\Lambda$. 

\begin{lem}\label{lem new5.5}
If $\Omega\subset \Lambda$ is a 	bounded above poset ideal, then objects in $(\A^b_S\rtimes U\n) \Mod^{T,\Omega}$ (resp. $(\A^b_S\rtimes U\n) \mod^{T,\Omega}$) are contained in $\A^b_S\Mod^{B,\Omega}$ (resp. $\A^b_S\mod^{B,\Omega}$). 
\end{lem}
\begin{proof}
Since $\Omega$ is bounded above, there is a finite subset $\Omega'\subset \Lambda$ such that $\Omega\subset -2\rho+l\Omega'-\sum_{s\in \I} \N \alpha_s$. 
For any $\lambda\in \Lambda$ and $w\in W$, if $t(\lambda)w\bullet_l 0=l\lambda +w\bullet 0\in \Omega$, then $\lambda\in \Omega'-\sum_{s\in \I} \N \alpha_s$. 
Any module in $\fU^b_\zeta\Mod^{\Omega,0}_S$ is a quotient by a module composed by $M(x\bullet_l 0)_S$ for some $x\in W_\ex$ such that $x\bullet_l 0\in \Omega$. 
Hence any module in $\A^b_S\Mod^{T,\Omega}$ is a quotient of a module composed by $\sM_w\otimes \C_\lambda$ with $w\in W$ and $\lambda\in \Omega'-\sum_{s\in \I} \N \alpha_s$. 
Note that $\sM_w$ is a finitely generated $\A^b_S$-module and $\A^b_S$ is finite over $\C[\b_S]$, thus $\sM_w$ is finitely generated over $\C[\b_S]$, and in particular its $T$-weights are contained in $\Omega(w)-\sum_{s\in \I} \N \alpha_s$, for a finite subset $\Omega(w)\subset \Lambda$. 
Hence the $T$-weights of any module in $\A^b_S\Mod^{T,\Omega}$ are contained in $\Omega'+\bigcup_{w\in W} \Omega(w)-\sum_{s\in \I} \N \alpha_s$, which is a bounded above subset in $\Lambda$. 
In particular, the $U\n$-action on any module in $(\A^b_S\rtimes U\n) \Mod^{T,\Omega}$ is locally unipotent, which induces a $B$-action. 
\end{proof}

\begin{lem}\label{lem new5.6} 
Let $\Omega$ be a subset in $\Lambda$. 
There is a truncation functor 
$$\tau^{U\n,\Omega}:\ (\A^b_S\rtimes U\n) \Mod^{T}\rightarrow (\A^b_S\rtimes U\n) \Mod^{T,\Omega}$$
by taking the maximal quotient in $(\A^b_S\rtimes U\n) \Mod^{T,\Omega}$, which is left adjoint to the natural inclusion. 
\end{lem}
\begin{proof}
For any $M\in (\A^b_S\rtimes U\n) \Mod^{T}$, we define $\tau^{U\n,\Omega}(M)$ as the quotient of $M$ by the $\A^b_S\rtimes U\n$-submodule generated by $\ker(M\rightarrow \tau^{T,\Omega}(M))$. 
Then $\tau^{U\n,\Omega}(M)\in (\A^b_S\rtimes U\n) \Mod^{T,\Omega}$ since it is a quotient of $\tau^{T,\Omega}(M)$ as a $T$-equivariant $\A^b_S$-module. 
Any morphism from $M$ to a module in $(\A^b_S\rtimes U\n) \Mod^{T,\Omega}$ necessarily factors through $\tau^{U\n,\Omega}(M)$. 
Hence $\tau^{U\n,\Omega}(M)$ is the maximal quotient of $M$ in $(\A^b_S\rtimes U\n) \Mod^{T,\Omega}$. 
It gives a functor that is left adjoint to the inclusion. 
\end{proof}
\noindent 
The functor $\tau^{U\n,\Omega}$ restricts to a functor 
$$\tau^{B,\Omega}: \A^b_S\Mod^{B}\rightarrow \A^b_S\Mod^{B,\Omega}.$$ 

Consider the functor $\coind_T^B:\Rep(T)\rightarrow \Rep(B)$ that is right adjoint to the forgetful functor $\Rep(B)\rightarrow \Rep(T)$. 
Since $B/T$ is affine, the functor $\coind_T^B$ is exact. 
For any $M\in \A^b_S\Mod^{T}$, the action map $\A^b_S\otimes M\rightarrow M$ induces a homomorphism 
$$\A^b_S\otimes \coind_T^B(M)= \coind_T^B(\A^b_S\otimes M)\rightarrow \coind_T^B(M),$$
where the first equality is by tensor identity (since $\A^b_S$ is a $B$-module). 
It yields a $B$-equivariant $\A^b_S$-module structure on $\coind_T^B(M)$, which defines a functor 
$$\coind_T^B:\ \A^b_S\Mod^{T} \rightarrow \A^b_S\Mod^{B},$$ 
which is exact and is right adjoint to the forgetful functor. 

\begin{lem}\label{lem new5.7}
Let $\Omega$ be a poset ideal in $\Lambda$. 
\begin{enumerate}
\item The functor $\coind_T^B$ sends $\A^b_S\Mod^{T,\Omega}$ to $\A^b_S\Mod^{B,\Omega}$. 
\item There is a commutative diagram 
\begin{equation}\label{equ new5.2}
\begin{tikzcd} 
\A^b_S\Mod^{B} \arrow[r,"\tau^{B,\Omega}"]\arrow[d,"\for"'] 
& \A^b_S\Mod^{B,\Omega} \arrow[d,"\for"]\\ 
\A^b_S\Mod^{T} \arrow[r,"\tau^{T,\Omega}"] 
& \A^b_S\Mod^{T,\Omega}.
\end{tikzcd}
\end{equation}
\end{enumerate}
\end{lem}
\begin{proof}
(1) Since any module in $\A^b_S\Mod^{T,\Omega}$ is a quotient by a module that is composed by $\sM_x$ with $x\in W_\ex$ such that $x\bullet_l 0\in \Omega$, by exactness it is enough to consider $\coind_T^B(\sM_x)$. 
Since $\sM_x$ is $B$-equivariant by Corollary \ref{cor 4.9}, we have an isomorphism $\coind_T^B(\sM_x)=\sM_x\otimes \C[B/T]$, where $B$ acts on $\sM_x\otimes \C[B/T]$ diagonally. 
Hence $\coind_T^B(\sM_x)$ is composed by $\sM_{x}\otimes \C_{\eta}$ with $\eta\leq 0$, which shows that $\coind_T^B(\sM_x)\in \A^b_S\Mod^{B,\Omega}$. 

(2) By (1) we have a commutative diagram 
$$\begin{tikzcd} 
\A^b_S\Mod^{B} & \A^b_S\Mod^{B,\Omega} \arrow[l,hook']  \\ 
\A^b_S\Mod^{T} \arrow[u,"\coind_T^B"']
& \A^b_S\Mod^{T,\Omega} \arrow[l,hook'] \arrow[u,"\coind_T^B"'] ,
\end{tikzcd}$$ 
from which one obtains the diagram (\ref{equ new5.2}) by taking the left adjoint functors. 
\end{proof}

\begin{lem}\label{lem new5.8}
For any object in $\A^b_S\mod^{B}$ admitting a Verma flag in $\A^b_S\mod^{T}$, it also admits a Verma flag in $\A^b_S\mod^{B}$. 
\end{lem}
\begin{proof}
Let $P\in \A^b_S\mod^{B}$ admitting a Verma flag in $\A^b_S\mod^{T}$. 
By (\ref{equ new2.4}) and \eqref{equ 4}, there is a short exact sequence $0\rightarrow M \rightarrow P\rightarrow P'\rightarrow 0$ in $\A^b_S\mod^{T}$, where $M\simeq \sM_x^{\oplus n}$ ($n\geq 1$) as $T$-equivariant $\A^b_S$-modules and $P'$ is composed by $\sM_y$ for some $y\in W_\ex$ such that $y\bullet_l 0\ngeq x\bullet_l 0$. 
By adjunction we have 
$$\Hom_{\A^b_S\mod^{T}}(\sM_x^{\oplus n},P)=\Hom_{\A^b_S\mod^{B}}(\sM_x^{\oplus n},P\otimes \C[B/T]),$$
where we use the fact that $\coind_T^B(P)=P\otimes \C[B/T]$ since $P$ is $B$-equivariant. 
There is a short exact sequence in $\A^b_S\mod^{B}$, 
$$0\rightarrow P\rightarrow P\otimes \C[B/T]\rightarrow P\otimes \C[B/T]_{<0} \rightarrow 0,$$ 
where $\C[B/T]_{<0}$ is the quotient module of $\C[B/T]$ by the zero weight component. 
Since $P\otimes \C[B/T]_{<0}$ is composed by $\sM_y\otimes \C_\eta$ with $\eta<0$ and $y\bullet_l 0\ngtr x\bullet_l 0$, by (\ref{equ new2.3}) and \eqref{equ 4} any $T$-equivariant (hence any $B$-equivariant) $\A^b_S$-homomorphism from $\sM_x^{\oplus n}$ to $P\otimes \C[B/T]$ factors through $M\subset P\subset P\otimes \C[B/T]$. 
Hence any $T$-equivariant $\A^b_S$-homomorphism 
$$\sM_x^{\oplus n}\xs M\hookrightarrow P$$ is $B$-equivariant. 
It shows that $M$ is a $B$-equivariant submodule of $P$ isomorphic to $\sM_x^{\oplus n}$. 
By induction on the length of Verma flag, $P'$ admits a Verma flag in $\A^b_S\mod^{B}$. 
It completes the proof. 
\end{proof}

Consider the induction functor (which should not be confused by the induction $\ind_c$ for $\C[\b]_S\Mod^T$ defined earlier) 
$$\ind_c:=(\A^b_S\rtimes U\n) \otimes_{\A^b_S}- :\ \A^b_S\Mod^T\rightarrow (\A^b_S\rtimes U\n)\Mod^T,$$ 
which is left adjoint to the forgetful functor. 
Note that $\ind_c (M)=U\n \otimes M$ as vector spaces, for any $\A^b_S$-module $M$. 

\begin{lem}\label{lem new5.9}
Let $\Omega$ be a bounded above poset ideal in $\Lambda$. 
The family $\{\tau^{U\n,\Omega}(\ind_c(P))\}$ with $P$ running over a generating family of projective objects in $\A^b_S\mod^T$ forms a generating family of projective objects in $\A^b_S\mod^{B,\Omega}$. 
Moreover, any projective object in $\A^b_S\mod^{B,\Omega}$ admits a Verma flag. 
\end{lem}
\begin{proof}
The module $\tau^{U\n,\Omega}(\ind_c(P))$ is contained in $\A^b_S\mod^{B,\Omega}$ by Lemma \ref{lem new5.5}. 
Since $\tau^{U\n,\Omega}\circ \ind_c$ is left adjoint to the exact forgetful functor from $\A^b_S\mod^{B,\Omega}$ to $\A^b_S\mod^T$, 
it sends projective objects to projective objects. 
For any $M\in \A^b_S\mod^{B,\Omega}$, we choose a surjection $P\twoheadrightarrow M$ from a projective module $P$ in $\A^b_S\mod^T$. 
It induces surjections $\ind_c(P)\rightarrow \tau^{U\n,\Omega}(\ind_c(P))\rightarrow M$. 
Hence we obtain a generating family of projective objects in $\A^b_S\mod^{B,\Omega}$. 

For the second assertion, by Lemma \ref{lem new5.7}(1) the forgetful functor $\A^b_S\Mod^{B,\Omega}\rightarrow \A^b_S\Mod^{T,\Omega}$ admits an exact right adjunction. 
Hence a projective object in $\A^b_S\mod^{B,\Omega}$ is projective in $\A^b_S\mod^{T,\Omega}$. 
Since (by the corresponding statement for $\fU^b_\zeta\mod^{\Omega,0}_S$) any projective object in $\A^b_S\mod^{T,\Omega}$ admits a Verma flag, our assertion follows from Lemma \ref{lem new5.8}. 
\end{proof}

Till the end of the subsection, we will only consider the truncation by $\Omega=\{\lambda\in \Lambda|\lambda\leq l\nu\}$, and we will abbreviate ``$\Omega$" by ``$\leq l\nu$" in the superscripts above. 

\subsubsection{Truncation by $\leq^\semi$} 
We now consider the truncation of $\A^b_S\mod^B$ by semi-infinite order $\leq^\semi$. 
We have the following proposition, whose proof is postponed to \textsection\ref{subsect 5.2.1}. 
\begin{prop}\label{prop 5.4} 
We have 
\begin{equation}\label{equ new5.9}
\Ext^i_{\A^b_S\mod^B}(\sM_{x},\sM_{x'})\neq 0 \quad \text{only if}\quad 
x\leq^{\frac{\infty}{2}} x'.
\end{equation}
\end{prop} 

For a module $Q\in \A^b_S\mod^B$ admitting a Verma flag, by (\ref{equ new5.9}) we can form the quotient 
$$\tau_c^{\leq^\semi t(\nu)}(Q)$$ 
of $Q$ by the submodule that is composed by the Verma factors $\sM_x$ with $x\nleq^\semi t(\nu)$. 
We define
$$\A^b_S\mod^{B,\leq^\semi t(\nu)}$$ 
as the the full subcategory of modules in $\A^b_S\mod^{B}$ that admit a surjection $Q\twoheadrightarrow M$ from a module $Q$ admitting a Verma flag with factors $\sM_x$ with $x\leq^\semi t(\nu)$. 
Since $W_\ex$ is covered by the poset ideals of the form $\{x\in W_\ex|x\leq^\semi t(\nu)\}$, any object in $\A^b_S\mod^{B}$ is a finite direct sum of objects in $\A^b_S\mod^{B,\leq^\semi t(\nu)}$ for some $\nu\in \Lambda$. 

\begin{lem}\label{lem new5.12}
There is a truncation functor 
$$\tau_c^{\leq^\semi t(\nu)}:\ (\A^b_S\rtimes U\n)\mod^{T}\rightarrow \A^b_S\mod^{B,\leq^\semi t(\nu)}$$ 
that is left adjoint to the natural inclusion. 
\end{lem}
\begin{proof}
Since $\A^b_S\mod^{B,\leq^\semi t(\nu)}$ is contained in $\A^b_S\mod^{B,\leq l\nu}$, any morphism from $M\in (\A^b_S\rtimes U\n)\mod^{T}$ to an object in $\A^b_S\mod^{B,\leq^\semi t(\nu)}$ factors through $\tau^{B,\leq l\nu}M$. 
Hence it is enough to define the functor 
\begin{equation}\label{equ new5.5.0}
\tau_c^{\leq^\semi t(\nu)}:\ \A^b_S\mod^{B,\leq l\nu} \rightarrow \A^b_S\mod^{B,\leq^\semi t(\nu)}.
\end{equation}
Let $Q'\in \A^b_S\mod^{B}$ admitting a Verma flag with factors in $\{M(x\bullet_l 0)_S\}_{x\leq^\semi t(\nu)}$, and let $Q'\twoheadrightarrow M'$ be a surjection. 
Let $Q\in \A^b_S\mod^{B,\leq l\nu}$ be a projective object. 
Then any morphism from $Q$ to $M'$ can be lifted to $Q'$. 
By \eqref{equ new5.9} any morphism from $Q$ to $Q'$ factors through $\tau_c^{\leq^\semi t(\nu)}(Q)$. 
Hence any morphism from $Q$ to $M'$ factors through $\tau_c^{\leq^\semi t(\nu)}(Q)$, which then is the maximal quotient of $Q$ in $\A^b_S\mod^{B,\leq^\semi t(\nu)}$. 

In general, let $M\in \A^b_S\mod^{B,\leq l\nu}$ and choose a projective resolution $Q_2\rightarrow Q_1\rightarrow M\rightarrow 0$ in $\A^b_S\mod^{B,\leq l\nu}$. 
Then we set 
$$\tau_c^{\leq^\semi t(\nu)}(M):=\mathrm{coker}\big(\tau_c^{\leq^\semi t(\nu)}(Q_2)\rightarrow \tau_c^{\leq^\semi t(\nu)}(Q_1)\big).$$ 
Then $\tau_c^{\leq^\semi t(\nu)}(M)$ is contained in $\A^b_S\mod^{B,\leq^\semi t(\nu)}$. 
For any $M'\in \A^b_S\mod^{B,\leq^\semi t(\nu)}$, we have a commutative diagram with exact rows 
$$\begin{tikzcd}
0\arrow[r] & \Hom(\tau_c^{\leq^\semi t(\nu)}(M),M') \arrow[d]\arrow[r] 
& \Hom(\tau_c^{\leq^\semi t(\nu)}(Q_1),M') \arrow[d,equal] \arrow[r] 
& \Hom(\tau_c^{\leq^\semi t(\nu)}(Q_2),M') \arrow[d,equal]\\ 
0\arrow[r] & \Hom(M,M') \arrow[r] 
& \Hom(Q_1,M') \arrow[r] 
& \Hom(Q_2,M'). 
\end{tikzcd}$$ 
Hence the left vertical map is an isomorphism, which shows that $\tau_c^{\leq^\semi t(\nu)}(M)$ is the maximal quotient of $M$ in $\A^b_S\mod^{B,\leq^\semi t(\nu)}$. 
It gives the desired functor. 
\end{proof} 

\begin{lem}\label{lem new5.14}
\begin{enumerate}
\item The family $\{\tau_c^{\leq^\semi t(\nu)}(\ind_c(P))\}$ with $P$ running over a generating family of projective objects in $\A^b_S\mod^T$ forms a generating family of projective objects in $\A^b_S\mod^{B,\leq^\semi t(\nu)}$. 
Moreover, any projective object in $\A^b_S\mod^{B,\leq^\semi t(\nu)}$ admits a Verma flag (with factors in $\{\sM_x\}_{x\leq^\semi t(\nu)}$). 
\item The category $\A^b_S\mod^{B,\leq^\semi t(\nu)}$ is a Serre subcategory of $\A^b_S\mod^{B}$. 
\end{enumerate}
\end{lem}
\begin{proof}
(1) The proof of the first assertion is similar to the one of Lemma~\ref{lem new5.9}. 
For the second assertion, for any projective object $P\in \A^b_S\mod^T$, by construction $\tau_c^{\leq^\semi t(\nu)}(\ind_c(P))=\tau_c^{\leq^\semi t(\nu)}\circ \tau^{B,\leq l\nu}(\ind_c(P))$ admits a Verma flag (with factors in $\{\sM_x\}_{x\leq^\semi t(\nu)}$). 
Hence any projective object $Q\in \A^b_S\mod^{B,\leq^\semi t(\nu)}$ is a direct summand of an object admitting a Verma flag. 
Under the equivalence \eqref{equ 4.2}, the image of $Q$ is a direct summand of a module admitting a Verma flag in $\fU^b_\zeta\mod^{\Lambda,0}_S$, which thus admits a Verma flag by a standard argument (see e.g. \cite[Lem 2.5]{Fie03}), and so does $Q$ in $\A^b_S\mod^T$. 
Now the assertion follows from Lemma~\ref{lem new5.8}. 

(2) By definition $\A^b_S\mod^{B,\leq^\semi t(\nu)}$ is closed under taking quotient objects. 
Let $M$ be a subobject of $M'\in \A^b_S\mod^{B,\leq^\semi t(\nu)}$. 
Then the inclusion $M\hookrightarrow M'$ factors through the quotient $\tau_c^{\leq^\semi t(\nu)}(M)$, hence $M=\tau_c^{\leq^\semi t(\nu)}(M)$. 
So $\A^b_S\mod^{B,\leq^\semi t(\nu)}$ is also closed under taking subobjects. 

Now we show that $\A^b_S\mod^{B,\leq^\semi t(\nu)}$ is closed under extension. 
Let $0\rightarrow M_1\rightarrow M\rightarrow M_2\rightarrow 0$ be a short exact sequence in $\A^b_S\mod^{B}$ with $M_1,M_2\in \A^b_S\mod^{B,\leq^\semi t(\nu)}$. 
Then $M\in \A^b_S\mod^{B,\leq l\nu}$, and we can choose a surjection $Q\twoheadrightarrow M$ from a projective object $Q\in \A^b_S\mod^{B,\leq l\nu}$. 
We have short exact sequence $0\rightarrow Q'\rightarrow Q\rightarrow \tau_c^{\leq^\semi t(\nu)}(Q)\rightarrow 0$, where $Q'$ is the subobject of $Q$ composed by Verma factors $\sM_x$ with ${x\nleq^\semi t(\nu)}$. 
Then $\tau_c^{\leq^\semi t(\nu)}(Q')=0$, so $\Hom(Q',M_i)=0$ ($i=1,2$). 
It follows that $\Hom(Q',M)=0$. 
Hence the surjection $Q\twoheadrightarrow M$ factors through $\tau_c^{\leq^\semi t(\nu)}(Q)\twoheadrightarrow M$, which implies that $M\in \A^b_S\mod^{B,\leq^\semi t(\nu)}$. 
\end{proof}

Recall the truncation functor by $\nu\in \Lambda$ (for any commutative Noetherian $S$-algebra $R$) 
$$\tau^{\leq \nu}_c:\ (\C[\b_{R}]\rtimes U\n) \Mod^T\rightarrow \C[\b_{R}]\Mod^{B,\leq \nu}.$$ 

\begin{lem}\label{lem 5.19new}
The functor $\V_c$ restricts to 
$$\V_c:\ \A^b_S\mod^{B,\leq^\semi t(\nu)}\rightarrow \C[\b_S\times_{\t/W} \t]\mod^{B,\leq \nu}.$$ 
We have a natural isomorphism of functors from $(\A^b_S\rtimes U\n)\mod^{T}$ to $\C[\b_S\times_{\t/W} \t]\mod^{B,\leq \nu}$, 
\begin{equation}\label{equ new5.5}
\V_c\circ \tau_c^{\leq^\semi t(\nu)}=\tau_c^{\leq \nu}\circ \V_c. 
\end{equation}
\end{lem}
\begin{proof}
Recall the functor $\pi_*: \A^b_S\mod^T\rightarrow \C[\b_S]\mod^T$ and its right adjunction $\pi^!$. 
They induces functors (still denoted by $\pi_*$ and $\pi^!$) for $U\n$-equivariant or $B$-equivariant modules. 

We show the first assertion. 
By the exactness of $\V_r$, it only needs to show that $\V_r \sM_x$ is contained in $\C[\b_S\times_{\t/W} \t]\mod^{B,\leq \nu}$, for $x\leq^\semi t(\nu)$. 
By forgetting $S\otimes_{\k[\t/W]}\k[\t]$-action to $S$-action and forgetting the $B$-equivariant structure, it is equivalent to show that $\pi_* \sM_x$ is contained in $\C[\b_S]\mod^{T,\leq \nu}$. 
Indeed, by \cite[Lem 3.8(2)]{Situ1}, for any $w\in W$ and $\lambda\in \Lambda$, in $\sO_S$ we have $\rT^{-\rho}_0M(w\bullet 0+l\lambda)_S=M(-\rho+l\lambda)_S$. 
By (\ref{equ new4.8}) we deduce that $\pi_*\sM_{t(\lambda)w}=\C[\b_S]\otimes \C_\lambda$ (as $T$-equivariant $\C[\b_S]$-modules), which proves the claim. 

We then obtain a natural transformation $\tau_c^{\leq \nu}\circ \V_c\rightarrow \V_c\circ \tau_c^{\leq^\semi t(\nu)}$. 
To show that it is an isomorphism, by forgetting $S\otimes_{\k[\t/W]}\k[\t]$-action to $S$-action, it is enough to show the natural isomorphism 
\begin{equation}\label{equ new5.11}
\tau_c^{\leq \nu}\circ \pi_* \rightarrow \pi_* \circ \tau_c^{\leq^\semi t(\nu)}
\end{equation}
By \cite[Lem 3.8(2)]{Situ1} again, the module $\rT_{-\rho}^0 M(-\rho+l\lambda)_S$ is composed by $M(w\bullet 0+l\lambda)_S$ with $w\in W$. 
By (\ref{equ new4.9}) and Lemma \ref{lem new5.8}, the object $\pi^!\C[\b_S]\otimes \C_\lambda$ is composed by $\sM_{t(\lambda)w}$ with $w\in W$, as $T$-equivariant $\A^b_S$-module. 
Any object in $\C[\b_S]\mod^{B,\leq \nu}$ is a quotient of an object composed by $\C[\b_S]\otimes \C_\lambda$ with $\lambda\leq \nu$. 
By the exactness of $\pi^!$ (it is exact as a functor for $T$-equivariant modules, by (\ref{equ new4.9}) again), we deduce that $\pi^!$ sends $\C[\b_S]\mod^{B,\leq \nu}$ to $\A^b_S\mod^{B,\leq^\semi t(\nu)}$. 
Hence we have a commutative diagram 
$$\begin{tikzcd}
(\A^b_S\rtimes U\n)\mod^{T} & \A^b_S\mod^{B,\leq^\semi t(\nu)} \arrow[l,hook']  \\ 
(\C[\b_S]\rtimes U\n)\mod^{T} \arrow[u,"\pi^!"] 
& \C[\b_S]\mod^{B,\leq \nu} \arrow[l,hook'] \arrow[u,"\pi^!"]. 
\end{tikzcd}$$ 
Now (\ref{equ new5.11}) follows by taking the left adjunctions. 
\end{proof}

\subsubsection{Proof of Proposition \ref{prop 5.4}}\label{subsect 5.2.1} 
We translate our problem to a problem of computing the extensions of some constructible sheaves on the affine flag variety, using Bezrukavnikov's equivalence \cite{Bezru16}. 

Let $\A_{(\t/W)_{\hat{0}}}$ be the completion of $\A$ at $0\in \t/W$. 
Since the restriction of $\pi^* \A_{(\t/W)_{\hat{0}}}$ on $\{e\}\times \b\hookrightarrow \widetilde{\g}$ is $\A^b_S$, we can identify 
\begin{equation}\label{equ 5.3} 
\A^b_S\mod^B=\Coh^G(\pi^* \A_{(\t/W)_{\hat{0}}})  
\end{equation}
by induction from $B$ to $G$. 
The $\A^b_S$-module $\sM_{w_0}=\sE^{\vee,b}_S$ corresponds to $\sE^{\vee}_{\hat{0}}:=\sE^{\vee}\otimes_{\k[\t/W]} \k[\t/W]_{\hat{0}}$ under (\ref{equ 5.3}). 
Let $\St=\widetilde{\g}\times_{\g} \widetilde{\g}$ be the Steinberg variety, and let $\widehat{\mathrm{St}}$ be its completion at $0\in \t/W$. 
Let $\pi_2: \widehat{\mathrm{St}}\rightarrow \widetilde{\g}\times_{\t/W} (\t/W)_{\hat{0}}$ be the projection to the second factor. 
There is an equivalence 
$$\pi_2^*\sE_{{\hat{0}}}\otimes^L_{\pi^* \A_{(\t/W)_{\hat{0}}}}-:\ 
\Db\Coh^G(\pi^* \A_{(\t/W)_{\hat{0}}}) \xs \Db\Coh^G(\widehat{\mathrm{St}}), $$ 
sending $\sE^{\vee}_{\hat{0}}$ to the structure sheaf $\sO_{\widetilde{\g}_{\Delta}}$ of the diagonal $\widetilde{\g}\times_{\t/W} (\t/W)_{\hat{0}}$ in $\widehat{\mathrm{St}}$. 

Let $\check{G}$ be the Langlands dual group of $G$. 
Let $L\check{G}$ be the loop group of $\check{G}$. 
Let $\check{I}$ be the Iwahori subgroup and let $\check{I}_0$ be its pro-unipotent radical. 
The enhanced affine flag variety is $\widetilde{\Fl}=L\check{G}/\check{I}_0$. 
By \cite{Bezru16}, there is an equivalence of triangulated monoidal categories 
\begin{equation}\label{equ 5.5} 
\Db\Coh^G(\widehat{\mathrm{St}}) \simeq \widehat{D}_{\check{I}_0}(\widetilde{\Fl}), 
\end{equation} 
where $\widehat{D}_{\check{I}_0}(\widetilde{\Fl})$ is a completion of the category of $\check{I}_0$-equivariant monodromic complexes on $\widetilde{\Fl}$. 
There are standard and costandard sheaves $j_{x!}$, $j_{x*}$ in $\widehat{D}_{\check{I}_0}(\widetilde{\Fl})$ associated to the embedding $j_x$ from the Schubert cell labelled by $x\in W_\ex$ to $\widetilde{\Fl}$. 
The equivalence (\ref{equ 5.5}) sends $\sO_{\widetilde{\g}_{\Delta}}$ to $j_{e!}$, and in general $\sO_{\widetilde{\g}_{\Delta}}(w_0\lambda)$
to the \textit{Wakimoto sheaf} $\sJ_{\lambda}$ (see \cite[\textsection 3.3]{Bezru16}), for any $\lambda\in \Lambda$ (In \textit{loc. cit.} the line bundle labelled by $\lambda$ is sent to $\sJ_\lambda$, and here our correspondence is due to the different convention that $\sO_{\sB}(\lambda)$ is semi-ample if and only if $\lambda$ is anti-dominant). 

We view $\widehat{\mathrm{St}}$ as the base change of $\pi:\widetilde{\g}\rightarrow \g$ (as the second factor) by $\widetilde{\g}\times_{\t/W} (\t/W)_{\hat{0}} \rightarrow \g$ (as the first factor), then by Theorem \ref{thm 3.1} there are functors $\T^c_s$, $s\in \I$ on $\Db\Coh^G(\widehat{\mathrm{St}})$. 
By \cite{Bezru16} the equivalence (\ref{equ 5.5}) intertwines the functor $\T^c_s$ with the convolution with $j_{s!}$ from the right. 
Therefore, by (\ref{equ 4.13}) the composed equivalence 
$$\Db(\A^b_S\mod^B)\simeq \widehat{D}_{\check{I}_0}(\widetilde{\Fl})$$ 
sends 
$$\sM_{t(\nu)w} \mapsto \sJ_{w_0\nu}*j_{w_0w!}, \quad \forall w\in W,\ \forall \nu\in \Lambda.$$

\begin{proof}[Proof of Proposition \ref{prop 5.4}] 
We suppose $w_0x=t(\lambda)w$ and $w_0x'=t(\lambda')w'$, where $\lambda,\lambda'\in \Lambda$ and $w,w'\in W$. 
Then we have 
\begin{align*}
\Ext^i(\sM_{x},\sM_{x'}) 
&=\Ext^i(\sJ_\lambda* j_{w!},\sJ_{\lambda'}* j_{w'!})\\ 
&=\Ext^i(\sJ_{\lambda+\mu}* j_{w!}, \sJ_{\lambda'+\mu}* j_{w'!}), \quad \forall \mu\in \Lambda.
\end{align*} 
Let $\mu$ be dominant enough, then we have $\sJ_{\lambda+\mu}=j_{t(\lambda+\mu)*}$ (see \cite[\textsection 3.3]{Bezru16}). 
Since $t(\lambda)$ is maximal in the Bruhat order on $t(\lambda)W$ if $\lambda$ is dominant, we have 
$\sJ_{\lambda+\mu}* j_{w!}=j_{t(\lambda+\mu)*}*j_{w!}=j_{t(\lambda+\mu)w*}$. 
Note that 
$$\Ext^i(j_{t(\lambda+\mu)w*},j_{t(\lambda'+\mu)w'*}) = j^*_{t(\lambda'+\mu)w'} j_{t(\lambda+\mu)w*}$$ 
which is nonzero only if $t(\lambda'+\mu)w'\leq t(\lambda+\mu)w$ in the Bruhat order. 
Since $t(\lambda'+\mu)w'$ and $t(\lambda+\mu)w$ translate the fundamental alcove to the dominant alcoves, the Bruhat order on them coincides with the semi-infinite order (see e.g. \cite[Claim 4.14]{Soe97}). 
We deduce that 
$$\Ext^i(\sM_{x},\sM_{x'})\neq 0$$ 
only if $t(\lambda'+\mu)w'\leq^{\frac{\infty}{2}} t(\lambda+\mu)w$, equivalently $t(\lambda')w'\leq^{\frac{\infty}{2}} t(\lambda)w$ and $x \leq^{\frac{\infty}{2}}x'$. 
\end{proof}

\subsection{Proof of the main result}\label{subsect 5.3} 
In this section, we prove our main result. 

\begin{lem}\label{lem 5.7} 
There is a natural isomorphism 
\begin{equation}\label{equ new5.14}
\V_r\circ \ind_{r} \simeq \ind_{r} \circ \V^b_r
\end{equation}
of functor from $\fU^b_\zeta\mod^{\Lambda,0}_S$ to $U^\hb_\zeta\mod^{\Lambda,-\rho}_{S\otimes_{\C[\t/W]}\C[\t]}$, and a natural isomorphism 
\begin{equation}\label{equ new5.12}
\V_c\circ \ind_c \simeq \ind_c \circ \V^b_c
\end{equation}
of functors from $\A^b_S\mod^T$ to $(\C[\b_{S}\times_{\t/W} \t]\rtimes U\n)\mod^T$. 
\end{lem}
\begin{proof} 
For any $M$ in $\fU^b_\zeta\mod^{\Lambda,0}_S$, we have the following natural isomorphisms of $U^\hb_\zeta$-modules 
\begin{align*}
\V_r\circ \ind_{r}(M) 
&= \pr_{-\rho}\big( (U^\hb_\zeta\otimes_{\fU^b_\zeta} M) \otimes V(\rho)^*\big)\\ 
&\simeq \pr_{-\rho}\big( U^\hb_\zeta\otimes_{\fU^b_\zeta} (M \otimes V(\rho)^*)\big)\\ 
&= U^\hb_\zeta\otimes_{\fU^b_\zeta} \big(\pr_{-\rho}(M \otimes V(\rho)^*)\big)= \ind_{r} \circ \V^b_r(M) , 
\end{align*} 
where the second isomorphism is by tensor formula. 
Since the action of $Z_\HC$ on $U^\hb_\zeta\otimes_{\fU^b_\zeta} M$ coincides with the one induced by the $Z_\HC$-action on $M$, the equality $\V_r\circ \ind_{r}(M)=\ind_{r} \circ \V^b_r(M)$ is also an isomorphism of $S\otimes_{\k[\t/W]} \k[\t]$-modules. 
It implies (\ref{equ new5.14}). 

The isomorphism (\ref{equ new5.12}) is straight forward to check. 
\end{proof}

\begin{prop}\label{prop 5.8} 
The functors $\V_r$ and $\V_c$ are fully-faithful on the objects admitting Verma flags. 
\end{prop}
\begin{proof} 
We firstly prove the assertion for $\V_c$ in a special case: let $M\in \A^b_S\mod^{B,\leq^\semi t(\nu)}$ admitting Verma flags, and let $Q=\tau_c^{\leq^\semi t(\nu)}(\ind_c(M'))$ with $M'\in \A^b_S\mod^{T}$ admitting Verma flags. 
We show that the natural map 
$$\Hom(Q,M)\rightarrow \Hom(\V_c Q,\V_c M)$$ 
is an isomorphism. 
To that end, consider the following commutative diagram 
$$\begin{tikzcd}
\Hom(Q,M) \arrow[d,"\V_c"]\arrow[r,"\simeq"] 
& \Hom(\ind_c(M'),M) \arrow[d,"\V_c"] \arrow[r,"\simeq"] 
& \Hom(M',M) \arrow[d,"\V_c^b"]\\ 
\Hom (\V_c Q,\V_c M) \arrow[r,"\simeq"] 
& \Hom (\V_c(\ind_c(M')),\V_c M) \arrow[r,"\simeq"] 
& \Hom(\V_c^b M',\V_c^b M), 
\end{tikzcd}$$ 
where the horizontal isomorphisms follow from the adjunctions for $\tau_c^{\leq^\semi t(\nu)}$, $\tau_c^{\leq \nu}$ and $\ind_c$, and the natural isomorphisms (\ref{equ new5.5}) and (\ref{equ new5.12}). 
It is now equivalent to show that the right vertical map is an isomorphism, which is true by Proposition \ref{prop 4.11}. 

In general, let $M_1,M_2\in \A^b_S\mod^{B,\leq^\semi t(\nu)}$ admitting Verma flags. 
By Lemma \ref{lem new5.14}, there is a resolution $Q_2\rightarrow Q_1\rightarrow M_1\rightarrow 0$ with $Q_i=\tau_c^{\leq^\semi t(\nu)}(\ind_c(P_i))$ for some projective objects $P_i$ in $\A^b_S\mod^{T}$. 
We have a commutative diagram with exact rows 
$$\begin{tikzcd}
0\arrow[r] & \Hom(M_1,M_2) \arrow[d]\arrow[r] 
& \Hom(Q_1,M_2) \arrow[d,equal] \arrow[r] 
& \Hom(Q_2,M_2) \arrow[d,equal]\\ 
0\arrow[r] & \Hom(\V_c M_1,\V_c M_2) \arrow[r] 
& \Hom(\V_c Q_1,\V_c M_2) \arrow[r] 
& \Hom(\V_c Q_2,\V_c M_2), 
\end{tikzcd}$$ 
where the vertical isomorphisms are by our previous result in the special case. 
It shows that $\V_c$ is fully-faithful on $M_1$ and $M_2$. 

The assertion for $\V_r$ can be proved similarly, or more directly as in Proposition \ref{prop 4.11}. 
\end{proof}

Recall the following result 
\begin{thm}[{\cite[Thm 3.6]{Situ2}}] 
The functor 
$$\I^\hb_1:\ (\k[\b_S]\rtimes U\n)\Mod^{T} \rightarrow U^\hb_\zeta\Mod^{\Lambda}_S$$ 
induces an equivalence of abelian categories 
$$\mathfrak{G}:\ \Coh^B(\b_S)\xs \sO^{-\rho}_{S},$$ 
sending $\k[\b_S]\otimes \k_\lambda$ to $M(-\rho+l\lambda)_S$, for any $\lambda\in \Lambda$. 
\end{thm}
\noindent 
By (\ref{equ basechange}), for any commutative Noetherian $S$-algebra $R$, the equivalence $\mathfrak{G}$ induces an equivalence of the categories on both sides after base changed to $R$. 
In particular, we have 
$$\mathfrak{G}:\ \Coh^B(\b_S\times_{\t/W} \t)\xs \sO^{-\rho}_{S\otimes_{\k[\t/W]}\k[\t]}.$$ 

The following statement is our main result. 

\begin{thm}\label{thm 5.13}
There is an equivalence of $S\otimes_{\k[\t/W]}\k[\t]$-linear abelian categories 
$$\fF:\ \Coh^G(\pi^* \A_{(\t/W)_{\hat{0}}}) \xs \sO_S^{0}, $$ 
such that 
$$\fF\big(\T^c_w(\sE^{\vee}_{\hat{0}})\otimes \sO_{\widetilde{\g}}(\nu)\big) = M(w_0w^{-1}\bullet 0+l\nu)_S , \quad \forall w\in W,\ \forall \nu\in \Lambda. $$ 
\end{thm} 
\begin{proof}
Let $P$ be a projective module in $\fU^b_\zeta\mod^{\Lambda,0}_S$, and set $\sP=(\fF^{b})^{-1}(P)$ in $\A^b_S\mod^T$. 
Set $Q=\ind_r(P)$ and $\sQ=\ind_c(\sP)$. 
We have the following isomorphisms by (\ref{equ new5.8}) and (\ref{equ new5.14}), 
$$\V_r(\tau^{\leq^\semi t(\nu)}_r Q)
=\tau^{\leq -\rho+l\nu}_r(\V_r Q)
=\tau^{\leq -\rho+l\nu}_r(\ind_r \V^b_r(P)),$$ 
and similarly by (\ref{equ new5.5}) and (\ref{equ new5.12}) we have 
$$\V_c(\tau^{\leq^\semi t(\nu)}_c \sQ)
=\tau^{\leq \nu}_c (\V_c \sQ)
=\tau^{\leq \nu}_c (\ind_c \V^b_c(\sP)).$$ 
By (\ref{equ new4.11.0}), (\ref{equ new4.10}) and (\ref{equ new4.12}), we have isomorphisms 
\begin{align*}
\mathfrak{G}\big(\tau^{\leq \nu}_c (\ind_c \V^b_c(\sP))\big) 
&=\tau^{\leq -\rho+l\nu}_r\circ \I^\hb_1 \circ \ind_c (\V^b_c(\sP))\\ 
&=\tau^{\leq -\rho+l\nu}_r \circ \ind_r \big(\mathfrak{G}^b (\V^b_c(\sP))\big)\\
&=\tau^{\leq -\rho+l\nu}_r(\ind_r \V^b_r(P)). 
\end{align*}
It follows that $\mathfrak{G}\big(\V_c(\tau^{\leq^\semi t(\nu)}_c \sQ)\big)=\V_r(\tau^{\leq^\semi t(\nu)}_r Q)$. 
By Lemmas \ref{lem new5.17}(1) and \ref{lem new5.14}(1), we obtain identifications 
$$\fG \circ \V_c\big(\Proj(\A^b_S\mod^{B,\leq^\semi t(\nu)})\big)= \V_r\big(\Proj(\sO^{0,\leq^\semi t(\nu)}_S)\big), \quad \forall \nu \in \Lambda ,$$ 
of subcategories in $\sO^{-\rho}_{S\otimes_{\k[\t/W]} \k[\t]}$. 
By Proposition \ref{prop 5.8}, it induces an equivalence of additive categories 
$$\Proj(\A^b_S\mod^{B,\leq^\semi t(\nu)}) \xs \Proj(\sO^{0,\leq^\semi t(\nu)}_S), \quad \forall \nu \in \Lambda,$$ 
which extends to an equivalence of abelian categories 
$$\fF:\ \A^b_S\mod^{B}\xs \sO^{0}_S.$$ 
We then identify $\A^b_S\mod^{B}=\Coh^G(\pi^* \A_{(\t/W)_{\hat{0}}})$ by induction from $B$ to $G$, to get the desired equivalence. 

For the second assertion, we consider the following commutative diagram by construction 
$$\begin{tikzcd} 
\A^b_S\mod^{B} \arrow[r,"\fF","\simeq"']\arrow[d,"\for"'] 
& \sO^{0}_S \arrow[d,"\for"]\\ 
\A^b_S\mod^{T} \arrow[r,"\fF^b","\simeq"'] 
& \fU^b_\zeta\mod^{\Lambda,0}_{S}. 
\end{tikzcd}$$ 
Hence $\fF(\sM_x)$ ($x\in W_\ex$) is isomorphic to $M(x\bullet_l 0)_S$ as a $\Lambda$-graded $\fU^b_\zeta\otimes S$-module. 
It forces that $\fF(\sM_x)= M(x\bullet_l 0)_S$ as $\Lambda$-graded $U^\hb_\zeta\otimes S$-modules. 
Now the assertion follows from (\ref{equ 4.13}). 
\end{proof} 

\begin{rmk}\label{rmk 5.14} 
The assumption that $l$ is a prime power is used in the results in \textsection\ref{subsect 3.2} concerning the localization theorem for $\fU_\zeta$. 
However, it is expected that this assumption is not necessary (see a proof for type $A$ in \cite{Tan22}). 
For Theorem \ref{thm 5.13}, one may remove this assumption by showing that the equivalence class of the category $\sO^0_S$ is independent of $\zeta$ (note that similar independence for small quantum groups appeared in \cite{AJS94}.) 
\end{rmk}

Let $\pi_{_{\widetilde{\sN}}}: \widetilde{\sN}=G\times^B \n \rightarrow \g$ be the Springer resolution. 
We set $\A'$ be the specialization of $\A$ at $0\in \t$, which satisfies an equivalence \cite{BM13} 
$$\Db\Coh(\widetilde{\sN})\ \simeq\ \Db(\A'\mod),$$ 
and therefore is called the \textit{non-commutative Springer resolution}. 
Recall that $\A_0$ is the specialization of $\A$ at $0\in \g$. 
We abbreviate $\sO^{\chi_0}_\C$ as the full subcategory of modules in $\sO_\C$ on which the action of $Z_\HC$ factors through $\chi_0: Z_\HC\rightarrow \C$. 

\begin{corollary}\label{cor 5.14}
There is an equivalences of abelian categories 
\begin{equation}\label{equ 5.7}
\Coh^G(\pi_{_{\widetilde{\sN}}}^* \A)=\A_\n\mod^B \xs \sO_\k^{0}, 
\end{equation}
which by restrictions on $0\in \t$ and $0\in \g$ induces equivalences 
$$\Coh^G(\pi_{_{\widetilde{\sN}}}^* \A') \xs \sO_\k^{{\chi_0}}, \quad \text{and}\quad \A_0\mod^B \xs \sO^{u_\zeta U^+_\zeta,0}_\k.$$ 
\end{corollary} 

\begin{rmk}
Our construction can be applied to obtain a new approach to a similar equivalence for modular analogue of BGG category $\sO$, which is previously established by Losev \cite[Equ (1.3) in Thm 1.7]{Los23}. 
To be more precise, assume moreover that $l$ is prime and set $\F$ be an algebraically closed field of characteristic $l$. 
Let $G_\F$ be a connected semisimple algebraic group over $\F$, with a Borel subgroup $B_\F$ and a Cartan subgroup $T_\F$. 
Let $N_\F$ be the unipotent radical of $B_\F$, and let $\g_\F=\Lie(G_\F)=\n^-_\F \oplus \t_\F \oplus \n_\F$ be the corresponding triangular decomposition with $\t_\F=\Lie(T_\F)$ and $\n_\F=\Lie(N_\F)$. 
The modular analogue of BGG category $\sO$ is the category 
$$\sO_\F:=\g_\F\mod^{B_\F}$$ 
of $B_\F$-equivariant finitely generated $\g_\F$-modules on which the actions of $\b_\F$ by  restriction of $\g_\F$-action and by differential of $B_\F$-action coincide. 
The category $\sO_\F$ can be interpreted as the category $\sO$ (with integral weights) associated to the modular analogue $U^\hb_\F$ of $U^\hb_\zeta$, which is an algebra with a triangular decomposition 
$$U^\hb_\F=U\n^-_\F \otimes U\t_\F \otimes \mathrm{Dist}(N_\F),$$ 
where $\mathrm{Dist}(N_\F)$ is the distribution algebra of $N_\F$. 
Let $\A_\F$ be the non-commutative Grothendieck resolution of $\g^{(1)}_\F$ (the Frobenius twist of $\g_\F$). 
Let $\pi_{\widetilde{\sN}^{(1)}_\F}: \widetilde{\sN}^{(1)}_\F\rightarrow \g^{(1)}_\F$ be the Springer resolution. 
Then there is an equivalence of categories
$$\Coh^{G_\F^{(1)}}(\pi^{*}_{\widetilde{\sN}^{(1)}_\F} \A_\F) \ \simeq\ \sO^0_\F,$$
where $\sO^0_\F$ is the principal block of $\sO_\F$. 
To adapt our construction to obtain the equivalence above, one may replace Tanisaki's localization theorem in \textsection\ref{subsect 3.2} by the result for $U\g_\F$ by Bezrukavnikov--Mirkovi\'{c}--Rumynin \cite{BMR08} and replace Bezrukavnikov's equivalence (used in \textsection\ref{subsect 5.2.1}) by its modular analogue recently proved by Bezrukavnikov--Riche \cite{BR24}. 
\end{rmk}

\section{$\fF$ as equivariantization of $\fF^b$}\label{sect 6} 
In this section we will use Theorem \ref{thm 5.13} to show that the equivalence $\fF^b$ intertwines rational $B$-actions on both categories, so that $\fF$ is the equivalence of their equivariantizations. 
This essentially follows the idea in the work of Arkhipov--Gaitsgory \cite{AG03}. 
As an application, we show that $\fF$ intertwines reflection functors $\Theta^r_s$ and $\Theta^c_s$, for all $s\in \I_\af$. 
For simplicity, we only consider the non-deformed versions of $\fF^b$ and $\fF$. 

\subsection{Rational group action}\label{subsect 6.1} 
In this section, we recall some notions about rational action of a linear algebraic group on a category, following \cite{Ne21}. 

Let $K$ be a linear algebraic group over $\k$. 
Let $\sD$ be a cocomplete $\k$-linear abelian category. 
For any commutative $\k$-algebra $R$, we denote by $\sD_R$ the category of $R$-linear objects in $\sD$, namely the category of pairs $(M, R\rightarrow \End_\sD(M))$ with $M\in \sD$. 
For any commutative $\k$-algebra $R'$ with a ring homomorphism $\iota: R\rightarrow R'$, we have a functor $\iota_*:=R'\otimes_R-$ from $\sD_{R}$ to $\sD_{R'}$ (it is well-defined by the cocompleteness). 
Note that $\iota\otimes_R -$ gives a natural morphism $\id\rightarrow \iota_*$ of endo-functors on $\sD_{R}$. 

Consider the coordinate ring $\k[K]$ of $K$ with comultiplication $\Delta$ and counit $\varepsilon$. 
A \textit{rational action of $K$ on $\sD$} is the data of 
\begin{itemize}
\item a functor $\psi: \sD\rightarrow \sD_{\k[K]}$; 
\item a coassociative natural isomorphism $\sigma: \Delta_*\circ \psi\xs \psi^2$; 
\item a natural isomorphism $\eta: \varepsilon_*\circ \psi\xs \id_\sD$. 
\end{itemize} 
Any $R$-point $\iota:\k[K] \rightarrow R$ of $K$ associates a functor $\iota_*\circ \psi: \sD \rightarrow \sD_R$, which defines an action of the abstract group $K(R)$ on $\sD$. 

Given a rational $K$-action on $\sD$, we define the \textit{equivariantization} $\sD^K$ as the category of $K$-equivariant objects in $\sD$, namely the category of pairs $(M, M\xrightarrow{\rho_M} \psi(M))$ with $M\in \sD$ and $\rho_M$ coassociative and counital, i.e. $\rho_M$ satisfies 
$$\psi(\rho_M)\circ \rho_M= \sigma_M \circ (\Delta\otimes_{\k[K]}-) \circ \rho_M, 
\quad \id_M=\eta_M\circ (\varepsilon\otimes_{\k[K]}-) \circ \rho_M.$$ 

Now we consider the action of the monoidal category $\Rep(K)$ on an abelian category $\sC$. 
By definition, it is an exact bifunctor 
$$-\otimes -:\ \sC \times \Rep(K) \rightarrow \sC$$ 
with compatibility of the monoidal structure on $\Rep(K)$. 
Given an action of $\Rep(K)$ on $\sC$, the \textit{de-equivariantization} $\sC_{K}$ is defined as the category of $\k[K]$-modules in $\sC$, namely the category of pairs $(M,\ M\otimes \k[K]\xrightarrow{a_M} M)$ with $M\in \sC$ and $a_M$ associative and unital. 
Here the $K$-module structure on $\k[K]$ is induced from the right multiplication of $K$. 

Suppose $\sC$ is cocomplete. 
Then there is a canonical rational $K$-action on $\sC_{K}$, via the functor 
$$\psi: \sC_{K}\rightarrow (\sC_{K})_{\underline{\k[K]}}, \quad M \mapsto \underline{\k[K]}\otimes M,$$ 
where $\underline{\k[K]}$ is the algebra $\k[K]$ with trivial $K$-action, and $\k[K]$ acts on $\underline{\k[K]}\otimes M$ via comultiplication. 
We have the following lemma 
\begin{lem}[{\cite[Prop A.2]{Ne21}}] \label{lem 6.1a} 
Let $\sC$ be a cocomplete abelian category, equipped with an action of $\Rep(K)$. 
There is an equivalence 
$$\sC\xs (\sC_{K})^K, \quad M\mapsto M\otimes \k[K],$$ 
where $M\otimes \k[K]$ is an object in $(\sC_{K})^K$ via the $\k[K]$-action and the $\underline{\k[K]}$-coaction on $\k[K]$. 
\end{lem}
 
\subsection{Coinduction functor}\label{subsect 6.2} 
We define $\U_\zeta$, $\mathbf{u}_\zeta$ as the quotients of $U_\zeta$, $u_\zeta$ by the relations $K_\lambda^l-1$, $\lambda\in \Lambda$. 
In this section, we will work with the quantum groups $\U^{\hb}_\zeta$ and $\fU^n_\zeta$ with triangular decompositions 
\begin{equation}\label{equ 6.1}
\U^{\hb}_\zeta=\fU^-_\zeta\otimes \U^0_\zeta\otimes U^+_\zeta, \quad
 \fU^n_\zeta=\fU^-_\zeta\otimes \u^0_\zeta\otimes u^+_\zeta. 
\end{equation}
Note that $\fU^n_\zeta$ is a subalgebra of $\U^{\hb}_\zeta$. 

The category $\sO_\k$ can be interpreted as the category of finitely generated $\U^{\hb}_\zeta$-modules integrable over $\U^\geq_\zeta$. 
We define $\widehat{\sO}_\k$ in the same way as $\sO_\k$ but removing the condition of being finitely generated. 
Then $\widehat{\sO}_\k$ is the ind-completion of $\sO_\k$, and the equivalence (\ref{equ 5.7}) induces an equivalence 
$$\widehat{\fF}:\ \A_\n\Mod^{B} \xs \widehat{\sO}^0_\k.$$ 

Similarly we have $\fU^b_\zeta\mod^\Lambda_\k=\fU^n_\zeta\mod^\Lambda_\k$. 
The equivalence (\ref{equ 4}) induces an equivalence 
$$\A_\n\mod \xs \fU^n_\zeta\mod_{\chi_0}$$ 
by forgetting the $\Lambda$-grading, which further induces an equivalence 
$$\widehat{\fF}^n:\ \A_\n\Mod\xs \fU^n_\zeta\Mod_{\chi_0}.$$ 
Note that $\widehat{\fF}$ and $\widehat{\fF}^n$ are compatible with the forgetful functor. 

For a $\fU^n_\zeta$-module $M$, we endow the space $\Hom_{\fU^n_\zeta}(\U^{\hb}_\zeta, M)$ with a natural $\U^{\hb}_\zeta$-module structure induced by the right multiplication of $\U^{\hb}_\zeta$. 
We denote by $\coind_r(M)$ the maximal submodule of $\Hom_{\fU^n_\zeta}(\U^{\hb}_\zeta, M)$ that is integrable over $\U^\geq_\zeta$. 
It defines a functor 
$$\coind_r:\ \fU^n_\zeta\Mod\rightarrow \widehat{\sO}_\k.$$ 
The functor $\coind_r$ is compatible with block decompositions, namely we have 
$$\coind_r:\ \fU^n_\zeta\Mod_{\chi_\omega}\rightarrow \widehat{\sO}^{\omega}_\k, \quad \forall \omega\in \Xi.$$ 
The following lemma is standard. 

\begin{lem}\label{lem 6.1} 
The functor $\coind_r$ is right adjoint to the forgetful functor $\widehat{\sO}_\k\rightarrow \fU^n_\zeta\Mod$. 
More precisely, for any $M_1\in \widehat{\sO}_\k$ and any $\fU^n_\zeta$-module $M_2$, there is a natural isomorphism 
$$\Hom_{\U^{\hb}_\zeta}(M_1,\coind_r(M_2))\simeq \Hom_{\fU^n_\zeta}(M_1, M_2),$$ 
which maps $\varphi$ on the LHS to the homomorphism $m_1\in M_1 \mapsto \varphi(m_1)(1)$, and maps $\phi$ on the RHS to the homomorphism $m_1\in M_1 \mapsto \big(u\in \U^{\hb}_\zeta \mapsto \phi(um_1)\big)$. 
\end{lem} 

We define $\sC(\U^\geq_\zeta)$ as the category of integrable $\U^\geq_\zeta$-modules. 
Similarly, we have a functor 
$$\coind^b_r:\ \u_\zeta\Mod \rightarrow \sC(\U^\geq_\zeta)$$ 
right adjoint to the forgetful functor $\sC(\U^\geq_\zeta)\rightarrow \u_\zeta\Mod$. 
By (\ref{equ 6.1}), we have isomorphisms 
$$\Hom_{\fU^n_\zeta}(\U^{\hb}_\zeta, M)=\Hom_{\u^\geq_\zeta}(\U^\geq_\zeta,M), \quad \forall M\in \fU^n_\zeta\Mod,$$ 
of $\U^\geq_\zeta$-modules. 
Hence $\coind_r(M)=\coind^b_r(M)$ as $\U^\geq_\zeta$-modules. 

There is a Hopf algebra structure on $\Hom_{\fU^n_\zeta}(\U^{\hb}_\zeta, \k)$ induced by the Hopf algebra structure on $\U^{\hb}_\zeta$, making $\coind_r(\k)$ as a Hopf subalgebra. 
Moreover $\coind_r(\k)$ is an algebra in the monoidal category $\widehat{\sO}_\k$. 
We have 
$$\coind_r(\k)=\coind^b_r(\k)\subset \Hom_{\u^\geq_\zeta}(\U^\geq_\zeta,\k)=(U\b)^*.$$
The natural pairing between $\k[B]$ and $U\b$ identifies $\coind_r(\k)=\k[B]$ as Hopf algebras. 
Therefore, any module $\coind_r(M)$ is equipped with a (right) action from $\k[B]=\coind_r(\k)$. 

In the notations in \textsection\ref{subsect 6.1}, $(\widehat{\sO}_{\k})_B$, $\sC(\U^\geq_\zeta)_{B}$ are the categories of $\k[B]$-modules in $\widehat{\sO}_\k$, $\sC(\U^\geq_\zeta)$, respectively. 
Since $\k[B]$ is trivial as a $\fU^n_\zeta$-module, for any $M\in (\widehat{\sO}_{\k})_B$, the quotient $M\otimes_{\k[B]} \k$ admits an action from $\fU^n_\zeta$. 

\begin{prop}\label{prop 6.3}
There is an equivalence 
$$\coind_r:\ \fU^n_\zeta\Mod\xs (\widehat{\sO}_{\k})_B,$$ 
with inverse functor $-\otimes_{\k[B]} \k$. 
\end{prop}
\begin{proof} 
One checks that the adjunction in Lemma \ref{lem 6.1} induces an adjunction 
$$-\otimes_{\k[B]} \k:\ (\widehat{\sO}_{\k})_B\ \rightleftarrows\ \fU^n_\zeta\Mod\ : \coind_r.$$ 
We have natural morphisms 
\begin{equation}\label{equ 6.2}
\coind_r(-)\otimes_{\k[B]}\k \rightarrow \id, \quad \id\rightarrow \coind_r(-)\otimes_{\k[B]}\k.
\end{equation} 
Recall that for any $\fU^n_\zeta$-module $M$, $\coind_r(M)=\coind^b_r(M)$ as $\U^\geq_\zeta$-modules. 
Therefore, to show that (\ref{equ 6.2}) are isomorphisms, it is enough to show that $\coind^b_r$ induces an equivalence 
$$\coind^b_r:\ \u^\geq_\zeta\Mod\xs \sC(\U^\geq_\zeta)_{B}$$ 
with inverse functor $-\otimes_{\k[B]} \k$. 
This follows from \cite[Thm 2.8]{AG03} applied to the sequence of Hopf algebras $\u^\geq_\zeta\hookrightarrow \U^\geq_\zeta \twoheadrightarrow U\b$. 
\end{proof} 

By Lemma \ref{lem 6.1a}, we have 
\begin{corollary}\label{cor 6.4}
The category $\fU^n_\zeta\Mod$ admits a natural rational $B$-action with an equivalence 
$$\widehat{\sO}_{\k}\xs (\fU^n_\zeta\Mod)^B.$$ 
\end{corollary} 
\begin{rmk}
Intuitively, the $B$-action on $\fU^n_\zeta\Mod$ is induced from the adjoint action of $\U^{\geq}_\zeta$ on $\fU^n_\zeta$. 
Indeed, $\u^\geq_\zeta$ acts via inner automorphisms in $\fU^n_\zeta$, and therefore acts trivially on the category $\fU^n_\zeta\Mod$. 
Hence the action of $\U^{\geq}_\zeta$ on $\fU^n_\zeta\Mod$ factors through the quantum Frobenius map. 
\end{rmk}

Similarly, we consider the functor 
$$\coind_c:\ \A_\n\Mod\rightarrow \A_\n\Mod^B, \quad M\mapsto M\otimes \k[B].$$ 
It is right adjoint to the forgetful functor $\A_\n\Mod^B\rightarrow \A_\n\Mod$, and induces an equivalence 
\begin{equation}\label{equ 6.3}
\coind_c:\ \A_\n\Mod\xs (\A_\n\Mod^B)_{B}.
\end{equation}

\subsection{$\widehat{\fF}$ as equivariantization of $\widehat{\fF}^n$}\label{subsect 6.3} 
\begin{lem}\label{lem 6.5}
The equivalence $\fF$ (resp. $\widehat{\fF}$) is compatible with actions from $\rep(B)$ (resp. $\Rep(B)$). 
\end{lem}
\begin{proof} 
We prove the assertion for $\fF$, then the assertion for $\widehat{\fF}$ follows. 
We have a commutative diagram 
$$\begin{tikzcd} 
\A^b_S\mod^{B} \arrow[r,"\fF","\simeq"']\arrow[d,"\V_c"'] 
& \sO^{0}_S \arrow[d,"\V_r"]\\ 
\Coh^B(\b_S\times_{\t/W} \t) \arrow[r,"\fG","\simeq"'] 
& \sO^{0}_{S\otimes_{\k[\t/W]} \k[\t]}. 
\end{tikzcd}$$ 
Note that $\V_c$, $\V_r$ and $\fG$ are compatible with $\rep(B)$-action. 
Note also that the $\rep(B)$-action preserves the subcategory of objects admitting Verma flags. 
Hence by Proposition \ref{prop 5.8}, $\fF$ is compatible with $\rep(B)$-action on such subcategory in $\A^b_S\mod^{B}$. 
Since any object can be resolved by objects admitting Verma flags, such compatibility extends to the whole category. 
\end{proof}

\begin{corollary}\label{cor 6.6} 
The equivalence $\widehat{\fF}^n: \A_\n\Mod\xs \fU^n_\zeta\Mod_{\chi_0}$ is compatible with the natural rational $B$-action on both sides. 
Therefore, $\widehat{\fF}$ can be obtained as the $B$-equivariantization of $\widehat{\fF}^n$. 
\end{corollary}
\begin{proof} 
Since $\coind_r$ and $\coind_c$ are right adjoint to the forgetful functors, we have a commutative diagram 
$$\begin{tikzcd} 
\A_\n\Mod \arrow[r,"\widehat{\fF}^n","\simeq"']\arrow[d,"\coind_c"'] 
& \fU^n_\zeta\Mod_{\chi_0} \arrow[d,"\coind_r"]\\ 
\A_\n\Mod^{B} \arrow[r,"\widehat{\fF}","\simeq"'] 
& \widehat{\sO}_\k^0. 
\end{tikzcd}$$ 
By Proposition \ref{prop 6.3} and (\ref{equ 6.3}), the coinduction functors identify the source categories with the de-equivariantizations of the target categories, and moreover the rational $B$-action on the sources come from the $\Rep(B)$-action on the targets. 
Hence the compatibility between $\widehat{\fF}^n$ and $B$-action follows from the compatibility between $\widehat{\fF}$ and $\Rep(B)$-action, which is by Lemma \ref{lem 6.5}. 

The second assertion follows from the commutative diagram 
\begin{equation}\label{equ 6.4*} 
\begin{tikzcd} 
\A_\n\Mod^{B} \arrow[r,"\widehat{\fF}","\simeq"']\arrow[d,"\simeq"'] 
& \widehat{\sO}_\k^0 \arrow[d,"\simeq"]\\ 
(\A_\n\Mod)^{B} \arrow[r,"(\widehat{\fF}^n)^B","\simeq"'] 
& (\fU^n_\zeta\Mod_{\chi_0})^B, 
\end{tikzcd}
\end{equation}
where the left vertical equivalence is tautological, the right one is by Corollary \ref{cor 6.4}. 
\end{proof}

As an application, we have the following 

\begin{corollary}\label{cor 6.8}
The equivalence 
$$\A_\n\mod^B\xs \sO_\k^0$$ 
intertwines the reflection functors $\Theta^c_s$ and $\Theta^r_s$, for any $s\in \I_\af$. 
\end{corollary} 
\begin{proof}
It is enough to show the compatibility between $\widehat{\fF}$ and reflection functors. 
The reflection functors are compatible with the actions of $B$ and $\Rep(B)$, and with $\widehat{\fF}^n$ by Proposition \ref{prop 4.1}. 
Hence they are compatible with $(\widehat{\fF}^n)^B$, and therefore with $\widehat{\fF}$ by the commutative diagram (\ref{equ 6.4*}). 
\end{proof}

\section{Graded multiplicities}\label{sect 7} 

Recall that the tilting bundle $\sE$ on $\widetilde{\g}$ can be equipped with a $\C^\times$-equivariant structure, which makes $\A$ a $\C^\times$-equivariant algebra. 
Using the equivalence (\ref{equ 5.7}), we can define a graded version of $\sO^0_\C$ via  
$$\sO_\C^{0,\gr}\ :=\ \Coh^{G\times \C^\times}(\pi_{\widetilde{\sN}}^*\A)\ \simeq\ \A_\n\mod^{B\times \C^\times},$$ 
with a degrading functor $\sO_\C^{0,\gr}\rightarrow \sO^0_\C$. 
In this section, we compute the graded multiplicity of simple modules in Verma modules in $\sO_\C^{0,\gr}$, which turns out to be given by the generic Kazhdan--Lusztig polynomials defined in \cite{Kato85} (see also \cite{Lus80h, Soe97}). 
We also obtain formulas for graded multiplicity of Verma module in projective module. 

\subsection{Periodic and generic polynomials}\label{subsect 7.1} 
We recall the periodic and generic polynomials introduced by Lusztig \cite{Lus80h} and Kato \cite{Kato85}, following Soergel's notations and normalizations \cite{Soe97}. 

Let $\ell(-):W_\ex\rightarrow \Z$ be the length function. 
The extended affine Hecke algebra $\H_\ex$ is the $\Z[v^{\pm 1}]$-algebra generated by $\{H_x\}_{x\in W_\ex}$ modulo the relations 
$$\begin{cases}
	(H_s+v)(H_s-v^{-1})=0 & \text{for } s\in \I_\af, \\ 
	H_xH_y=H_{xy} & \text{if } \ell(x)+\ell(y)=\ell(xy). 
\end{cases}$$ 
There is a ring homomorphism (called the bar-involution) $\H_\ex\rightarrow \H_\ex$, $H\mapsto \overline{H}$ given by $\overline{v}=v^{-1}$, $\overline{H_x}=(H_{x^{-1}})^{-1}$. 

Recall the semi-infinite order $\leq^\semi$ in Definition~\ref{defn 5.6}. 
Let $\P$ be the free $\Z[v^{\pm 1}]$-module with basis $\{H^{\frac{\infty}{2}}_x\}_{x\in W_\ex}$. 
It is equipped with a right $\H_\ex$-module structure (called the \textit{periodic Hecke module}) via the following formulas for any $x\in W_\ex$, $s\in \I_\af$ and $\gamma\in \Lambda/\rQ$, 
\begin{align*}
H^{\frac{\infty}{2}}_x\cdot H_s &=
\begin{cases}
H^{\frac{\infty}{2}}_{xs} & \text{if } x\leq^{\frac{\infty}{2}} xs, \\ 
H^{\frac{\infty}{2}}_{xs}+(v^{-1}-v) H^{\frac{\infty}{2}}_x & \text{if } xs\leq^{\frac{\infty}{2}} x; 
\end{cases}\\
H^{\frac{\infty}{2}}_x\cdot H_\gamma &=H^\semi_{x\gamma}.
\end{align*} 
For $\lambda\in \Lambda$ we set $e(\lambda):= \sum_{w\in W} v^{\ell(w)}\cdot H^{\frac{\infty}{2}}_{t(\lambda)w}$. 
Let $\P^\circ$ be the $\H_\ex$-submodule of $\P$ generated by $e(\lambda)$ with $\lambda\in \Lambda$. 

\begin{thm}[{\cite{Lus80h,Soe97}}] 
\begin{enumerate}
\item On $\P^\circ$ there is a unique involution $\P^\circ\rightarrow \P^\circ$, $P\mapsto \overline{P}$ that is $\H_\ex$-skew-linear (i.e. compatible with the bar-involution of $\H_\ex$) and satisfies $\overline{e(\lambda)}=e(\lambda)$ for each $\lambda\in \Lambda$. 
\item For any $x\in W_\ex$, there is a unique element $\underline{H}^{\frac{\infty}{2}}_x\in \P^\circ$ such that $\underline{H}^{\frac{\infty}{2}}_x=\overline{\underline{H}^{\frac{\infty}{2}}_x}$ and $\underline{H}^{\frac{\infty}{2}}_x\in H^{\frac{\infty}{2}}_x+\sum_{y\leq^{\frac{\infty}{2}} x} v\Z[v] \cdot H^{\frac{\infty}{2}}_y$. 
\end{enumerate}
\end{thm} 
\noindent 
We express $\underline{H}^{\frac{\infty}{2}}_x=\sum_y p_{y,x}(v)\cdot H^{\frac{\infty}{2}}_y$, then $p_{y,x}(v)\in \Z[v]$ is called the \textit{periodic Kazhdan--Lusztig polynomial}. 

Let $\hat{\P}$ be the completion of $\P$ consisting of the formal series in $\{H^{\frac{\infty}{2}}_x\}_{x\in W_\ex}$ that are bounded above (with respect to the order $\leq^{\frac{\infty}{2}}$). 
For $\eta\in \Lambda$, consider the $\Z[v^{\pm 1}]$-linear operator $\langle \eta \rangle$ on $\P$ by $\langle \eta \rangle(H^{\frac{\infty}{2}}_x)=H^{\frac{\infty}{2}}_{t(\eta)x}$, for any $x\in W_\ex$. 
We define two families of polynomials $\{q_{y,x}(v)\}_{x,y\in W_\ex}$ and $\{q'_{y,x}(v)\}_{x,y\in W_\ex}$ in $\Z[v]$ via the following equations in $\hat{\P}$ (for any $x\in W_\ex$) 
\begin{equation}\label{equ 7.1}
\big(\prod_{\alpha\in \Phi^+}(1+v^{2}\langle -\alpha \rangle+v^{4}\langle -2\alpha \rangle+\cdots)\big) (\underline{H}^{\frac{\infty}{2}}_x) = \sum_{y} q_{y,x}(v)\cdot H^{\frac{\infty}{2}}_y, 
\end{equation}
\begin{equation}\label{equ 7.2}
\big(\prod_{\alpha\in \Phi^+}(1+\langle -\alpha \rangle+\langle -2\alpha \rangle+\cdots)\big) (\underline{H}^{\frac{\infty}{2}}_x) = \sum_{y} q'_{y,x}(v)\cdot H^{\frac{\infty}{2}}_y. 
\end{equation}
Here $q_{y,x}(v)$ is called the \textit{generic Kazhdan--Lusztig polynomial}. 
Note that $q_{y,x}(1)=q_{y,x}'(1)$. 

\begin{rmk}
\begin{enumerate}
\item In \cite{Soe97} the periodic (resp. generic) polynomials are labeled by alcoves: $p_{y,x}(v)$ (resp. $q_{y,x}(v)$) is denoted by $p_{yA_\mathrm{fun},xA_\mathrm{fun}}(v)$ (resp. by $q_{yA_\mathrm{fun},xA_\mathrm{fun}}(v)$) in the \textit{loc. cit.}. 
\item The generic polynomials $\{q_{x,y}(v)\}_{x,y}$ are originally defined as the unique family of polynomials such that the element $\sum_{y} (-1)^{\ell(x)+\ell(y)}{q}_{y,x}(v^{-1})\cdot H^{\frac{\infty}{2}}_y\in \hat{\P}$ is contained in $H^{\frac{\infty}{2}}_x+\sum_{y} v^{-1}\Z[v^{-1}] \cdot H^{\frac{\infty}{2}}_y$ and is invariant under the involution on $\hat{\P}$ (obtained by extending the involution on $\P^{\circ}$), see \cite{Kato85}. 
Under this definition, the equality (\ref{equ 7.1}) is proved in \cite[Thm 3.5]{Kato85}, see also \cite[Thm 6.3]{Soe97}. 
\end{enumerate}
\end{rmk}

\subsection{Graded multiplicities}\label{subsect 7.2} 
\subsubsection{Grading on a category} 
A \textit{grading} on a category $\sC$ is the data of a triple $(\sC^\gr, \langle 1\rangle, \upsilon)$, where 
\begin{itemize}
\item $\langle 1\rangle$ is an auto-functor on the category $\sC^\gr$ (called the \textit{grading shift}), 
\item $\upsilon$ is a functor $\upsilon: \sC^\gr\rightarrow \sC$ (called the \textit{degrading functor}), 
\end{itemize}
together with the data of a natural isomorphism $\upsilon\circ \langle 1\rangle \xs \upsilon$ satisfying a natural isomorphism 
\begin{equation}\label{equ 7.0} 
\bigoplus_{d} \Hom_{\sC^\gr}(M_1,M_2\langle d\rangle) \xs \Hom_{\sC}(\upsilon M_1,\upsilon M_2),
\end{equation}
for any $M_1$, $M_2$ in $\sC^\gr$ (Here we abbreviate $\langle d\rangle=\langle 1\rangle^{\circ d}$ for any $d\in \Z$). 
For $M\in \sC$, an object $\widetilde{M}\in \sC^\gr$ with an isomorphism $\upsilon \widetilde{M}\cong M$ is called a \textit{lifting} of $M$ in $\sC^\gr$. 

For a graded algebra $A=A^\bullet$, the category of (finitely generated) graded modules $A^\bullet\gmod$ naturally provides a grading on $A\mod$. 
For a graded $A$-module $M=M^\bullet$, we set $M\langle d\rangle$ as the graded $A$-module with $(M\langle d\rangle)^i=M^{i-d}$, for any $i\in \Z$. 

\subsubsection{Graded $\A$-modules} 
Let $P(\lambda)_\C$ be the projective cover of $L(\lambda)_\C$ in $u_\zeta\mod^{\Lambda}_\C$. 
Let $Z(\lambda)_\C=u_\zeta\otimes_{u_\zeta^{\geq 0}} \C_\lambda$ be the \textit{baby Verma module}. 
Recall that $P^b(\lambda)_\C$ is the projective cover of $L(\lambda)_\C$ in $\fU^b_\zeta\mod^{\Lambda}_\C$. 
We have isomorphisms of $\fU^b_\zeta$-modules 
\begin{equation}\label{equ 7.4'}
P(\lambda)_\C\simeq P^b(\lambda)_\C \otimes_{Z_\Fr^-}\C , \quad Z(\lambda)_\C \simeq M(\lambda)_\C\otimes_{Z_\Fr^-}\C,
\end{equation}
where $\C$ is the trivial representation of $Z_\Fr^-$. 

Recall that we have a commutative diagram 
\begin{equation}\label{equ equival}
\begin{tikzcd} 
\A_\n\mod^{B} \arrow[r] \arrow[d,"\simeq"'] & \A_\n\mod^{T} \arrow[d,"\simeq"']  & \A_0\mod^{T} \arrow[l,hook'] \arrow[d,"\simeq"] \\ 
\sO^0_\C \arrow[r] 
& \fU^b_\zeta\mod^{\Lambda,0}_\C & u_\zeta\mod^{\Lambda,0}_\C \arrow[l,hook'] , 
\end{tikzcd}
\end{equation}
where the left horizontal maps are by forgetful functors, and the right inclusions are induced by the quotients $\A_\n\twoheadrightarrow \A_0$ and $\fU^b_\zeta\twoheadrightarrow u_\zeta$. 
For any $x\in W_\ex$ and bounded above poset ideal $\Omega\subset \Lambda$, we denote by 
$$\sL_x,\quad \sZ_x,\quad \sP_x,\quad \sP^b_x,\quad {\sQ}_x^{\Omega}$$ 
the preimages of $L(x\bullet_l 0)_\C$, $Z(x\bullet_l 0)_\C$, $P(x\bullet_l 0)_\C$, $P^b(x\bullet_l 0)_\C$, $Q(x\bullet_l 0)^{\Omega}_\C$ under the equivalences above. 
By Lemma \ref{lem 2.1*}, the $\A_\n$-module $\sL_x$ can be endowed with a $B$-equivariant structure, which makes $\sL_x$ a simple object in $\A_\n\mod^B$. 
Recall the Verma objects $\sM_x$ ($x\in W_\ex$) in $\A^b_S\mod^{B}$. 
We set 
$$\sM^\C_x=\sM_x\otimes_{S} \C$$ 
and call $\sM^\C_x$ (resp. $\sZ_x$) the \textit{Verma object} in $\A_\n\mod^{B}$ and $\A_\n\mod^{T}$ (resp. in $\A_0\mod^{T}$). 
A \textit{Verma flag} is a finite filtration with composition factors given by Verma objects. 

\begin{lem}\label{lem 7.3} 
Let $x\in W_\ex$, and let $\Omega\subset \Lambda$ be a bounded above poset ideal. 
\begin{enumerate}
\item The objects $\sZ_x$, $\sP_x$, $\sP^b_x$ in $\A_\n\mod^{T}$ admit liftings $\widetilde{\sZ}_x$, $\widetilde{\sP}_x$, $\widetilde{\sP}^b_x$ in $\A_\n\mod^{T\times \C^\times}$, and the liftings are unique up to grading shifts. 
\item The objects $\sL_x$, $\sM^\C_x$, $\sQ_x^{\Omega}$ in $\A_\n\mod^{B}$ admits liftings $\widetilde{\sL}_x$, $\widetilde{\sM}^\C_x$, $\widetilde{\sQ}_x^{\Omega}$ in $\A_\n\mod^{B\times \C^\times}$, and the liftings are unique up to grading shifts. 
\item Let $P$ be an object in $\A_\n\mod^{B}$ with a lifting $\widetilde{P}$ in $\A_\n\mod^{B\times \C^\times}$. 
If $P$ admits a Verma flag, then $\widetilde{P}$ also admits a Verma flag (i.e. a filtration with composition factors of the form $\widetilde{\sM}^\C_x\langle d\rangle$). 
Similar result holds for $\A_\n\mod^{T}$ and $\A_0\mod^{T}$. 
\end{enumerate}
\end{lem}
\begin{proof} 
\textit{Step 1.} We firstly show (1) and the assertion for $\sL_x$, $\sM^\C_x$ in (2). 
By a standard argument (see e.g. \cite[Proof of Prop 5.2.3]{BM13}), one can show that any $B$-equivariant simple $\A_\n$-module (resp. $T$-equivariant projective module of $\A_\n$ or $\A_0$) can be equipped with a $\C^\times$-equivariant structure, and up to isomorphism any lifting differs by a $\C^\times$-character. 

By its construction (\ref{equ 4.13}), any Verma object $\sM^\C_x$ can be naturally equipped with a $\C^\times$-equivariant structure. 
Let $\widetilde{\sM}^\C_x$ be a lifting of $\sM^\C_x$, then $\widetilde{\sZ}_x=\widetilde{\sM}^\C_x\otimes_{\C[\n]} \C$ provides a lifting of $\sZ_x$ by (\ref{equ 7.4'}). 
Suppose $\widetilde{\sM}_{1}$, $\widetilde{\sM}_{2}$ are two liftings of $\sM^\C_x$. 
Since $\End(\sM^\C_x)=\C$, by (\ref{equ 7.0}) we have $\Hom(\widetilde{\sM}_{1},\widetilde{\sM}_{2}\langle i\rangle)=\C \delta_{i,d}$ for some $d\in \Z$, and any nonzero map from $\widetilde{\sM}_{1}$ to $\widetilde{\sM}_{2}\langle d\rangle$ is an isomorphism. 
The uniqueness of $\widetilde{\sZ}_x$ follows similarly. 

\textit{Step 2.}
Now we show (3). 
We only show the assertion for $\A_\n\mod^{B}$. 
We prove by increasing induction on the length of Verma flag $\sum_{x} (P:\sM^\C_x)$. 
Let $x$ be such that $x\bullet_l 0$ is maximal among the weights $y\bullet_l 0$ with $\sM^\C_y$ appearing as a Verma factor of $P$. 
Then we have $\Hom(\widetilde{\sM}^\C_x\langle d\rangle,\widetilde{P})\neq 0$ for some $d$, and any nonzero element gives an inclusion $\widetilde{\sM}^\C_x\langle d\rangle \hookrightarrow \widetilde{P}$ (since any nonzero map $\sM^\C_x\rightarrow P$ is an inclusion). 
Now $\widetilde{P}/\widetilde{\sM}^\C_x\langle d\rangle$ is a lifting of $P/\sM^\C_x$, which by induction hypothesis admits a Verma flag. 
Hence $\widetilde{P}$ admits a Verma flag. 

\textit{Step 3.}
We show the assertion for $\sQ_x^{\Omega}$ in (2). 
Recall the truncated categories $\A^b_S\Mod^{T,\Omega}$, $\A^b_S\Mod^{B,\Omega}$, and the truncation functors $\tau^{T,\Omega}$, $\tau^{U\n,\Omega}$ introduced in \textsection\ref{subsect 5.2.2}. 
We set $\A_\n\Mod^T=\A_\n\Mod^T\cap \A^b_S\Mod^{T,\Omega}$ and define $\A_\n\Mod^{B,\Omega}$ similarly. 
The truncation functors restrict to 
$$\tau^{T,\Omega}:\ \A_\n\Mod^T\rightarrow \A_\n\Mod^{T,\Omega} \quad \text{and}\quad 
\tau^{U\n,\Omega}:\ (\A_\n\rtimes U\n)\Mod^T\rightarrow \A_\n\Mod^{B,\Omega}.$$
We set $\A_\n\Mod^{T\times \C^\times, \Omega}$ as the preimage of $\A_\n\Mod^{T,\Omega}$ under the degrading functor $\A_\n\Mod^{T\times \C^\times}\rightarrow \A_\n\Mod^{T}$, and define $\A_\n\Mod^{B\times \C^\times, \Omega}$ similarly. 
Consider the truncation functor 
$$\tau^{T\times \C^\times,\Omega}:\ \A_\n\Mod^{T\times \C^\times}\rightarrow \A_\n\Mod^{T\times \C^\times, \Omega},$$
by sending $M$ to its quotient by the minimal graded submodule containing $\ker(M\rightarrow \tau^{T,\Omega}(M))$. 
The functor $\tau^{T\times \C^\times,\Omega}$ is left adjoint to the natural inclusion. 
We define 
$$\tau^{U\n\times \C^\times,\Omega}:\ (\A_\n\rtimes U\n)\Mod^{T\times \C^\times}\rightarrow \A_\n\Mod^{B\times \C^\times, \Omega},$$
by sending $M$ to its quotient by the (graded) $\A_\n\rtimes U\n$-module generated by $\ker(M\rightarrow \tau^{T\times \C^\times,\Omega}(M))$. 
Then $\tau^{U\n\times \C^\times,\Omega}$ is left adjoint to the natural inclusion. 

We claim that the following natural transformations are isomorphisms 
\begin{equation}\label{equ 7.5new}
\tau^{T,\Omega}\circ \upsilon \xs \upsilon\circ \tau^{T\times \C^\times,\Omega}  
\quad \text{and}\quad
\tau^{U\n,\Omega}\circ \upsilon \xs \upsilon\circ \tau^{U\n\times \C^\times,\Omega} 
\end{equation}
where $\upsilon$ is the degrading functor. 
Note that the second isomorphism follows from the first one, hence we only need to show the first isomorphism. 
By the right exactness of truncation functors and degrading functors, it is enough to construct isomorphism $\upsilon\circ \tau^{T\times \C^\times,\Omega}(P^\gr)\simeq \tau^{T,\Omega}(P)$ for any $T$-equivariant projective $\A_\n$-module $P$ and its lifting $P^\gr$ (which exists by (1)). 
By (3), $P^\gr$ admits a Verma filtration. 
By \eqref{equ new2.4} and (\ref{equ 7.0}) applied to the degrading $\Db(\A_\n\mod^{T\times \C^\times})\rightarrow \Db(\A_\n\mod^{T})$, we have $\Ext^1(\widetilde{\sM}^\C_x,\widetilde{\sM}^\C_y\langle d\rangle)=0$ if $x\bullet_l 0\nless y\bullet_l 0$. 
Hence $\tau^{T\times \C^\times,\Omega}(P^\gr)$ is the quotient of $P^\gr$ by the submodule composed by Verma factors $\widetilde{\sM}^\C_y\langle d\rangle$ with $y\bullet_l 0\notin \Omega$, which shows that $\tau^{T\times \C^\times,\Omega}(P^\gr)=\tau^{T,\Omega}(P)$ after forgetting the grading structure. 
It proves the claim. 

Finally, consider the induction functor 
$$\ind_c:=(\A_\n\rtimes U\n)\otimes_{\A_\n}-:\ \A_\n\mod^{T\times \C^\times}\rightarrow (\A_\n\rtimes U\n) \mod^{T\times \C^\times}.$$ 
Let $x\in W_\ex$ such that $x\bullet_l 0\in \Omega$. 
The module $\widetilde{\sQ}_x^\Omega= \tau^{U\n\times \C^\times,\Omega}(\ind_c \widetilde{\sP}^b_x)$ is the projective cover of $\widetilde{\sL}_x\langle d\rangle$ (for some $d\in \Z$) in $\A_\n\mod^{B\times \C^\times,\Omega}$, which provide a lifting of $\sQ^\Omega_x=\tau^{U\n,\Omega}(\ind_c {\sP}^b_x)$. 
The uniqueness of lifting follows from the uniqueness of projective cover. 
\end{proof}

For an object $P$ in $\A_\n\mod^{B\times \C^\times}$ admitting Verma flags, we set $(P: \widetilde{\sM}^\C_x\langle d\rangle)$ as the multiplicity of the Verma factor $\widetilde{\sM}^\C_x\langle d\rangle$, which does not depend on the choice of Verma flag. 
We define the \textit{graded multiplicity} as 
\begin{equation}\label{equ 7.3}
(P: \widetilde{\sM}^\C_x)_{\gr}=\sum_d (P: \widetilde{\sM}^\C_x\langle d\rangle) \cdot v^d 
\ \in \Z[v^{\pm 1}].
\end{equation}
For any object $M$ in $\A_\n\mod^{B\times \C^\times}$, it admits a separated decreasing $\Z_{\leq 0}$-filtration (may of infinite length) with simple composition factors. 
We set $[M: \widetilde{\sL}_x\langle d\rangle]$ as the multiplicity (which is finite) of $\widetilde{\sL}_x\langle d\rangle$ in the factors, which dose not depend on the choice of such filtration, and we define the \textit{graded multiplicity} 
\begin{equation}\label{equ 7.4}
[M: \widetilde{\sL}_x]_{\gr}=\sum_d [M: \widetilde{\sL}_x\langle d\rangle] \cdot v^d \ \in \Z[v^{\pm 1}].
\end{equation}
We also introduce similar notations for $\A_\n\mod^{T\times \C^\times}$ and $\A_0\mod^{T\times \C^\times}$. 

Let $\hat{K}_0(\A_\n\mod^{B\times \C^\times})$ be the space of $\Z[v^{\pm 1}]$-linear formal series in $\{[\widetilde{\sL}_x]\}_{x}$ that are bounded above (with respect to the order $\leq^{\frac{\infty}{2}}$). 
We have an inclusion 
$$K_0(\A_\n\mod^{B\times \C^\times})\rightarrow \hat{K}_0(\A_\n\mod^{B\times \C^\times}), \quad [M]\mapsto \sum_x [M: \widetilde{\sL}_x]_{\gr}\cdot [\widetilde{\sL}_x].$$ 
Note that (the images of) $\{[\widetilde{\sM}_x]\}_x$ and $\{[\widetilde{\sZ}_x]\}_x$ give other topological basis on $\hat{K}_0(\A_\n\mod^{B\times \C^\times})$. 

\subsubsection{Graded multiplicities} 
Here is the main result of this section. 
\begin{thm}\label{thm 7.3}
Let $\Omega$ be a bounded above poset ideal of $\Lambda$. 
Under the degrading functor $\sO_\C^{0,\gr} \rightarrow \sO_\C^0$, the modules $L(x\bullet_l 0)_\C$, $M(x\bullet_l 0)_\C$ and $Q(x\bullet_l 0)_\C^{\Omega}$ ($x\in W_\ex$) admit liftings $\widetilde{L}(x\bullet_l 0)_\C$, $\widetilde{M}(x\bullet_l 0)_\C$ and $\widetilde{Q}(x\bullet_l 0)_\C^{\Omega}$ in $\sO_\C^{0,\gr}$, such that the graded multiplicities are given by 
\begin{equation}\label{equ 7.6'}
(\widetilde{Q}(x\bullet_l 0)_\C^{\Omega}: \widetilde{M}(y\bullet_l 0)_\C)_\gr=
\begin{cases}
	q'_{w_0y,w_0x}(v) & \text{if } y\bullet_l 0 \in \Omega, \\ 
	0 & \text{if else}, 
\end{cases}
\end{equation}
and 
\begin{equation}\label{equ 7.7'}
[\widetilde{M}(x\bullet_l 0)_\C: \widetilde{L}(y\bullet_l 0)_\C]_\gr= q_{w_0x,w_0y}(v),
\end{equation}
for any $x,y\in W_\ex$. 
\end{thm}

Before proving Theorem \ref{thm 7.3}, we firstly consider the graded multiplicities in $\A_0\mod^{T\times \C^\times}$. 
Set $\rB$ as the $\Lambda$-graded (alternatively, $T$-equivariant) algebra with an isomorphism 
$$\rB^\op=\bigoplus_{\lambda\in \Lambda} \Hom\big(\bigoplus_{w\in W} P(w\bullet 0)_\C,\bigoplus_{w\in W} P(w\bullet 0-l\lambda)_\C\big).$$ 
Note that $\rB$ is finite dimensional, and there is an equivalence 
$$\bigoplus_{\lambda}\Hom(\bigoplus_{w} P(w\bullet 0+l\lambda)_\C,-):\ u_\zeta\mod_\C^{\Lambda,0}\xs \rB\mod^T.$$ 
We denote by $L_x$, $Z_x$, $P_x$ the image of $L(x\bullet_l 0)_\C$, $Z(x\bullet_l 0)_\C$, $P(x\bullet_l 0)_\C$ ($x\in W_\ex$) under this equivalence. 
The following theorem is a recollection of results by Andersen--Jantzen--Soergel in \cite[\textsection18]{AJS94} (with conventions adopted from \cite{Soe97}). 

\begin{thm}[{\cite{AJS94}}]\label{thm 7.4} 
The algebra $\rB$ admits a $T$-stable Koszul grading $\rB^\bullet$. 
Under the degrading functor $\rB^\bullet\gmod^T\rightarrow \rB\mod^T$,  
the modules $L_x$, $Z_x$, $P_x$ ($x\in W_\ex$) admit liftings $\widetilde{L}_x$, $\widetilde{Z}_x$, $\widetilde{P}_x$ in $\rB^\bullet\gmod^T$. 
Moreover, these liftings can be chosen such that $\widetilde{L}_{t(\lambda)x}=\widetilde{L}_x\otimes \C_\lambda$ for any $\lambda\in \Lambda$ (and similarly for $\widetilde{Z}_x$, $\widetilde{P}_x$), and 
$$(\widetilde{P}_y:\widetilde{Z}_x)_{\gr}= [\widetilde{Z}_x: \widetilde{L}_y]_{\gr}=p_{w_0x,w_0y}(v), \quad \forall x,y\in W_\ex.$$ 
(Here the graded multiplicities are defined as in (\ref{equ 7.3}) and (\ref{equ 7.4}).) 
\end{thm}

We use the theorem above to show the following 

\begin{prop}\label{prop 7.5} 
The liftings $\widetilde{\sL}_x$, $\widetilde{\sZ}_x$, $\widetilde{\sP}_x$ ($x\in W_\ex$) in $\A_0\mod^{T\times \C^\times}$ can be chosen such that $\widetilde{\sL}_{t(\lambda)x}=\widetilde{\sL}_x\otimes \C_\lambda$ for any $\lambda\in \Lambda$ (and similarly for $\widetilde{\sZ}_x$, $\widetilde{\sP}_x$), and 
$$(\widetilde{\sP}_y:\widetilde{\sZ}_x)_{\gr}= [\widetilde{\sZ}_x: \widetilde{\sL}_y]_{\gr}=p_{w_0x,w_0y}(v), \quad \forall x,y\in W_\ex.$$ 
\end{prop}
\begin{proof} 
Through the equivalences $\rB\mod^T\simeq u_\zeta\mod_\C^{\Lambda,0}\simeq \A_0\mod^{T}$, the grading $\rB^\bullet$ on $\rB$ induces a different grading, hence a different graded multiplicity, on $\A_0\mod^{T}$. 
Our main task is to compare this graded multiplicity to the natural one induced by $\A_0\mod^{T\times \C^\times}$. 

By \cite[\textsection5.2.3, \textsection5.5]{BM13} (see also \cite[Thm 4.5]{BL23}), there is a $T\times \C^\times$-equivariant vector bundle $\sE^{new}$ on $\widetilde{\g}$ satisfying Theorem \ref{thm BM}(2), such that the graded algebra $\End^{\op}_{\widetilde{\g}}(\sE^{new})$ is Koszul (where grading is induced from $\C^\times$-equivariant structure). 
We can replace $\sE$ by $\sE^{new}$, then $\A^\bullet$ is a Koszul algebra. 
In particular, the graded algebra $\A_0^\bullet=\A^\bullet\otimes_{\C[\g]} \C$ satisfies: (1) $\A_0^\bullet$ is non-negatively graded; (2) $\A^0_0$ is a semi-simple algebra; (3) $\A_0^\bullet$ is generated by $\A^1_0$ over $\A^0_0$. 

We can choose a lifting $\widetilde{\sP}_x$ of $\sP_x$ ($x\in W_\ex$) such that it is generated by degree 0 elements as a graded $\A_0^\bullet$-module. 
Define a graded algebra $\rB'^\bullet$ via 
$$(\rB'^\bullet)^\op = \bigoplus_{d\in \Z, \lambda\in \Lambda} \Hom\big(\bigoplus_{w\in W} \widetilde{\sP}_w,\bigoplus_{w\in W} \widetilde{\sP}_w\otimes \C_{-\lambda} \langle d\rangle )\big).$$ 
Then we have an isomorphism of algebras $\rB\simeq \rB'$. 
From the properties of $\A_0^\bullet$ (resp. Koszulity of $\rB^\bullet$), we deduce that (1) $\rB'^\bullet$ (resp. $\rB^\bullet$) is non-negatively graded; (2) $\rB'^0$ (resp. $\rB^0$) is a semi-simple algebra; (3) $\rB'^\bullet$ (resp. $\rB^\bullet$) is generated by $\rB'^1$ (resp. $\rB^1$) over $\rB'^0$ (resp. $\rB^0$). 
Hence by \cite[Prop 2.4.1]{BGS96}, we have $\rad(\rB')^i=\bigoplus_{j\geq i} \rB'^{j}$ and $\rad(\rB)^i=\bigoplus_{j\geq i} \rB^{j}$, and therefore we have a ($T$-equivariant) isomorphism of graded algebras 
\begin{equation}\label{equ 7.5}
\rB^{\bullet} \xs \bigoplus_{i} \rad(\rB)^i/\rad(\rB)^{i+1} \simeq \bigoplus_{i} \rad(\rB')^i/\rad(\rB')^{i+1} \ \xleftarrow{\sim} \rB'^\bullet.
\end{equation}
It induces an equivalence $\rB^{\bullet}\gmod^T \simeq \rB'^{\bullet}\gmod^T$. 
Note that the isomorphism (\ref{equ 7.5}) and the earlier isomorphism $\rB\simeq \rB'$ might be different, but they induce the same isomorphism on the semi-simple quotients $\rB/\rad(\rB)\simeq \rB'/\rad(\rB')$. 
Therefore, under the composed equivalence $\rB^{\bullet}\gmod^T \simeq \rB'^{\bullet}\gmod^T \simeq \A_0\mod^{T\times \C^\times}$, the simple module $\widetilde{L}_x$ is send to $\widetilde{\sL}_x$ ($x\in W_\ex$) (up to shifts). 
As projective cover of $\widetilde{L}_x$, $\widetilde{P}_x$ is send to $\widetilde{\sP}_x$ (up to shifts). 
Since $\widetilde{Z}_x$ is the projective cover of $\widetilde{L}_x$ in the Serre subcategory generated by $\widetilde{L}_y\langle d\rangle$ with $y\leq^{\frac{\infty}{2}} x$, it follows that $\widetilde{Z}_x$ is send to $\widetilde{\sZ}_x$ (up to shifts). 

Now, we can deduce our assertion from Theorem \ref{thm 7.4}. 
\end{proof}

\begin{rmk}
Proposition \ref{prop 7.5} might be proved as the analogue statement for restricted enveloping algebra in \cite[Thm 1.2]{BL23}, instead of using Theorem \ref{thm 7.4}. 
\end{rmk}

\begin{proof}[Proof of Theorem \ref{thm 7.3}] 
We fix the liftings $\widetilde{\sL}_x$, $\widetilde{\sZ}_x$, $\widetilde{\sP}_x$ ($x\in W_\ex$) in Proposition \ref{prop 7.5}. 
And choose liftings $\widetilde{\sM}^\C_x$, $\widetilde{\sP}^b_x$ such that $\widetilde{\sM}^\C_x\otimes_{\C[\n]} \C=\widetilde{\sZ}_x$ and $\widetilde{\sP}^b_x\otimes_{\C[\n]} \C= \widetilde{\sP}_x$. 
Then we have $\widetilde{\sM}^\C_{t(\lambda)x}=\widetilde{\sM}^\C_x\otimes \C_\lambda$ for any $\lambda\in \Lambda$ (and similarly for $\widetilde{\sP}^b_x$). 

\textit{Step 1.}
We firstly consider the multiplicity $[\widetilde{\sM}^\C_x:\widetilde{\sL}_y]_\gr$. 
Consider the following sequence induced by Koszul resolution 
\begin{equation}\label{equ 7.6}
\widetilde{\sM}^\C_x\otimes_{\C[\n]} (\C[\n]\otimes \bigwedge^\bullet \n^*) \rightarrow \widetilde{\sZ}_x \rightarrow 0,
\end{equation}
which is exact, since the corresponding sequence in $\fU^b_\zeta\mod^{\Lambda}_\C$ 
$$M(x\bullet_l 0)_\C \otimes_{Z_\Fr^-}(Z_\Fr^-\otimes \bigwedge^\bullet \n^*)\rightarrow Z(x\bullet_l 0)_\C\rightarrow 0$$ 
is exact by the freeness of the $Z_\Fr^-$-module $M(x\bullet_l 0)_\C$. 
By abuse of notation, we view $\langle \eta \rangle$ ($\eta\in \Lambda$) as a continuous $\Z[v^{\pm 1}]$-linear operator on $\hat{K}_0(\A_\n\mod^{B\times \C^\times})$ sending $[\widetilde{\sM}^\C_x]$ to $[\widetilde{\sM}^\C_{t(\eta)x}]$, for any $x\in W_\ex$. 
By (\ref{equ 7.6}), we have the following equality in $K_0(\A_\n\mod^{B\times \C^\times})$, 
\begin{equation}\label{equ 7.11}
[\widetilde{\sZ}_x]=\big(\prod_{\alpha\in \Phi^+}(1- v^2\langle -\alpha\rangle)\big) [\widetilde{\sM}^\C_x]. 
\end{equation}
In $\hat{K}_0(\A_\n\mod^{B\times \C^\times})$, we express $[\widetilde{\sL}_y]$ into (topological) basis $\{[\widetilde{\sZ}_x]\}_x$ and $\{[\widetilde{\sM}^\C_x]\}_x$ as 
$$[\widetilde{\sL}_y]=\sum_x [\widetilde{\sL}_y: \widetilde{\sZ}_x]_\gr\cdot [\widetilde{\sZ}_x]=\sum_x [\widetilde{\sL}_y: \widetilde{\sM}^\C_x]_\gr\cdot [\widetilde{\sM}^\C_x].$$ 
Then (\ref{equ 7.11}) shows that 
\begin{equation}\label{equ 7.12}
\big(\prod_{\alpha\in \Phi^+}(1- v^2\langle -\alpha\rangle)\big)(\sum_x [\widetilde{\sL}_y: \widetilde{\sZ}_x]_\gr\cdot [\widetilde{\sM}^\C_x])=\sum_x [\widetilde{\sL}_y: \widetilde{\sM}^\C_x]_\gr\cdot [\widetilde{\sM}^\C_x]. 
\end{equation}
By \cite[\textsection 11]{Lus80h} (see also \cite[Thm 6.1]{Soe90}), there is an equality 
\begin{equation}\label{equ 7.13} 
\sum_x (-1)^{\ell(x)+\ell(y)} q_{x,y}(v)\cdot p_{w_0x,w_0z}(v) =\delta_{y,z}, \quad \forall y,z\in W_\ex.
\end{equation}
Hence it follows from Proposition \ref{prop 7.5} that 
\begin{equation}\label{equ 7.14}
[\widetilde{\sL}_y: \widetilde{\sZ}_x]_\gr=(-1)^{\ell(x)+\ell(y)} q_{x,y}(v), \quad \forall x,y\in W_\ex. 
\end{equation}
Combining (\ref{equ 7.1}), (\ref{equ 7.12}) with (\ref{equ 7.14}), we deduce that 
$$[\widetilde{\sL}_y: \widetilde{\sM}^\C_x]_\gr=(-1)^{\ell(x)+\ell(y)} p_{x,y}(v), \quad \forall x,y\in W_\ex.$$ 
By (\ref{equ 7.13}) again, we have 
\begin{equation}\label{equ 7.18}
[\widetilde{\sM}^\C_x:\widetilde{\sL}_y]_\gr=q_{w_0x,w_0y}(v), \quad \forall x,y\in W_\ex. 
\end{equation}

\textit{Step 2.} 
Now we study the multiplicity $(\widetilde{\sQ}^\Omega_x: \widetilde{\sM}^\C_y)_\gr$. 
Apply $-\otimes_{\C[\n]}\C$ to any Verma flag of $\widetilde{\sP}^b_x$, we obtain a Verma flag of $\widetilde{\sP}_x$, since the corresponding statement in $\fU^b_\zeta\mod^{\Lambda}_\C$ is true by the freeness of Verma modules over $Z_\Fr^-$. 
Hence we have 
\begin{equation}\label{equ 7.7} 
(\widetilde{\sP}^b_x:\widetilde{\sM}^\C_y)_\gr=(\widetilde{\sP}_x:\widetilde{\sZ}_y)_\gr. 
\end{equation}

Consider the $(T\times \C^\times)$-equivariant $\A_\n \rtimes U\n$-module 
$$\widetilde{\sQ}_x:= \ind_c(\widetilde{\sP}^b_x)= (\A_\n \rtimes U\n) \otimes_{\A_\n} \widetilde{\sP}^b_x,$$ 
and set $\sQ_x$ the module by forgetting the $\C^\times$-action of $\widetilde{\sQ}_x$. 
We have $\widetilde{\sQ}_x^\Omega= \tau^{U\n\times \C^\times,\Omega}(\widetilde{\sQ}_x)$, where $\tau^{U\n\times \C^\times,\Omega}$ is the truncation functor in the proof of Lemma \ref{lem 7.3}. 
There is a finite filtration 
\begin{equation}\label{equ 7.17}
0=\sQ_0\subset \sQ_1 \subset \sQ_2\subset \cdots \subset \sQ_n =\widetilde{\sQ}_x,
\end{equation}
with composition factors $\sQ_{i+1}/\sQ_i$ of the form $\ind_c(\widetilde{\sM}^\C_y\langle d\rangle)$, where $\widetilde{\sM}^\C_y\langle d\rangle$ are Verma factors of $\widetilde{\sP}^b_x$. 
Since $\widetilde{\sM}^\C_y$ is a $\A_\n \rtimes U\n$-module, there is an isomorphism of $\A_\n \rtimes U\n$-modules $\ind_c(\widetilde{\sM}^\C_y)\simeq \widetilde{\sM}^\C_y\otimes U\n$, where $U\n$ acts on $\widetilde{\sM}^\C_y\otimes U\n$ diagonally. 
There is a separated decreasing filtration 
$$\widetilde{\sM}^\C_y\otimes U\n = \sM_0 \supset \sM_1 \supset \cdots ,$$
with composition factors $\sM_i/\sM_{i+1}$ of the form $\widetilde{\sM}^\C_y\otimes \C_\eta$, $\eta\geq 0$. 
We formally write 
$$(\widetilde{\sQ}_x:\widetilde{\sM}^\C_y)_\gr:=\sum_{i=0}^{n-1} (\sQ_{i+1}/\sQ_i:\widetilde{\sM}^\C_y)_\gr.$$ 
Then there is an equality in $\hat{\P}$, 
$$\big(\prod_{\alpha \in \Phi^+} (1+\langle -\alpha \rangle +\langle -2\alpha \rangle +\cdots)\big)(\sum_{y} (\widetilde{\sP}^b_x:\widetilde{\sM}^\C_y)_\gr\cdot H^{\frac{\infty}{2}}_{w_0y}) = \sum_y (\widetilde{\sQ}_x:\widetilde{\sM}^\C_y)_\gr\cdot H^{\frac{\infty}{2}}_{w_0y}. $$ 
By (\ref{equ 7.2}) and (\ref{equ 7.7}) and Proposition \ref{prop 7.5}, we have 
$$(\widetilde{\sQ}_x:\widetilde{\sM}^\C_y)_\gr=q'_{w_0y,w_0x}(v), \quad \forall x,y\in W_\ex.$$ 

To complete the proof, we show that as a truncation of $\widetilde{\sQ}_x$, the module $\widetilde{\sQ}^{\Omega}_x$ admits graded multiplicity obtained as truncation from the one for $\widetilde{\sQ}_x$, as (\ref{equ 7.6'}). 
As $(T\times \C^\times)$-equivariant $\A_\n$-modules, we have $\widetilde{\sM}^\C_y\otimes U\n \simeq \bigoplus_\eta \widetilde{\sM}^\C_y\otimes (U\n)_\eta$. 
Since $\Ext^1(\widetilde{\sM}^\C_x,\widetilde{\sM}^\C_y\langle d\rangle)=0$ if $x\bullet_l 0\nless y\bullet_l 0$, from the filtration (\ref{equ 7.17}) one can obtain a separated decreasing filtration of $\widetilde{\sQ}_x$ in $\A_\n\Mod^{T\times \C^\times}$, with composition factors by Verma objects. 
Hence $\tau^{T\times \C^\times,\Omega}(\widetilde{\sQ}_x)$ (see again the proof of Lemma \ref{lem 7.3} for the truncation functor $\tau^{T\times \C^\times,\Omega}$) is the quotient of $\widetilde{\sQ}_x$ by the submodule composed by $\widetilde{\sM}^\C_y\langle d\rangle$ with $y\bullet_l 0\notin \Omega$, and therefore its graded multiplicity of Verma objects is as the RHS of (\ref{equ 7.6'}). 
Since $\widetilde{\sQ}_x^\Omega= \tau^{U\n\times \C^\times,\Omega}(\widetilde{\sQ}_x)$ is a quotient of $\tau^{T\times \C^\times,\Omega}(\widetilde{\sQ}_x)$, we have 
\begin{equation}\label{equ 7.19}
(\sQ_x^\Omega:{\sM}^\C_y) \leq \big(\tau^{T,\Omega}({\sQ}_x) : {\sM}^\C_y\big) =
\begin{cases}
	q'_{w_0y,w_0x}(1) & \text{if } y\bullet_l 0 \in \Omega, \\ 
	0 & \text{if else}. 
\end{cases}
\end{equation}
On the other hand, by the BGG reciprocity \cite[Prop 3.4]{Situ1} (in \textit{loc. cit.} it is stated for $\Omega=\{\lambda\in \Lambda| \lambda\leq \nu\}$ for some $\nu$, but the proof works for general $\Omega$) and (\ref{equ 7.18}) we have 
$$(\sQ_x^\Omega:{\sM}^\C_y)=
\begin{cases}
	q_{w_0y,w_0x}(1) & \text{if } y\bullet_l 0 \in \Omega, \\ 
	0 & \text{if else}. 
\end{cases}$$
Hence (\ref{equ 7.19}) is an equality. 
It follows that $\widetilde{\sQ}_x^\Omega=\tau^{U\n\times \C^\times,\Omega}(\widetilde{\sQ}_x)=\tau^{T\times \C^\times,\Omega}(\widetilde{\sQ}_x)$, whose graded multiplicity $(\widetilde{\sQ}_x^\Omega:\widetilde{\sM}^\C_y)_\gr=(\tau^{T\times \C^\times,\Omega}(\widetilde{\sQ}_x):\widetilde{\sM}^\C_y)_\gr$ is equal to the RHS of (\ref{equ 7.6'}). 
\end{proof}

\ 

\nocite{*} 

\bibliographystyle{plain} 
\bibliography{MyBibtex3}

\end{document}